\documentclass{article}

\usepackage{arxiv}

\RequirePackage{amsthm,amsmath,amsfonts,amssymb}
\RequirePackage[authoryear]{natbib} 
\RequirePackage{graphicx}
\usepackage{bbm,bm}
\usepackage{enumitem}
\usepackage[latin1]{inputenc}
\usepackage[T1]{fontenc}
\usepackage{hyperref}
\usepackage{dsfont}
\usepackage{multirow}
\usepackage{color}
\theoremstyle{plain}
\newtheorem{theorem}{Theorem}
\newtheorem{proposition}[theorem]{Proposition}
\newtheorem{corollary}[theorem]{Corollary}
\newtheorem{lemma}[theorem]{Lemma}

\theoremstyle{remark}

\newtheorem{example}[theorem]{Example}
\newtheorem{remark}[theorem]{Remark}

\newcommand{\NN}{\mathbb{N}}
\newcommand{\RR}{\mathbb{R}}

\newcommand{\EE}{\mathbb{E}}
\newcommand{\PP}{\mathbb{P}}

\newcommand{\Var}{\mathrm{Var}}
\newcommand{\Cov}{\mathrm{Cov}}

\DeclareMathOperator{\argmin}{argmin}

\newcommand{\inv}{\ensuremath{\leftarrow}}

	\title{Estimation of the Spectral Measure from Convex Combinations of 
			Regularly Varying Random Vectors}

		\author{Marco Oesting\\
		Stuttgart Center for Simulation Science and Institute of Stochastics and Applications,\\
		University of Stuttgart, 70569 Stuttgart, Germany,\\
				\texttt{marco.oesting@mathematik.uni-stuttgart.de}\\
		\And
Olivier Wintenberger\\
LPSM,
Sorbonne Universit\'e,\\ 4 place Jussieu, 75005 Paris, France\\
Wolfgang Pauli Institut, c/o Fakult\"at f\"ur Mathematik,\\
	Universit\"at Wien, 1090 Vienna, Austria\\
	\texttt{ olivier.wintenberger@sorbonne-universite.fr}\\ }
\begin{document}

\maketitle

\begin{abstract}
The extremal dependence structure of a regularly varying random vector $\bm X$
is fully described by its limiting spectral measure. In this paper, we investigate how to recover characteristics of the measure, such as extremal coefficients, from the extremal behaviour of convex combinations of components of $\bm X$. Our considerations result in a class of new estimators of moments of the corresponding combinations for the spectral vector. We show asymptotic normality by means of a functional limit theorem and, focusing on the estimation of extremal coefficients, we verify that the minimal asymptotic variance can be achieved by a plug-in estimator using subsampling bootstrap. We illustrate the benefits of our approach on simulated and real data.
\end{abstract}


\begin{keywords}
{Asymptotic statistics; extremal coefficient; moment estimator; multivariate extremes.}
\end{keywords}

\section{Introduction}

One of the main challenges in multivariate extreme value analysis is the inference of the dependence among the extremes of a random vector $\bm X = (X_1,\ldots,X_d)^\top$ from a sample of independent copies $\bm X_1,\ldots,\bm X_n$. In risk management applications, particular interest lies in the case that $\bm X$ possesses asymptotically dependent components, i.e.\ extremes tend to appear jointly in various components, with heavy tails. Many of these cases are covered by the framework of multivariate regular variation. We fix an orthant where the extremes occur, and restrict ourselves to the first one  with no loss of generality; thus, we consider $\bm X = (X_1,\ldots,X_d)^\top \in [0,\infty)^d$.

It is common practice in extreme value analysis to consider the extremal dependence structure of $\bm X$ separately from the marginal properties of the components $X_1,\ldots, X_d$. Thus, focus is often put on the joint tail behaviour of some standardized vector $\bm X^*$ obtained from $\bm X$ via marginal transformations. More precisely, the margins are standardized such that
the tails are regularly varying with index $1$ and asymptotically identical in the sense that 
$$ \lim_{x \to \infty} \frac{\PP( X_i^* > x)}{\PP(X_j^* > x)} = 1 \quad \text{for all } i,j \in \{1,\ldots,d\}. $$
By such a transformation, the property of multivariate regular variation is inherited from $\bm X$ to $\bm X^*$. Thus, extremal dependence of $\bm X^*$ can be fully described via its so-called spectral measure, that is, the limit distribution of the angular component $\bm X^* / \|\bm X^*\|_\infty$ conditional on $\| \bm X^*\|_\infty > u$ as $u \to \infty$. We denote the vector distributed as the spectral measure by $\Theta$.

In the past decades, the estimation of the spectral measure and related concepts has attracted a lot of attention, starting from  the bivariate setting ($d=2$), where estimation of the spectral measure is equivalent to inference for the so-called Pickands' dependence function. Here, seminal estimation procedures come from  \cite{pickands1975statistical} with refinements such as \cite{caperaa-fougeres-genest-1997} using a copula approach assuming known margins. These procedures have been extended  by \cite{genest-segers-2009} to the setting with unknown margins of $\bm X$ using the rank transform. All the above mentioned works were assuming that the observations are max-stable, i.e.~that the asymptotic extreme regime is achieved, a typical
assumption for observations that correspond to maxima over certain temporal blocks such as annual maxima. The pioneer works of \cite{einmahl-dehaan-huang-1993} and \cite{einmahl-dehaan-sinha-1997} deal with the estimation of the cumulative distribution function of the spectral measure with observations in the domain of attraction of max-stable distribution, including the case of regularly varying distributions. The two-step procedure in \cite{einmahl-dehaan-sinha-1997} for dealing with unknown margins is replaced by a more efficient rank transform in \cite{einmahl-dehaan-piterbarg-2001}.  An alternative approach due to \cite{drees1998best} focuses on the estimation of the stable tail dependence function
$$
L(\bm x) = \EE\Big[\max_{i=1,\ldots,d} x_i\Theta_i\Big]\,, \qquad \bm x = (x_1,\ldots,x_d) \in (0,\infty)^d\,,
$$
using the rank transform for $d=2$. This work was extended to the multivariate case $d>2$ in \cite{einmahl-krajina-segers-2012}. We refer to Chapter 7 of the monograph \cite{dehaan-ferreira-2006} for a general overview of non-parametric approaches.

When $d$ is large, any inference of the complete spectral distribution suffers from the curse of dimensionality. Thus, it is more reasonable in practice to focus on the joint tail behaviour of the components of  lower-dimensional subvectors $\bm X^*_I = (X^*_i)_{i \in I}$ for sufficiently many ``directions'' $I$ that are non-empty subsets of $\{1,\ldots,d\}$. In this framework, one often makes use of summary statistics such as the so-called extremal coefficients \citep[cf.][]{smith90, schlather-tawn-2003}
given by 
\begin{equation} \label{eq:def-tau-I}
\tau_I = \lim_{u\to \infty}\dfrac{\PP(\max_{i \in I} X^*_i > u)}{\PP(X^*_j > u)}, \qquad j\in I\,,
\end{equation}
a number between $1$ and $|I|$. Such parameter has attracted a lot of attention as it quantifies the effective number of asymptotically independent components among $I$
\citep[cf.][]{schlather-tawn-2003}.  Some properties, such as self-consistency for sets of extremal coefficients, and various procedures for estimating  extremal coefficients were introduced by \cite{schlather-tawn-2003}. For max-stable random fields, \cite{cooley-naveau-2006} showed that inference can also be based on the madogram, a graphical tool originally developed in geostatistics.\\
In this paper, we also consider the extremal dependence structure of subvectors $\bm X_I^*$ described in terms of the corresponding  spectral vector $\bm \Theta^I$, that is, the limit in distribution of $\bm X^*_I / \|\bm X_I^*\|_\infty$ conditional on  $\ell_I(\bm X^*) := \|\bm X^*_I\|_\infty > u$ as $u \to \infty$. While existing works typically focus on a direct estimation of the distribution of $\bm \Theta^I$, i.e.\ the spectral measure, we focus on moment estimators of $\bm \Theta^I$. Our approach further differs from previous works since we impose the assumption that the original observations $\bm X$ is such that $\bm X^* = (r_1^{-1} X_1^{\alpha_1}, \ldots,  r_d^{-1} X_d^{\alpha_d})$ is standard regularly varying with index $1$ and with tail equivalent margins. This restriction allows us to consider a two-step procedure in the spirit of \cite{einmahl-dehaan-sinha-1997} with a simpler first standardization step relying on the componentwise normalization via upper order statistics and the estimation of the coefficients $\alpha_1,\ldots,\alpha_d > 0$ via Hill type estimators. Our second step is then based on the estimation of moments of convex combinations of the spectral vector, i.e.\ estimators of $\EE[(\bm v^\top \bm \Theta^I)^p]$ with $p \geq 1$ and $\bm v$ on the simplex of the direction $I$. We prove that the collection of such moments $\EE[(\bm v^\top \bm \Theta^I)^p]$, $p \geq 1$, characterizes the spectral measure of $\bm X^*_I$. Perceiving the resulting estimators as functions of the vector $\bm v$ on the simplex, we obtain functional central limit theorems with precise asymptotic (co-)variances. Here, we consider both the case that the margins of $\bm X$ are known, i.e.\ we can directly work with exactly standardized observations $\bm X^*$, and the case that the margins are unknown 
and a preliminary empirical standardization step has to be included.\\
In accordance to the above-mentioned paradigm that, in large dimensions $d$, emphasis is often put on simpler summary statistics, we suggest to mainly focus on small $p$ for which our approach provides simple functions that contain important information on the dependence among extremes. For instance, by tail equivalence of the marginals of $\bm X^*$ we have that $\EE[\Theta_i^I]=\EE[\Theta_j^I]=1/\tau_I$ for any $i,j\in I$ and, consequently, for any positive weights $\bm v$ such that $\sum_{i\in I}v_1=1$ we have that $\EE[\bm v^T \bm \Theta^I] = 1/\tau_I$. Thus, the moment of order $p=1$ of convex combinations of $\bm \Theta^I$ does not depend on the weight vector $\bm v$. This observation motivates an additional third step of our procedure. More precisely, similarly to 
\citet{mainik-rueschendorf-2010} who estimate the optimal portfolio of jointly regularly varying assets, we use the functional central limit theorem to estimate the weight vector $\bm v^*$ that minimizes the asymptotic variance and show that this minimal asymptotic variance can be achieved by a plug-in estimator. We note that the determination of $\bm v^*$ is interesting on its own and provides important information on the dependence among extremes.  The determination of a prominent direction in the data is related to recent developments on extreme PCA in \cite{cooley2019decompositions} and \cite{drees2019principal}. Moreover we discuss how the  weights $\bm v^*$ are impacted by the fact that the margins are known or not.

The paper is organized as follows. In Section \ref{sec:lin-comb}, we gather several results on convex combinations of the spectral components of a regularly varying vector. The asymptotic normality of empirical moment estimators is given in Section \ref{sec:asnorm} both for the case of known and the case of unknown margins. These results specified to first order moments are then applied in Section \ref{sec:extr-coeff} in order to obtain estimators for extremal coefficients with minimal asymptotic variance. Some comparison to benchmark estimators based on the definition of $\tau_I$ and some illustrations on simulated and real data are provided in Section \ref{sec:illustr}. We conclude with a short discussion on our results in Section \ref{sec:conclusion}.

\section{Extremes of convex combinations of regularly varying random vectors} 
\label{sec:lin-comb}

Let $\bm X = (X_1,\ldots,X_d)^\top$ be a non-standard regularly varying 
$E=[0,\infty)^d$-valued random vector, with different tail indices $\alpha_1,\ldots,\alpha_d > 0$. We assume the existence of scaling factors $r_1,\ldots,r_d > 0$ such   the transformed vector $\bm X^* = (r_1^{-1} X_1^{\alpha_1}, \ldots,  r_d^{-1} X_d^{\alpha_d})$ is  regularly varying with index $1$ and satisfies
\begin{equation} \label{eq:vagueconv-norm}
a^*(u) \PP( u^{-1} \bm X^* \in \cdot) \stackrel{v}{\longrightarrow} \mu^*(\cdot)\,, \qquad u \to \infty\,,
\end{equation}
for some normalized exponent measure $\mu^*$ which is homogeneous of order $-1$
and
\begin{equation} \label{eq:a-normal}
 \mu^*(\{\bm x \in E: \ x_i > 1\}) 
    = \lim_{u \to \infty} a^*(u) \PP(X_i^* > u) = 1, 
\end{equation}    
i.e.\ $a^*(u) \sim \PP(X_i^*>u)^{-1}$ for all $i=1,\ldots,d$. This verifies that the tail behaviour of 
all the components of $\bm X^*$ is asymptotically identical, i.e.\ its marginal distribution functions
$F_1^*, \ldots, F_d^*$ satisfy
$$ \frac{1-F_i^*(u)}{1-F_j^*(u)} \to 1, \qquad u \to \infty,$$
for all $i,j \in \{1,\ldots,d\}$, which also implies
\begin{equation} \label{eq:quantile-equivalence}
 \frac{(F_i^*)^{-1}(p)}{(F_j^*)^{-1}(p)} \to 1, \qquad p \to 1.
 \end{equation}

Eq.~\eqref{eq:vagueconv-norm} can be equivalently expressed in terms of a spectral component for an
arbitrary norm $\|\cdot\|$ on $\RR^d$. More precisely, we obtain the weak convergence 
\begin{equation} \label{eq:rv}
\mathcal{L}(u^{-1} \bm X^* \mid \|\bm X^*\| > u) \stackrel{w}{\longrightarrow} \mathcal{L}(Y \bm \Theta)\,, \qquad u \to \infty\,,
\end{equation}
where $Y$ is a unit Pareto random variable and, independently from $Y$,
$\bm \Theta$ is a $[0,\infty)^d$-valued random vector satisfying 
$\|\bm \Theta\|=1$ almost surely, i.e.\
$$ \PP(\bm \Theta \in S_{d-1}^+) = 1, \quad \text{where } 
S_{d-1}^+ = \{ \bm x \in [0,\infty)^d: \ \|\bm x\| = 1 \}.$$
In the following, we will focus on the maximum norm $\|\cdot\|=\|\cdot\|_\infty$
which is not a serious restriction as all norms on $\RR^d$ are equivalent. 
\medskip

There is a one-to-one correspondence between the normalized exponent measure
$\mu^*$ and the distribution of $\bm \Theta$, the so-called (normalized)
\emph{spectral measure} (w.r.t.\ the norm $\|\cdot\|$) via
\begin{equation} \label{eq:exp-spec}
\mu^*(\{\bm x \in E: \ \|\bm x\| > r, \ \bm x / \|\bm x\| \in B\}) = \tau r^{-1} \PP(\bm \Theta \in B), \qquad r>0, \ B \subset S^+_{d-1}
\end{equation}
for some constant $1\le \tau \le d$ given by
\begin{equation} \label{eq:norm-const}
\tau = \mu^*(\{\bm x \in E: \, \|\bm x\| > 1\}) 
= \lim_{u \to \infty} a^*(u) \PP(\|\bm X^*\| > u) \,.
\end{equation}
The extremal dependence structure of the components of the normalized
random vector $\bm X^*$, and, thus, also of the original vector $\bm X$, is
fully described by the distribution of the spectral vector $\bm \Theta$
which we will estimate in the following.
\medskip

More precisely, we now describe the tail behaviour of convex combinations of $X^*_1, \ldots, X^*_d$ in terms of the corresponding spectral measure. To this end, we will make use of the following well-known lemma which is a direct consequence of the weak convergence in Eq.~\eqref{eq:rv}.

\begin{lemma} \label{lem:functions}
 Let $\bm X^*$, $Y$ and $\bm \Theta$ be as above. Furthermore, let 
 $g: (0,\infty) \times S_{d-1}^+ \to \RR$ be a continuous bounded function.
 Then, we have
 $$ \lim_{u \to \infty} 
 \EE \left[ g(u^{-1}\|\bm X^*\|, \bm X^* / \|\bm X^*\|) \mid \|\bm X^*\| > u \right] 
 = \EE[ g(Y, \bm \Theta) ]. $$
\end{lemma}	

In particular, the lemma implies that
\begin{equation} \label{eq:lim-theta}
 \lim_{u \to \infty} \EE[ \tilde g( \bm X^* / \|\bm X^*\|) \mid \|\bm X^*\| > u] 
 = \EE[\tilde g(\bm \Theta)]
\end{equation}
for every continuous function $\tilde g: S^+_{d-1} \to \RR$.
\medskip

Furthermore, the lemma extends to indicator functions of sets that are continuous with respect to the measure of $(Y,\bm\Theta)$, i.e.\ conditional
probabilities of events that can be expressed in terms of $u^{-1} \|\bm X^*\|$ 
and $\bm X^* / \|\bm X^*\|$. For instance, for convex combinations of $X^*_1,
\ldots, X^*_d$ with coefficient vector $\bm v \in B_1(\bm 0) = 
\{\bm x \in \RR^d: \ \|\bm x\|_1 \leq 1\}$, $\bm v \neq \bm 0$, we obtain
\begin{align} \label{eq:lin-limits}
 \lim_{u \to \infty} \PP( \bm v^\top \bm X^* > u \mid \|\bm X^*\| > u) 
 ={} \PP( Y \bm v^\top \bm \Theta > 1 )
 ={}& \EE\left[ (\bm v^\top \bm \Theta)_+ \wedge 1\right]
 = \EE\left[ (\bm v^\top \bm \Theta)_+\right]
\end{align}
using that $\bm v^\top \bm \Theta \leq \|\bm v\|_1 \cdot \|\bm \Theta\| = \|\bm v\|_1 \leq 1$ a.s. 
It can be shown that the law of $\bm \Theta$ is uniquely determined by the set 
of all the limits in \eqref{eq:lin-limits} with 
$\bm v \in \partial B_1(\bm 0) = \{\bm x \in \RR^d: \ \|\bm x\|_1 = 1\}$,
see \citet{basrak-davis-mikosch-2002,klueppelberg-pergamenchtchikov-2007,
	boman-lindskog-2009}. This fact motivates the use of extremes of convex combinations of observations 
for inference of the spectral measure. In this paper, we modify the above approach by assessing the limiting quantity 
on the RHS of Eq.~\eqref{eq:lin-limits} in an alternative way based on the 
following result which is a trivial consequence of Lemma \ref{lem:functions}.

\begin{lemma} \label{lem:functions2}
Let $\bm X^*$ be a regularly varying random vector with standard Pareto margins.
Then, for every $\bm v \in  \RR^d$ and $p \in \NN$, we have
 \begin{equation} \label{eq:all-limits}
  \EE\left[ \left(\bm v^\top \frac{\bm X^*}{\|\bm X^*\|}\right)_{+}^p  \, \bigg| \, \|\bm X^*\|>u \right] \longrightarrow \EE[ (\bm v^\top \bm \Theta)_{+}^p] \qquad (u \to \infty)\,.
 \end{equation} 
\end{lemma}
\begin{proof}
The result directly follows by applying Eq.~\eqref{eq:lim-theta} to the
continuous functions $\tilde g_{\bm v,p}(\theta)=(\bm v^\top \theta)_+^p$ with
$\bm v \in \RR^d$ and $p\in \NN$.
\end{proof}

Note that, for $\bm v \not \geq \bm 0$, the term $\bm v^\top \bm \Theta$ is negative with positive probability, which hampers the calculation of closed-form 
expressions for the RHS of \eqref{eq:all-limits}.
In order to avoid these difficulties, we restrict our attention to
non-negative vectors $\bm v \in  \partial B_1^+(\bm 0) 
= \partial B_1(\bm 0) \cap [0,\infty)^d$, i.e.\ we consider only convex
combinations of components of $\bm X^*$ with positive coefficients, while 
allowing for arbitrary powers of these combinations. This ensures the
identifiability of the spectral measure from the resulting limits as the 
following proposition shows.

\begin{proposition} \label{prop:limits}
 Let $\bm X^*$ be a non-negative regularly varying random vector with 
 index 1. Then, the spectral measure of $\bm \Theta$ is uniquely determined by the moments
 \begin{equation} \label{eq:set-all-limits}
  \left\{ \EE[ (\bm v^\top \bm \Theta)_{+}^p]: \ \bm v \in \partial B^+_1(\bm 0), \ p \in \NN \right\} = \left\{ \EE[ (\bm v^\top \bm \Theta)^p]: \ \bm v \in \partial B^+_1(\bm 0), \ p \in \NN \right\}.
 \end{equation}
\end{proposition}
\begin{proof}
 By a series expansion, for all $\bm s \in [0,\infty)^d$, the limits in 
 \eqref{eq:set-all-limits} uniquely determine
 $$ \EE [ \exp(-\bm s^\top \bm \Theta) ] 
  = \sum_{p=0}^\infty (-1)^p \cdot \frac{\|\bm s\|_1^p}{p!} \cdot \EE\left[ \left(\frac{\bm s^\top}{\|\bm s\|_1} \bm \Theta\right)^p \right], \qquad \bm s \in [0,\infty),$$ 
 i.e.\ the Laplace transform of the spectral measure. 
\end{proof}

Let us list some properties of the function 
$\bm v \mapsto m_p(\bm v)=\EE[(\bm v^\top \bm \Theta)^p]$ for $p=1,2,\ldots$
and $\bm v \in \partial B^+_1(\bm 0)$. We identify
$m_1(\bm v)=\sum\nolimits_{i=1}^d v_i \EE[\Theta_i]$, 
i.e.\ the function is linear for $p=1$. 
As, for all $i=1,\ldots,d$,
\begin{align} \label{eq:mean-spectral}
\EE[\Theta_i] ={}& \int_0^1 \PP(\Theta_i > x) \, \mathrm{d} x
{}={} \int_1^\infty y^{-2} \PP(\Theta_i > 1/y) \, \, \mathrm{d} y \nonumber \\
={}& \PP(Y \Theta_i > 1) {}={} \lim_{u \to \infty} \PP( X_i^* > u \mid \|\bm X^*\| > u) {}={} \lim_{u \to \infty}  \frac{1}{a^*(u) \PP(\|\bm X^*\|>u)} = \frac 1 \tau,
\end{align} 
where $\tau$ is defined as in Eq.~\eqref{eq:norm-const}, we even obtain
$ m_1(\bm v) = \frac 1 {\tau} \sum_{i=1}^d v_i$, 
which implies that $m_1$ is constant on $\partial B_1^+(\bm 0)$.
For $p \geq 2$, the function $m_p$ is strongly convex. Thus, one easily deduces from Jensen's inequality that $ m_p(\bm v)\le \max_{1\leq i\leq d}\EE[\bm \Theta_i^{p}]$.

For many popular limit models, closed-form expressions for 
$\EE[ (\bm v^\top \Theta)^p]$ are difficult to obtain as the spectral measure
w.r.t.\ the maximum norm $\|\cdot\| = \|\cdot\|_\infty$ is often difficult to handle in high dimensional setting.  Instead, spectral vectors
are often normalized w.r.t.\ some less complex functional 
$\ell_I(\bm x) = \bigvee_{i \in I} x_i$ over some smaller subset $I \subset \{1,\ldots,d\}$. Such functionals $\ell_I$ play an important role when estimating extremal coefficients, see Section \ref{sec:extr-coeff}. 
Since $\partial \{\bm x \in E:\ \ell_I(\bm x) > 1\}
\subset \{\bm x \in E: \ \ell_I(\bm x) = 1\}$ is a $\mu^*$-null set, the
limit  
	\begin{equation} \label{eq:norm-const-subset}
	\tau_I :=  \mu^*(\{\bm x \in E: \ \ell_I(\bm x) > 1\}) 
	= \lim_{u \to \infty} a^*(u) \PP(\ell_I(\bm X^*) > u)
	\end{equation}
exists and, by Eq.~\eqref{eq:vagueconv-norm}, equals 
$\mu^*(\{\bm x \in E: \ \ell_I(\bm x) > 1\})$. Furthermore, for every $I \neq \emptyset$, it satisfies the 
relation 
\begin{equation} \label{eq:const-relation}
 \tau_I = \EE[\ell_I(\bm \Theta)] \cdot \tau \in [1,|I|].
\end{equation}
Using the property
$\EE[\ell_I (\bm \Theta)] >0$ for every $I\neq \emptyset$, we can study the behaviour of the 
$\ell_I$-spectral vector $\bm \Theta^I$ whose distribution can be
defined from the original spectral measure via the relation
\begin{equation} \label{eq:spectral-trafo}
\PP\left( \bm \Theta^I \in A \right)
= \frac{1}{\EE[\ell_I(\bm \Theta)]} \int_{[0,\infty)^d} \mathds 1\{ \bm \theta / \ell_I(\bm \theta) \in A\} \ell_I(\bm \theta) \, \PP(\bm \Theta \in \mathrm{d} \bm \theta),
\end{equation} 
where $ A \subset \{ \bm w \in [0,\infty)^d: \ \ell_I (\bm w) = 1 \}$.
Note that, from Eq.~\eqref{eq:const-relation} and Eq.~\eqref{eq:spectral-trafo}, we directly obtain an analogous result to
Eq.~\eqref{eq:mean-spectral}, namely
\begin{equation} \label{eq:mean-spectral-ell}
\EE[ \Theta_i^I] = \frac{1}{\tau_I}, \quad i=1,\ldots,d.
\end{equation}
Analogues to Eq.~\eqref{eq:lin-limits} and to Lemma
\ref{lem:functions2} are given in the following proposition.

\begin{proposition} \label{prop:limits-ell}
 Let $\bm X^*$ be a normalized regularly varying random vector with index $1$.
 Then, for every $\bm v \in \RR^d$ and 
 $p \in \NN$, we have
 \begin{align}
 \PP\left(\bm v^\top \bm X^* > u \mid \ell_I(\bm X^*) > u\right) 
 \longrightarrow{}&  \EE\left[ \left( \bm v^\top \bm \Theta^I\right)_+ \wedge 1  \right]=\frac{\EE\left[(\bm v^\top \bm \Theta)_+ \wedge \ell_I(\bm \Theta)\right]}{\EE[\ell_I(\bm \Theta)]}     \label{eq:p-limits-ell} \\
  \EE\left[  \left(\bm v^\top \frac{\bm X^*}{\ell_I(\bm X^*)}\right)_+^p  \, \bigg| \, \ell_I(\bm X^*) > u \right] 
  \longrightarrow{}& \EE\left[ \left( \bm v^\top \bm \Theta^I  \right)_+^p \right]=\frac {\EE\left[ \left( \bm v^\top \frac{\bm \Theta}{\ell_I(\bm \Theta)} \right)_+^p \cdot \ell_I(\bm \Theta) \right]}{\EE\left[ \ell_I(\bm \Theta) \right]}  \label{eq:e-limits-ell}
 \end{align}
 as $u \to \infty$.  
\end{proposition}
\begin{proof}
 First, we notice that  $\ell_I(\bm x) \le \|\bm x\|$ on $[0,\infty)^d \setminus \{\bm 0\}$ by definition of $\ell_I$.
Thus $\ell_I(\bm X^*) > u$ implies that $\|\bm X^*\| > u$
and we have
\begin{align*}
  & \lim_{u \to \infty} \PP\left(\bm v^\top \bm X^* > u \mid \ell_I(\bm X^*) > u\right) 
  {}={} \lim_{u \to \infty} \frac{\PP\left( \bm v^\top \bm X^* > u, \, \ell_I(\bm X^*) > u \mid \|\bm X^*\| > u  \right)}
  {\PP\left( \ell_I(\bm X^*) > u \mid \|\bm X^*\| > u  \right)} \\
  ={}& \frac{\PP\left( Y ( \bm v^\top \bm \Theta) > 1, \,  Y \ell_I (\bm \Theta) > 1 \right)}{\PP\left(Y \ell_I(\bm \Theta) > 1 \right)} 
  {}={} \frac{ \EE\left[  (\bm v^\top \bm \Theta)_+  \wedge \ell_I(\bm \Theta) \wedge 1 \right]}{\EE\left[ \ell_I (\bm \Theta)  \wedge 1\right]} \\
  ={}& \frac{1}{\EE[\ell_I(\bm \Theta)]} \cdot \EE\left[(\bm v^\top \bm \Theta)_+ \wedge \ell_I(\bm \Theta)\right]
  {}={} \EE\left[ \left( \bm v^\top \bm \Theta^I \right)_+ \wedge 1  \right],
\end{align*}
where we used the fact that $\ell_I(\bm \Theta) \leq \|\bm \Theta\| = 1$ and we apply the change of measure \eqref{eq:spectral-trafo}. 
Using the same arguments and applying Lemma \ref{lem:functions} to the bounded measurable function
$g(y,\theta)=(\bm v^\top \theta)_+^p\mathds 1\{\ell_I(y\theta)>1\}$ whose discontinuity set $\{\ell_I(y\theta)= 1\}$ has no limiting mass, i.e. $\PP(\ell_I(Y\bm \Theta)= 1)=\PP(Y\ell_I(\bm \Theta)=1)=0$,
we further obtain
\begin{align*}
    & \lim_{u \to \infty} \EE\left[ \left(\bm v^\top \frac{\bm X^*}{\ell_I(\bm X^*)}\right)_+^p  \, \bigg| \, \ell_I(\bm X^*) > u \right] \\
 {}={}& \lim_{u \to \infty} \frac{\EE\left[ \left(\bm v^\top \frac{\bm X^*}{\ell_I(\bm X^*)}\right)_+^p   \, \mathds 1\{\ell_I(\bm X^*) > u] \, \bigg| \, \|\bm X^*\| > u \right\} }
                               {\PP\left( \ell_I(\bm X^*) > u \mid \|\bm X^*\| > u \right)} \\
 {}={}& \frac{\EE\left[ \left( \bm v^\top \frac{\bm \Theta}{\ell_I(\bm \Theta)} \right)_+^p \mathds 1\{Y \ell_I(\bm \Theta) > 1\}  \right]}{\PP\left(Y \ell_I(\bm \Theta) > 1  \right)} 
 {}={} \frac{ \EE\left[ \left( \bm v^\top \frac{\bm \Theta}{\ell_I(\bm \Theta)} \right)_+^p \cdot ( \ell_I(\bm \Theta)\wedge 1 ) \right]}{\EE\left[\ell_I(\bm \Theta) \wedge 1\right]} \displaybreak[0]\\
 {}={}& \frac 1 {\EE\left[ \ell_I(\bm \Theta) \right]} \cdot \EE\left[ \left( \bm v^\top \frac{\bm \Theta}{\ell_I(\bm \Theta)} \right)_+^p \cdot \ell_I(\bm \Theta) \right]
 {}={} \EE\left[ \left( \bm v^\top \bm \Theta^I  \right)_+^p \right].
\end{align*}
\end{proof}  
Similarly to Prop.~\ref{prop:limits}, it can be shown that the distribution of $\bm \Theta^I$ is uniquely determined by
  \begin{equation*}
  \left\{ \EE[ (\bm v^\top \bm \Theta^I)_{+}^p]; \ \bm v \in \partial B^+_1(\bm 0), \ p \in \NN \right\} = \left\{ \EE[ (\bm v^\top \bm \Theta^I)^p]; \ \bm v \in \partial B^+_1(\bm 0), \ p \in \NN \right\}.
 \end{equation*}

While the $\ell_I$-spectral measure, i.e.\ the distribution of $\bm \Theta^I$,
can be written in terms of the original spectral measure as in
\eqref{eq:spectral-trafo}, the original spectral measure can be uniquely 
recovered from the distribution of $\bm \Theta^I$ only if 
$\PP(\ell_I(\bm \Theta) = 0) = 0$. Then, the inverse transformation is 
\begin{equation} \label{eq:spectral-back-trafo}
 \PP\left( \bm \Theta \in A \right) = \frac{1}{\EE[\|\bm \Theta^I\|]}
 \int_{[0,\infty)^d} \mathds 1\{ \bm \theta / \|\bm \theta\| \in A\} \|\bm \theta\| \, \PP(\bm \Theta^I \in \mathrm{d} \bm \theta),
  \qquad A \subset S_{d-1}^+.
\end{equation}  
Note that the discussion above can be easily extended to any positively
 homogeneous continuous functional $\ell: [0,\infty)^d \to [0,\infty)$ with corresponding constant $\tau_\ell$ which is positive.

\section{Asymptotic normality}\label{sec:asnorm}

\subsection{Asymptotic normality for deterministic thresholds}
\label{subsec:fixed-thresh}

Assume that we observe independent copies $\bm X^*_l = (X^*_{l1}, \ldots X^*_{ld})^\top$, $l = 1,\ldots,n$, of the normalized regularly varying random
vector $\bm X^*$. Consider  a non-empty subset $I\subset \{1,\ldots,d\}$. One of the classical approaches of \cite{pickands1975statistical} to estimate the distribution of the $\ell_I$-spectral vector $\bm \Theta^I$  is based on the conditional exceedance probabilities of convex combinations given in relation \eqref{eq:p-limits-ell}, i.e.\
\begin{equation*} 
\lim_{u \to \infty} \PP( \bm v^\top \bm X^* > u \mid \ell_I(\bm X^*) > u) 
 = \EE\left[ (\bm v^\top \bm \Theta^I)_+ \wedge 1\right] 
\end{equation*}
for
$\bm v \in \partial B_1(\bm 0) = \{\bm x \in \RR^d: \, \|\bm x\|_1 = 1\}$.   The associated empirical estimators are of the form
\begin{align} \label{eq:def-benchmark-estimator}
  \frac{\widehat{P}^{\mathrm{(conv)}}_{n,u,I} (\bm v )}{\widehat{P}_{n,u,I} } 
  ={}& \frac{ n^{-1} \sum_{l=1}^n \mathds 1\{ \bm v^\top \bm X^*_{l} > u, \ell_I(\bm X^*_l) > u\}}
  { n^{-1} \sum_{l=1}^n \mathds 1\{ \ell_I(\bm X^*_l) > u\}},
  \qquad \bm v \in \partial B_1(\bm 0)\,.
  \end{align}
In view of Prop.~\ref{prop:limits}, we promote in this paper the use of the alternative characterization of the distribution of the $\ell_I$-spectral vector $\bm \Theta^I$ based on the limits of moments of convex combinations provided by Eq.~\eqref{eq:e-limits-ell}, i.e.\
\begin{equation*}
  \lim_{u \to \infty} \EE\left[  \left(\bm v^\top \frac{\bm X^*}{\ell_I(\bm X^*)}\right)^p  \, \bigg| \, \ell_I(\bm X^*) > u \right] = \EE\left[ ( \bm v^\top \bm \Theta^I )^p \right]\,,
\end{equation*} 
for every $\bm v  \in \partial B_1^+(\bm 0)$ and $p \in \NN_0$. The associated empirical estimators are of the form
\begin{equation} \label{eq:def-ratio-estimator}
\frac{\widehat{M}_{n,u,I}(\bm v,p)}{\widehat{P}_{n,u,I}} 
={}  \frac{ n^{-1} \sum_{l=1}^n  
	\left(\bm v^\top \frac{\bm X^*_{l}}{\ell_I(\bm X^*_{l})}\right)^p
	\mathds 1\{ \ell_I(\bm X^*_{l})> u\}}
{ n^{-1} \sum_{l=1}^n \mathds 1\{\ell_I(\bm X^*_{l}) > u\}},
\qquad \bm v \in \partial B_1^+(\bm 0).
\end{equation}

In the following, we will consider the new estimators $\widehat{M}_{n,u,\ell}(\bm v,p)/\widehat{P}_{n,u,\ell}$ for any $p \in \NN_0$.  As the functional $\ell_I$ puts emphasis on the components $(X^*_i)_{i \in I}$ of $\bm X^*$ only, we will focus on the same components in the convex combination, i.e.\ we will consider weight vectors $\bm v_I$ where $\bm v_I$ is the vector with components $v_i$ for $i \in I$ and $0$ otherwise. Thus, we obtain $\bm v_I^\top \bm X^* = \sum_{i \in I} v_i X^*_i$.  

The next theorem will provide a functional central
limit theorem for the estimators
$ \widehat M_{n,u,I}(\bm v_I,p)$ 
noticing that $\widehat P_{n,u,I}  = \widehat M_{n,u, I}(\bm v_I, 0)$ for all 
$\bm v \in \partial B_1^{+}(\bm 0)$. Then we can apply the functional delta method to obtain the asymptotic result for the ratio estimator
$ \widehat M_{n,u,I}(\bm v_I, p)/ \widehat P_{n,u,I}.$  

In this subsection, we assume observations from the normalized random vector $\bm X^*$, in Subsection \ref{subsec:random-thresh} the more general case of observations from the regularly varying vector $\bm X$ with unknown index $\alpha>0$ and unknown marginal scales is considered. To allow for this extension, in the next theorem, we do not only prove a functional central limit theorem for $\widehat M_{n,u,I}(\bm v_I, p)$ with $\bm v \in \partial B_1^{+}(\bm 0)$ and $p \in \NN_0$, but already for a more general version.

Indeed, it will turn out to be useful to allow for multiplication of single components of $\bm X^*$ by a vector $\bm s \in [0,\infty)^d$ of factors, and for raising them up to a power of $1/\bm \beta$ with $\bm \beta \in (0,\infty)^d$. In the following, let $(\bm s \circ \bm X^*)^{1/\bm\beta} = (s_i^{1/\beta_i} {X_i^*}^{1/\beta_i})_{1\le i\le d}$.
With this notation, we can write
$$ \ell_I(\bm X^*) = \max_{i \in I} |X^*_i| = \| (\bm 1_I \circ \bm X^*)^{\bm 1}\|, $$
where $\bm 1 = (1,\ldots,1) \in \RR^d$, i.e., according to the notation introduced above, $\bm 1_I$ is the vector with the $i$th component being equal to $1$ if $i \in I$ and being equal to $0$ otherwise. Furthermore,
$\bm v_I^\top \bm X^* = \bm v^\top (\bm 1_I \circ \bm X^*)$. Thus, we
can rewrite
\begin{align} 
\widehat M_{n,u,I}(\bm v_I,p) ={}& \frac 1 n \sum_{l=1}^n  
\left(\bm v^\top \frac{\bm 1_I \circ \bm X^*_{l}}{\|\bm 1_I \circ \bm X^*_{l}\|}\right)^p
\mathds 1\{ \|\bm 1_I \circ \bm X^*_{l}\| > u\} \nonumber \\
={}& \frac 1 n \sum_{l=1}^n  
\left(\bm v^\top \frac{\bm 1_I \circ \bm X^*_{l}/u}{\|\bm 1_I \circ \bm X^*_{l}/u\|}\right)^p
\mathds 1\{ \|\bm 1_I \circ \bm X^*_{l}/u\| > 1\}, \label{eq:rewrite-m-hat}
\end{align}
where the right-hand side of the equation depends on the components $v_i$, $i \in I$, of $\bm v$ only.
In our extended version of the functional central limit theorem, we will not restrict our attention to the vector $\bm 1_I \circ \bm X^*$ being extreme,
but a perturbed vector $(\bm s \circ \bm X^*)^{1/\bm \beta}$ with $\bm s$ being close to $\bm 1_I$ and $\bm \beta$ being close to $\bm 1$. More precisely, we consider $\bm s \in A_{\delta,I}$
and $\bm\beta \in [(1+\delta)^{-1}, 1+\delta]^d$ where
\begin{equation*}
A_{\delta,I} = \left\{ \bm x \in (0,\infty)^d: \ 
x_i = 0  \text{ for all } i \notin I, \ (1+\delta)^{-1} \leq x_i \leq 1 +\delta \text{ for all } i \in I \right\}
\end{equation*}
for some $\delta \geq 0$, i.e.\ we allow for a perturbation of a factor and a power between $(1+\delta)^{-1}$ and $1+\delta$ in each component of $\bm s$ and in each component of $\bm\beta$. We also denote the disjoint unions
$$A_\delta = \bigcup_{\emptyset \neq I \subset \{1,\ldots,d\}} A_{\delta,I}\qquad\text{and}\qquad  A_\delta' = \bigcup_{\emptyset \neq I \subset \{1,\ldots,d\}} A_{\delta,I} \times [(1+\delta )^{-1}, 1+\delta]^d. $$	

Replacing $\bm 1_I \circ \bm X^*/u$ in Eq.~\eqref{eq:rewrite-m-hat} by $(\bm s \circ \bm X^*/u)^{1/\bm\beta}$, we obtain the generalized estimator
\begin{align*}\widehat{M}_{n,u}(\bm v,\bm s, \beta,p)
&={}  \frac 1 n \sum_{l=1}^n  
	\left(\bm v^\top \frac{(\bm s \circ \bm X^*_{l}/u)^{1/\bm\beta}}{\|(\bm s \circ \bm X^*_{l}/u)^{1/\bm\beta}\|}\right)^p
	\mathds 1\{ \| (\bm s \circ \bm X_{l}^*/u)^{1/\bm\beta} \| >1\} \\
	&={}  \frac 1 n \sum_{l=1}^n  
	\left(\bm v^\top \frac{(\bm s \circ \bm X^*_{l}/u)^{1/\bm\beta}}{\|(\bm s \circ \bm X^*_{l}/u)^{1/\bm\beta}\|}\right)^p
	\mathds 1\{ \| \bm s \circ \bm X_{l}^* \| >u\}
\end{align*}
for $\bm v \in \partial B_1^+(\bm 0)$, $(\bm s,\bm\beta) \in A_\delta'$ 
and $p \in \NN_0$. By definition, $\widehat{M}_{n,u,I}(\bm v_I, p) = \widehat{M}_{n,u}(\bm v_I,\bm 1_I,\bm 1, p)$ for every $I \subset \{1,\ldots,d\}$.
In particular, the corresponding index set $I$ of components being relevant for the maximum 
can be read off from the components of $\bm s$ that are strictly positive.
Analogously,  with the convention $0^0=1$ we define $\widehat{P}_{n,u}(\bm s) = \widehat{M}_{n,u}(\bm v,\bm s,\bm \beta, 0)$ which depends  neither on $\bm v \in \partial B_1^+(\bm 0)$ nor on 
$\bm \beta \in [(1+\delta )^{-1}, 1+\delta]^d$, and satisfies $\widehat{P}_{n,u,I} = \widehat{P}_{n,u}(\bm 1_I)$.

Analogously to Prop.~\ref{prop:limits-ell}, we obtain
\begin{align}
\lefteqn{a^*(u) \EE\left[ \widehat{M}_{n,u}(\bm v, \bm s,\bm \beta,p) \right] }\nonumber \\
   &={} a^*(u) \PP(\|\bm s \circ \bm X^*\| > u) 
   \cdot \EE \left[ \left(\bm v^\top \frac{(\bm s \circ \bm X^* /u)^{1/\bm\beta}}{\|(\bm s \circ \bm X^* /u)^{1/\bm\beta}\|}\right)^p \, \bigg| \, \|\bm s \circ \bm X^*\| > u \right] \nonumber \\
  & \underset{u\to \infty}{\longrightarrow}{}  \tau \cdot \EE\left[ \left(\bm v^\top  \frac{(Y \bm s \circ \bm \Theta)^{1/\bm\beta}}{\|(Y \bm s \circ \bm \Theta)^{1/\bm\beta}\|} \right)^p \mathds 1\{Y\|\bm s \circ \bm \Theta\|>1\}\right], \label{eq:mean-moment-est}
\end{align}
where we used Eq.~\eqref{eq:const-relation} and Eq.~\eqref{eq:spectral-trafo}.
Beyond the limiting behaviour of the expectation, the following theorem establishes asymptotic normality of $\widehat{M}_{n,u}(\bm v,\bm 1_I,\bm 1, p)$. A detailed proof is given in Appendix \ref{app:th5}.

\begin{theorem} \label{thm:fclt}
	Let $\bm X^*_l$, $l \in \NN$, be independent copies of a   regularly varying $[0,\infty)^d$-valued random vector $\bm X^*$ with index $\alpha=1$ 
	satisfying Eq.~\eqref{eq:a-normal} and with spectral component denoted by $\bm \Theta$.
	Furthermore, let $u_n \to \infty$ such that $n/a^*(u_n) \to \infty$. Then, for every finite set $K_0 \subset \NN_0$ and $\delta \geq 0$, the sequence of processes $(\{G_{n}(\bm v, \bm s,\bm \beta, p); \, \bm v \in \partial B_1^+(\bm 0), \, (\bm s,\bm \beta) \in A_{\delta}', \, p \in K_0\})_{n \in \NN}$
	with 
	\begin{align} \label{eq:new-gn}
	G_{n}(\bm v, \bm s, \bm \beta, p) ={}&
	\sqrt{\frac{n}{a^*(u_n)}} \left[ a^*(u_n) \widehat M_{n,u_n}(\bm v, \bm s,\bm \beta, p)  - a^*(u_n) \EE[\widehat M_{n,u_n}(\bm v, \bm s,\bm \beta, p)] \right]
	\end{align}     
	converges weakly in $\ell^\infty(\partial B_1^+(\bm 0) \times A_\delta' \times K_0)$ to a tight centered Gaussian
	process $G$ with covariance 
	\begin{align}   \label{eq:cov-g}
	& \Cov(G(\bm v, \bm s, \bm \beta,p_1), G(\bm w, \bm t,\bm \gamma, p_2)) \\
	 ={}& \tau 
	\EE\left[ 
	\left( \bm v^\top  \frac{(Y \bm s \circ \bm \Theta)^{1/\bm \beta}}{\|(Y \bm s \circ \bm \Theta)^{1/\bm \beta}\|} \right)^{p_1}
	\left( \bm w^\top  \frac{(Y \bm t \circ \bm \Theta)^{1/\bm \gamma}}{\|(Y \bm t \circ \bm \Theta)^{1/\bm \gamma}\|}  \right)^{p_2}
	\mathds 1\left\{Y\left( \|\bm s \circ \bm \Theta\| \wedge 
	\|\bm t \circ \bm \Theta\|  \right)>1\right\}  \right]. \nonumber
	\end{align}       
\end{theorem}

A direct consequence of the weak convergence of $G_n$ is that
$a^*(u_n) \widehat{M}_{n, u_n}(\bm v, \bm s, \bm \beta, p)$ converges in probability to 
the limit expression in Eq.~\eqref{eq:mean-moment-est}. Consequently,
the ratio estimator $\widehat{M}_{n, u_n}(\bm v, \bm s, \bm \beta, p) / \widehat{P}_{n, u_n}(\bm s)$ is consistent:
\begin{equation}
\frac{\widehat{M}_{n, u_n}(\bm v, \bm s, \bm \beta,p)}{\widehat{P}_{n, u_n}(\bm s)}
= \frac{\widehat{M}_{n, u_n}(\bm v, \bm s,\bm \beta, p)}{\widehat{M}_{n, u_n}(\bm v, \bm s, \bm \beta, 0)}
\to_p c(\bm v, \bm s,\bm \beta, p)\,,\qquad n\to \infty\,,
\end{equation}
for all $\bm v \in \partial B_1^+(\bm 0)$, $(\bm s,\bm \beta) \in A_\delta'$ and $p \in \NN_0$
where
\begin{equation*}
 c(\bm v, \bm s,\bm \beta, p)  := \frac 1 {\EE[\|\bm s \circ \bm \Theta\| ]} \cdot
\EE\left[ \left(\bm v^\top \frac{ (Y\bm s \circ \bm \Theta)^{1/\bm \beta}}{\|(Y\bm s \circ \bm \Theta)^{1/\bm \beta}\|}\right)^p \mathds 1\{Y \|\bm s \circ \bm \Theta\|>1\} \right].
\end{equation*} 

Applying Thm.~\ref{thm:fclt} and using the functional Delta method, we can even establish the desired functional limit theorem for the ratio estimator 
$\widehat{M}_{n, u_n}(\bm v, \bm s,\bm \beta, p) / \widehat{P}_{n, u_n}(\bm s)$ provided that we can neglect the preasymptotic bias that arises from the fact that Eq.~\eqref{eq:mean-moment-est} yields an asymptotic relation only.

\begin{corollary} \label{coro:fclt-ratio}
	Let the assumptions of Thm.~\ref{thm:fclt} hold and  assume that there exist some $\delta \geq 0$ and some finite subset $K \subset \NN$ such that,
	for all $p \in K \cup \{0\}$,
    \begin{equation} \label{eq:bias-negligible}
     \sqrt{\frac{n}{a^*(u_n)}}
     \left| a^*(u_n) \EE[\widehat{M}_{n,u_n}(\bm v, \bm s,\bm \beta, p)] - 
     \tau\, \EE[\|\bm s \circ \bm X^*\|] \, c(\bm v, \bm s,\bm \beta, p) 
 \right| \to 0
    \end{equation}
    uniformly in $\bm v \in \partial B_1^*(\bm 0)$ and $(\bm s,\bm \beta) \in A_{\delta}'$.
    
    Then, the sequence of processes 
	$(\{\widetilde G_n(\bm v, \bm s,\bm \beta, p); \, \bm v \in \partial B_1^+(\bm 0), 
	\, (\bm s,\bm \beta) \in A_{\delta}', \, p \in K\})_{n\in\NN}$ with
	$$ \widetilde G_n(\bm v, \bm s,\bm \beta, p) = \sqrt{\frac{n}{a^*(u_n)}} 
	\left( \frac{\widehat{M}_{n, u_n}(\bm v, \bm s,\bm \beta, p)}
	{\widehat{P}_{n, u_n}(\bm s)} - c(\bm v, \bm s,\bm \beta, p) 	\right),$$
	converges weakly in 
	$\ell^\infty(\partial B_1^+(\bm 0) \times A_\delta' \times K)$ to a tight centered Gaussian process $\widetilde G$.
\end{corollary}
The proof of Cor.~\ref{coro:fclt-ratio} is deferred to Appendix \ref{app:cor6}. The expression of the covariance of $\widetilde G$ is provided there in Eq.~\eqref{eq:covar-tildeg}.
\begin{remark} \label{rem:joint-conv-cont}
\begin{enumerate}
\item For the original estimator with deterministic threshold as given in
    \eqref{eq:def-ratio-estimator}, it is sufficient to restrict to
    the case that $\bm s = \bm t = \bm 1_I$ and $\bm \beta = \bm 1$ in Cor.~\ref{coro:fclt-ratio}. 
	Then, we have 
	$\|\bm s \circ \bm \Theta\|_\infty = \|\bm t \circ \bm \Theta\|_\infty
	= \ell_I(\bm \Theta)$ and the covariance of $\widetilde G$ given in
	\eqref{eq:covar-tildeg} simplifies to
	\begin{align*}
	&\Cov\left(\widetilde G(\bm v, \bm 1_I, \bm 1, p_1), \widetilde G(\bm w, \bm 1_I, \bm 1, p_2)\right) \\
	={}& \frac{1}{\tau \EE \ell_I(\bm \Theta)} \cdot \bigg( \frac 1 {\EE \ell_I(\bm \Theta)} 
	\EE\left[ \left(\bm v^\top \frac{\bm \Theta_I}
	{\ell_I(\bm \Theta)}\right)^{p_1} \cdot \left(\bm w^\top \frac{\bm \Theta_I}{\ell_I(\bm \Theta)}\right)^{p_2} \cdot \ell_I(\bm \Theta)\right] \\
	& \hspace{1.5cm}
	- \frac 1 {(\EE \ell_I(\bm \Theta))^2} \EE\left[ \left(\bm v^\top \frac{\bm \Theta_I}{\ell_I(\bm \Theta)}\right)^{p_1} \cdot \ell_I(\bm \Theta)\right]
	\EE\left[ \left(\bm w^\top \frac{\bm \Theta_I}{\ell_I(\bm \Theta)}\right)^{p_2} \cdot \ell_I(\bm \Theta)\right] \bigg) \\
	={}& \frac 1 {\tau_I} \Cov\left( (\bm v^\top \bm \Theta_I^I)^{p_1},
	 (\bm w^\top \bm \Theta_I^I)^{p_2} \right).
	\end{align*}
\item Using the same techniques as in the proof of Cor.~\ref{coro:fclt-ratio}, one can easily show the joint weak convergence of 
	$(G_n(\bm v, \bm s,\bm \beta, p_1),\widetilde G_n(\bm w, \bm t,\bm \gamma, p_2))_{n\in \NN}$  $(\bm v, \bm s,\bm \beta, p_1),$ $(\bm w, \bm t,\bm \gamma, p_2)$ in $\partial B_1^+(\bm 0) \times A_\delta' \times K$ to $(G(\bm v, \bm s,\bm \beta, p_1),\tilde G(\bm w, \bm t, \bm \gamma,p_2))$ which are jointly Gaussian. Then for $\bm \beta=\bm \gamma=\bm 1$, the covariances express as
	\begin{align*} 
	  & \Cov(G(\bm v, \bm s,\bm 1, p_1), \widetilde G(\bm w, \bm t,\bm1, p_2)) \nonumber \\ 
	 ={}& \frac{
	   \EE\Big[ \Big(\bm v^\top \frac{\bm s \circ \bm \Theta}{\|\bm s \circ \bm \Theta\|} \Big)^{p_1}
	   \Big(\bm w^\top \frac{\bm t \circ \bm \Theta}{\|\bm t \circ \bm \Theta\|} \Big)^{p_2} ( \|\bm s \circ \bm \Theta\| \wedge \|\bm t \circ \bm \Theta\|)\Big]}
	   {\EE[\|\bm t \circ \bm \Theta\|]} \nonumber \\
	   & - \frac{\EE\Big[ \Big(\bm v^\top \frac{\bm s \circ \bm \Theta}{\|\bm s \circ \bm \Theta\|} \Big)^{p_1}
	   ( \|\bm s \circ \bm \Theta\| \wedge \|\bm t \circ \bm \Theta\|)\Big]
	   \EE\Big[\Big(\bm w^\top \frac{\bm t \circ \bm \Theta}{\|\bm t \circ \bm \Theta\|} \Big)^{p_2} \|\bm t \circ \bm \Theta\|\Big]
	}{(\EE[\|\bm t \circ \bm \Theta\|])^2}.
	\end{align*}
\item From the proofs of Thm.~\ref{thm:fclt} and Cor.~\ref{coro:fclt-ratio}, it follows that, for fixed $p \in \NN_0$, the sequences of processes $G_n(\cdot,\cdot,\cdot,p)$ and $\widetilde G_n(\cdot,\cdot,\cdot,p)$ are asymptotically uniformly equicontinuous on $\partial B_1^+(\bm 0) \times A_\delta'$ w.r.t.\ the semimetric 
	$$ \rho((\bm v, \bm s, \bm \beta), (\bm w, \bm t, \bm \gamma)) = \max\{\|\bm v - \bm w\|, \|\bm s - \bm t\|,\|\bm \beta - \bm \gamma\|\}.$$
	Consequently, almost all sample paths of the tight limits $G$
	and $\widetilde G$, respectively, are uniformly $\rho$-continuous
	\citep[cf.\ Addendum 1.5.8 in][]{vdv-wellner-1996}. Since $\rho$ induces the
	standard topology on $\partial B_1^+(\bm 0) \times A_\delta'$ as subset of
	$\RR^{d}\times \RR^{|I|} \times \RR^{d}$, it follows that almost all sample paths of
	$G(\cdot,\cdot,\cdot,p)$ and $\widetilde G(\cdot, \cdot, \cdot, p)$ are uniformly continuous.
\end{enumerate}	
\end{remark}

\subsection{Asymptotic normality for random thresholds via order statistics}
\label{subsec:random-thresh}

Consider the general case of a random vector $\bm X$ that is non standard regularly varying such that $\bm X^* = (r_1^{-1} X_1^{\alpha_1}, \ldots,  r_d^{-1} X_d^{\alpha_d})$ is  regularly varying with index $1$ and satisfies \eqref{eq:vagueconv-norm}. In practice, the marginal distributions of $X_1,\ldots,X_d$ are unknown.
In particular, the indices $\bm\alpha=(\alpha_1,\ldots,\alpha_d)$ of regular variation and the scaling factors $r_1,\ldots,r_d$ that are needed for the transformation to  $\bm X^*$ cannot be assumed to be given.
Furthermore, the normalizing function $a^*(\cdot) \sim \PP(X^*_i > \cdot)$
is unknown.
To cope with these difficulties, we propose to work with an estimator $\widehat{\bm \alpha}$ of $\bm\alpha$ and  order statistics of $\bm X$.
More precisely, let $X_{k:n,i}$ denote the $k$th upper order statistic of
$X_{1i}, \ldots, X_{ni}$ for $i=1,\ldots,d$ and let $\bm X_{k:n}
= (X_{k:n,1}, \ldots, X_{k:n,d})$.  With no loss of generality we assume that $X_{k:n,i}>0$, $1\le i\le d$. We propose the following estimator
\begin{align*}
\frac{\widetilde{M}_{n,k,I}(\bm v,p)}{\widetilde{P}_{n,k,I}}
= \frac{n^{-1} \sum_{l=1}^n \left(\bm v_I^\top 
 \frac{(\bm X_{l}/\bm X_{k:n})^{\widehat{ \bm\alpha}}}{\ell_I((\bm X_{l}/\bm X_{k:n})^{\widehat{ \bm\alpha}})}\right)^p
	 \mathds 1\{\ell_I(\bm X_l / \bm X_{k:n}) > 1\}}
{n^{-1} \sum_{l=1}^n \mathds 1\{\ell_I(\bm X_l / \bm X_{k:n}) > 1\}} 
\end{align*}
for some $k < n$, where the ratio of vectors $\bm X_l / \bm X_{k:n}$
is to be interpreted componentwise. Expanding each of the  ratios $\bm X_l / \bm X_{k:n}$
by $r_1^{-1},\ldots,r_d^{-1}$, respectively, and taking the power of $\bm\alpha$ coordinatewise, it can be easily seen that 
\begin{align*}
\frac{\widetilde{M}_{n,k,I}(\bm v,p)}{\widetilde{P}_{n,k,I}}
= \frac{n^{-1} \sum_{l=1}^n \left(\bm v_I^\top 
	\frac{(\bm X^*_{l}/\bm X^*_{k:n})^{\widehat{\bm \alpha} /\bm \alpha}}{\ell_I((\bm X^*_{l}/\bm X^*_{k:n})^{\widehat{\bm \alpha} /\bm \alpha})}\right)^p
	\mathds 1\{\ell_I(\bm X^*_l / \bm X^*_{k:n}) > 1\}}
{n^{-1} \sum_{l=1}^n \mathds 1\{\ell_I(\bm X^*_l / \bm X^*_{k:n}) > 1\}},
\end{align*}
i.e.\ we can rewrite 
 \begin{equation} \label{eq:rel-deterministic-random}
 \frac{\widetilde{M}_{n,k,I}(\bm v,p)}{\widetilde{P}_{n,k,I}}
 = \frac{\widehat{M}_{n,u_n}(\bm v_I, u_n / \bm X^*_{k:n} \circ \bm 1_I,\bm \alpha / \widehat{\bm \alpha}, p)}{\widehat{P}_{n,u_n}(u_n / \bm X^*_{k:n} \circ \bm 1_I)}.
 \end{equation}
 For an appropriate sequence $\{k_n\}_{n \in \NN} \subset \NN$ with $k_n \to \infty$ such that
 $k_n / n \to 0$, we want to show that
 $$ \left(\left\{ \sqrt{k_n} \left( \frac{\widetilde{M}_{n,k_n,I}(\bm v,p)}
   {\widetilde{P}_{n,k_n,I}}
 - \EE[ (\bm v_I^\top \bm \Theta^I_I)^p] \right); \, \bm v \in \partial B_1^+(\bm 0) \right\}\right)_{n\in\NN} $$
 converges weakly to a Gaussian process.
The functional central limit theorem will essentially be proven by plugging the random scaling factor 
 $\bm s =  u_n / \bm X^*_{k:n} \circ \bm 1_I$ and the random power $\bm\beta = \bm \alpha / \widehat {\bm\alpha}$ 
 into the results of Subsection \ref{subsec:fixed-thresh} according to Eq.~\eqref{eq:rel-deterministic-random}. Provided that both the random $\bm s$ and the random $\bm\beta$ are asymptotically
 normal, we can use the functional central limit theorem given in Cor.~\ref{coro:fclt-ratio} to
 obtain the desired result.
 \medskip
 
 To this end, we will first specify our estimators $\widehat {\bm\alpha}_{n,k}=(\widehat\alpha_{n,k,1},\ldots,\widehat\alpha_{n,k,d})$ of $\bm\alpha$ which will be componentwise Hill estimators. Similarly to Section \ref{subsec:fixed-thresh}, we will first start with an estimator with a
 fixed sequence of thresholds $\{u_n\}_{n \in \NN}$ and normalized observations
 $\bm X^*_1,\ldots, \bm X^*_n$. To mimic the behaviour of the real non-normalized data $\bm X = (\bm r \circ \bm X^*)^{1/\bm\alpha}$ where $\bm r = (r_1,\ldots, r_d)$, we also allow for 
 a scaling factor $\bm s \in \bigcup_{i=1}^d A_{\delta,\{i\}}$, that is, a scaling factor with exactly one non-zero component, and take the logarithm of the scaled vector to the  coordinatewise power of $1/\bm\alpha$.
 Thus, for $\bm s \in A_{\delta, \{i\}}$, $1\le i\le d$, we obtain the estimator
 \begin{equation} \label{eq:hill-fixed}
 \frac{\widehat L_{n,u}(\bm s)}{\widehat P_{n,u}(\bm s)} =
 \frac{n^{-1} \sum_{l=1}^n \log\left( \left\| \left(\bm s \circ \frac{\bm X_l^*}{u_n}\right)^{1/\bm\alpha} \right\| \right) \mathds 1\{\|\bm s \circ \bm X_l^*\| > u_n\}}{\widehat P_{n,u}(\bm s)}. 
 \end{equation}
 
 Under the assumptions of Thm.~\ref{thm:fclt}, it can be shown that for $\bm s \in A_{\delta, \{i\}}$, $1\le i\le d$,
\begin{align*}
a^*(u_n) \EE[\widehat L_{n,u_n}(\bm s)] ={}& \frac{\PP(s_i X_i^* > u_n)}{\PP(X_i^* > u_n)} \cdot
\EE\Big[\frac 1 {\alpha_i} \log\Big(\frac{s_i X^*_i}{u_n}\Big)
\, \Big| \, s_i  X^*_i > u_n\Big] 
\to{} \frac {s_i}{\alpha_i} \quad (n \to \infty)\, .
\end{align*}

In order to neglect the bias we assume the following second order condition from \cite{dehaan-ferreira-2006}:

There exists an auxiliary positive function $A^*_i$ such that $A^*_i(t)\to 0$ as $t\to \infty$ and 
\begin{equation}\label{cond:secorder}
\lim_{x\to\infty}\dfrac{a^*(x)\PP(X_i^*>xs)-s^{-1}}{A^*_i( a^*(x))}=K_i(s)
\end{equation}
where $K_i$ is not identically $0$ for any $1\le i\le d$.

Now if we choose $u_n$ such that $u_n\to \infty$ and  $\sqrt{n/a^*(u_n)}A^*_i( a^*(u_n))\to 0$ as $n\to \infty$ for some $1\le i\le d$ we obtain\
 \begin{equation} \label{eq:bias-hill}
\sup_{\bm s \in A_{\delta,\{i\}}} \sqrt{\frac{n}{a^*(u_n)}} \left| a^*(u_n) \EE[\widehat L_{n,u_n}(\bm 1_{\{i\}},\bm s)] - \frac{s_i}{\alpha_i}\right| \to\,, \qquad n\to \infty\,,
\end{equation}
as well as
\begin{equation} \label{eq:one-dim-bias}
\sup_{(1+\delta)^{-1} \leq s_i  \leq 1 +\delta} 
 \sqrt{\frac{n}{a^*(u_n)}} \left| a^*(u_n) \PP(s_i  X^*_i > u_n) - s_i  \right| 
\to 0\,, \qquad n\to \infty\,,
\end{equation}
which is nothing else than  Eq.~\eqref{eq:bias-negligible} for $p=0$ and $\bm s \in  A_{\delta, \{i\}}$, $\delta > 0$.
As Eq.~\eqref{eq:bias-hill} and \eqref{eq:one-dim-bias} ensure that the biases of the marginal Hill estimators become asymptotically negligible, the second-order condition \eqref{cond:secorder} will be sufficient for us to establish asymptotic normality of our final marginal Hill  estimators $\widehat \alpha_{k,n,i}$ defined via order statistics
 \begin{equation} \label{eq:hill-def}
 \frac 1 {\widehat \alpha_{n,k,i}} = \frac{n^{-1} \sum_{l=1}^{n} \log\left( \ell_{\{i\}}(\bm X_l / \bm X_{k:n}) \right) \mathds 1 \left\{ \ell_{\{i\}}(\bm X_l / \bm X_{k:n}) > 1\right\}  }{\widetilde P_{n,k,\{i\}}}
 \end{equation}
 in the following theorem. A detailed proof is given in Appendix \ref{app:th8}.

 \begin{theorem} \label{thm:an-hill}
Let $\bm X_l$, $l \in \NN$, be independent copies of a non-standard regularly varying
$[0,\infty)^d$-valued random vector $\bm X$.
Assume that the  vector $\bm X^* = (r_1^{-1} X_1^\alpha,\ldots, r_d^{-1} X_d^\alpha)$ satisfies Eq.~\eqref{eq:a-normal} and \eqref{cond:secorder}. Let $\{k_n\}_{n \in \NN} \subset \NN$ be a sequence such that $k_n \to \infty$, $k_n/n\to0$ and  $\sqrt{k_ n}A^*_i( n/k_n)\to 0$  for all $1\le i\le d$.  Then we have
$$ \sqrt{k_n} \left( {\widehat {\bm \alpha}_{n,k_n}}^{-1} - {\bm \alpha}^{-1}\right) \to_d \widetilde {\bm H}, \qquad n \to \infty,$$
where $\widetilde {\bm H}$ is a $d$-dimensional centered Gaussian random vector with
covariances
$$ \Cov(\widetilde H_i, \widetilde H_j) = \frac{\tau}{\alpha_i \alpha_j} \EE\left[ \Theta_i \wedge \Theta_j\right] = \frac{2- \tau_{ij}}{\alpha_i \alpha_j}. $$ 
\end{theorem}

This result is in line to results for the joint asymptotic behaviour of marginal Hill estimators
obtained in \citet{stupfler-2019}.

\begin{remark} \label{rem:joint-conv-h}
	We can show joint convergence of the different estimates to the limit processes $G$ and $\widetilde G$ and the variables $\widetilde H_i$, $i=1,\ldots,d$, as obtained in Thm.~\ref{thm:fclt}, Cor.~\ref{coro:fclt-ratio} and Thm.~\ref{thm:an-hill}, respectively.
	The covariance between the processes $G$ and $\widetilde G$ is given in  Rem.~\ref{rem:joint-conv-cont}.2.
	Similar calculations as above reveal that, for all $\bm v \in \partial B_1^+(\bm 0)$,  $J \subset \{1,\ldots,d\}$ and $p \in \NN_0$,
	\begin{align} \label{eq:joint-covar-h-g}
	 \Cov(\widetilde H_i, G(\bm v, \bm 1_{J}, \bm 1, p)) 
	={}& \frac{\tau}{\alpha_i} \EE\Big[- \log\Big( 1 \wedge \frac{\| \bm \Theta_J\|}{ \Theta_i} \Big) \cdot \Big(\bm v^\top \frac{\bm \Theta_J}{\|\bm \Theta_J\|}\Big)^p \cdot (\Theta_i\wedge \|\bm \Theta_J\|)\Big].
	\end{align}
	Here, using similar arguments as above, we obtain
	\begin{align} \label{eq:joint-covar-h-tildeg}
	& \Cov(\widetilde H_i, \widetilde G(\bm v, \bm 1_{J}, \bm 1, p)) \\
	={}& \frac{1}{\alpha_i \EE[\|\bm \Theta_J\|]} 
	\bigg( \EE\left[ - \log\Big( 1 \wedge \frac{\| \bm \Theta_J\|}{ \Theta_i} \Big) \left(v^\top \frac{\bm \Theta_J}{\|\bm \Theta_J\|}\right)^p  (\Theta_i \wedge \bm \Theta_J)\right] \nonumber \\
	& \qquad \qquad \qquad - 
		\EE\left[- \log\Big( 1 \wedge \frac{\| \bm \Theta_J\|}{ \Theta_i} \Big) (\Theta_i \wedge \bm \Theta_J) \right] \frac{\EE[\|\bm \Theta_J\| \left(v^\top \frac{\bm \Theta_J}{\|\bm \Theta_J\|}\right)^p]}{\EE[\|\bm \Theta_J\|]} \bigg) \nonumber
	\end{align}
	for all $\bm v \in \partial B_1^+(\bm 0)$  and $p \in \NN_0$. 
	In particular, this implies that $\Cov(\widetilde H_i, \widetilde G(\bm v, \bm 1_{J}, \bm 1, p)) = 0$ if $i \in J$. 
\end{remark}
 
Making use of the asymptotic results for the Hill-type estimator in Thm.~\ref{thm:an-hill} and for the estimators in the case of deterministic thresholds given in Section \ref{subsec:fixed-thresh}, we will establish a functional 
central limit theorem for $\widetilde{M}_{n,k,I}(\bm v_I,p) / \widetilde{P}_{n,k,I}$. Compared to its analogue for deterministic thresholds,
the estimator is blurred by the random marginal normalization based on the order statistics. To ease notation, we just write $G^0(\bm s) = G(\cdot, \bm s, \cdot, 0)$.

Imposing some additional conditions on the function $c$, we obtain the  
following main result.

\begin{theorem} \label{thm:fclt-ratio-rank}
    Let $\bm X_l$,	$l \in \NN$, be independent copies of a non-standard regularly varying
	$[0,\infty)^d$-valued random vector $\bm X$ satisfying the assumptions of Cor.~\ref{coro:fclt-ratio} and Thm.~\ref{thm:an-hill} for $a^*(u_n) \sim  n/k_n$, some $\delta > 0$ and some $K \subset \NN$. Assume that all the partial derivatives
	\begin{align*}
	 c_{s_i}(\bm v, \bm s,\bm \beta, p) ={}& \frac{\partial}{\partial s_i}c(\bm v, \bm s, \bm \beta, p),
	\qquad i \in \{1,\ldots,d\}, \ p \in K,\\
	c_{\beta_i}(\bm v, \bm s, \bm \beta, p) ={}& \frac{\partial}{\partial \beta_i}c(\bm v, \bm s, \bm \beta, p),\qquad i \in \{1,\ldots,d\}, \ p \in K,
	\end{align*}
	exist and are continuous on $\partial B_1^+(\bm 0) \times A_{\delta}'$.
	Then,  for non-empty $I \subset \{1,\ldots,d\}$, we have
	\begin{align*}
	 & \left\{ \sqrt{k} \left( \frac{\widetilde{M}_{n,k,I}(\bm v_I,p)}{\widetilde{P}_{n,k,I}} - c(\bm v_I, \bm 1_I, \bm1, p)\right); \ \bm v \in \partial B_1^+(\bm 0), \,  p \in K \right\}\\
	  & {}\qquad \longrightarrow
	\left\{ \widetilde G(\bm v_I,\bm 1_I,\bm 1,p) - \sum\nolimits_{i \in I}\left( c_{s_i}(\bm v_I, \bm 1_I, \bm 1, p)  G^0(\bm 1_{\{i\}}) + c_{\beta_i}(\bm v_I, \bm 1_I, \bm 1, p) \bm \alpha_i\widetilde H_i\right) \right\}
	\end{align*}
	weakly in $\ell^\infty(\partial B_1^+(\bm 0) \times K)$ as $n\to \infty$.
\end{theorem}	
The proof of Thm.~\ref{thm:fclt-ratio-rank} is deferred to Appendix \ref{app:th10}.

\begin{remark} \label{rem:asymp-variance-orderstats}
In the case $p=1$ we notice that the limiting process is linear in $\bm v_I$ because the process $\widetilde G$ and the functions $c_{s_i}$ and $c_{\beta_i}$ are all linear in $\bm v_I$:
\begin{align*}
 &\widetilde G(\bm v_I,\bm 1_I,\bm 1,1) - \sum\nolimits_{i \in I}\left( c_{s_i}(\bm v_I, \bm 1_I, \bm 1, 1)  G^0(\bm 1_{\{i\}}) + c_{\beta_i}(\bm v_I, \bm 1_I, \bm 1, 1) \bm \alpha_i\widetilde H_i\right)\\
&=\sum_{j \in I}v_j  \left( \widetilde G(\bm 1_{j},\bm 1_I,\bm 1,1) -  \sum\nolimits_{i \in I}\left( c_{s_i}(\bm 1_{j}, \bm 1_I, \bm 1, 1)  G^0(\bm 1_{\{i\}}) + c_{\beta_i}(\bm 1_{j}, \bm 1_I, \bm 1, 1) \bm \alpha_i\widetilde H_i\right)\right)\,.
\end{align*}
\end{remark}

\section{Application: Estimation of extremal coefficients} \label{sec:extr-coeff}

One of the most popular summary statistics for extremal dependence in the (sub-)vector $(X^*_i)_{i \in I}$
is the extremal coefficient $\tau_I$ given by Eq.~\eqref{eq:norm-const-subset}, i.e.\
$$ \tau_I = \lim_{u \to \infty} a^*(u) \PP\left(\max_{i \in I} X^*_i > u\right)
  \sim \frac{\PP(\max_{i \in I} X^*_i > u) }{\PP(X^*_j > u) }\,, 
  \quad \text{ for all } j \in I.$$
From this representation and the union bound for $\PP(\max_{i \in I} X^*_i > u)$, it can be easily seen that $\tau_I \in [1,|I|]$ where $\tau_I= 1$ corresponds to the case that all the components of $(X_i)_{i \in I}$ are 
(asymptotically) fully dependent, while they are all asymptotically independent if $\tau_I = |I|$. This gives rise to the common interpretation of the extremal coefficient as the number of asymptotically independent variables among $(X_i^*)_{i \in I}$ \citep[cf.][]{schlather-tawn-2003}. Even though, in general, the spectral measure of the vector $\bm X = (X_1,\ldots,X_d)$ cannot be fully described by the set of extremal coefficients $(\tau_I)_{I \subset \{1,\ldots,d\}}$, they allow for the identification of the parameters in many popular parametric families and, consequently, statistical efficiency in the estimation of the extremal coefficients is expected to carry over to the estimation of the corresponding model parameters. Illustrations on the H\"usler-Reiss are given in Section \ref{sec:illustr}.

From Eq.~\eqref{eq:mean-spectral-ell}, we obtain that
$$ \EE\left[ \bm v_I^\top \bm \Theta^I \right]
= \EE\left[ (\bm v_I^\top \bm \Theta^I)_+ \wedge 1 \right] = \frac 1 {\tau_I}\,, $$
for all $\bm v_I \in \partial B_1^+(\bm 0)$. This motivates the use of the
two classes of estimators considered in Eq.~\eqref{eq:def-benchmark-estimator} and Eq.~\eqref{eq:def-ratio-estimator}, respectively. We will distinguish two cases; Section \ref{sec:deterministic} collects the asymptotic normality results for the case of known marginal distributions where we can exploit the results for deterministic thresholds given in Section \ref{subsec:fixed-thresh} with $\bm s = \bm 1_I$ and $\beta=1$. The results derived in Section \ref{subsec:random-thresh} can be used when the margins are unknown as we will investigate in Section \ref{sec:random}. 

\subsection{Known margins}\label{sec:deterministic}

In this section, we focus on the case where the margins are known, i.e.~$\bm X_i^*$ is observed.

\subsubsection{The inefficiency of the benchmark}

For this case, as a first class of estimators, we can directly consider the estimators based on relative frequencies of threshold
exceedances by convex combinations, see Eq.~\eqref{eq:def-benchmark-estimator}, namely
$$  \frac{\widehat{P}_{n,u,I}^{\mathrm{(conv)}}(\bm v_I)}{\widehat{P}_{n,u,I}}
  = \frac{n^{-1} \sum_{l=1}^n \mathbf{1}\{\bm v_I^\top \bm X^*_l > u\}}
         {n^{-1} \sum_{l=1}^n \mathbf{1}\{\ell_I(\bm X^*_l) > u\} } 
  , \qquad \bm v_I \in \partial B_1^+(\bm 0), $$
where we use the fact that $\bm v_I^\top \bm X^{*} > u$ already implies that
$\ell_I(\bm X^*) > u$. This type of estimator will serve as a benchmark.
By standard arguments, for a sequence $\{u_n\}_{n \in \NN}$ of thresholds such that $u_n \to \infty$ and $n / a^*(u_n) \to \infty$ as $n \to \infty$, we obtain the asymptotic normality
\begin{align*}
\sqrt{\frac{n}{a^*(u_n)}} \left( \begin{array}{c} 
    a^*(u_n) \widehat{P}^{\mathrm{(conv)}}_{n,u_n,I}(\bm v_I) 
  - a^*(u_n) \PP(\bm v_I^\top \bm X^* > u_n)\\
    a^*(u_n) \widehat{P}_{n,u_n,I} - a^*(u_n) \PP(\ell_I(\bm X^*) > u_n)
\end{array} \right)
\stackrel{d}{\longrightarrow} \mathcal{N}( \bm 0, \bm \Sigma),
\end{align*}
where 
\begin{align*} 
 \bm \Sigma =
 \begin{pmatrix} 
 \tau_I \EE[\bm v^\top_I \bm \Theta^I] & \tau_I \EE[\bm v^\top_I \bm \Theta^I] \\
 \tau_I \EE[\bm v^\top_I \bm \Theta^I] & \tau_I \end{pmatrix}
 = \begin{pmatrix} 1 & 1 \\ 1 & \tau_I \end{pmatrix}.
\end{align*}
Assuming that the bias is negligible, i.e.\
$$ \sqrt{\frac{n}{a^*(u_n)}} \left( a^*(u_n) \PP(\bm v_I^\top \bm X^* > u_n)
   - 1 \right) \to 0$$
and 
\begin{equation} \label{eq:bias-extr-coeff}
\sqrt{\frac{n}{a^*(u_n)}} \left( a^*(u_n) \PP(\ell_I(\bm X^*) > u_n) - \tau_I \right) \to 0,
\end{equation}
the Delta method yields asymptotic normality of the benchmark estimator for the case of known margins,
$$ \frac{1}{\widehat{\tau}_{I,n}^{BK}(\bm v_I)} =  \frac{\widehat{P}^{\mathrm{(conv)}}_{n,u_n,I}(\bm v_I)}{\widehat{P}_{n,u_n,I}}, \quad \bm v \in \partial B_1^+(\bm 0),$$
i.e.
 $$ \sqrt{\frac{n}{a^*(u_n)}} \left(\frac{1}{\widehat{\tau}_n^{BK}(\bm v_I)} - \frac 1 {\tau_I} \right)
\stackrel{d}{\longrightarrow} \mathcal{N}(0, \tau_I^{-3} (\tau_I-1)). $$
Here, we note that the asymptotic variance does not depend on the choice of $\bm v \in \partial B_1^+(\bm 0)$. 
 \medskip
 
Secondly, we consider moment-based estimators as discussed in Section \ref{subsec:fixed-thresh}, i.e.\ estimators based on empirical means of convex combinations, namely
$\widehat{M}_{n,u,I}(\bm v_I,1) / \widehat{P}_{n,u,I}$ with weights $\bm v_I \in \partial B_1^+(\bm 0)$
as defined in Eq.~\eqref{eq:def-ratio-estimator}. Provided that the bias is negligible, i.e.\
$$ \sqrt{\frac{n}{a^*(u_n)}} \left( a^*(u_n) \EE\left[\bm v_I^\top \frac{\bm X^*}{\ell_I(\bm X^*)} \cdot \mathds 1\{\ell_I(\bm X^*)  > u_n\}\right]
- 1 \right) \to 0$$
and \eqref{eq:bias-extr-coeff}, we obtain the asymptotic normality
 $$ \sqrt{\frac{n}{a^*(u_n)}} \left( \frac{\widehat{M}_{n,u_n,I}(\bm v_I)}
 {\widehat{P}_{n,u_n,I}} - \frac 1  {\tau_I} \right)
\stackrel{d}{\longrightarrow} 
\mathcal{N}\left(0,  \tau_I^{-1} \Var(\bm v_I^\top \bm \Theta^I)\right),$$
 cf.\ the results in Cor.~\ref{coro:fclt-ratio} and the first part of Rem.~\ref{rem:joint-conv-cont}.
\medskip

Since $\bm v_I^\top \bm \Theta^I \leq 1$ a.s., we have that
$$ \frac 1 {\tau_I} \Var(\bm v_I^\top \bm \Theta^I)
 = \frac 1 {\tau_I} \left(\EE[(\bm v_I^\top \bm \Theta^I)^2] - \frac 1 {\tau_I^2}\right)
 \leq{} \frac 1 {\tau_I} \left( \EE[(\bm v_I^\top \bm \Theta^I)] - \frac 1 {\tau_I^2}  \right) = \frac{\tau_I - 1}{\tau_I^3} $$
and, therefore, the asymptotic variance of the moment-based approach is always less than or equal to the asymptotic variance of the benchmark estimator. Notice that the variances of both
estimators are equal, i.e.\ $\EE[(\bm v_I^\top \bm \Theta^I)^2] = \EE[\bm v_I^\top \bm \Theta^I] = 1/\tau_I$, if and only if $\Theta_i^I \equiv 1$ a.s.\ for all $i$ with $v_i > 0$. This is the case if and only if $\tau_I =1$ and then both asymptotic variances are equal to $0$. Thus, except for this degenerate case, the moment-based approach is always preferable to the benchmark  in terms of the asymptotic variance.

Furthermore, in general, the asymptotic variance of the moment-based estimator depends on the vector $\bm v_I \in \partial B_1^+(\bm 0)$. As we will see in Example \ref{ex:minvar-known-margins}, the minimal variance is typically not achieved for a standard vector of the type $\bm v_{\{i\}}$ for $i \in I$, i.e.\ the empirical estimators of $\EE[\Theta_i]$, but by a non-trivial mixture of different basis vectors. This improvement can be attributed to the fact that the moment-based estimator is linear in $\bm v$. Thus optimizing in $\bm v_I \in \partial B_1^+(\bm 0)$ can be interpreted as aggregating different estimators. It is well-known that the use of a combination of estimators can improve the accuracy. We note that we could also aggregate the benchmark estimator for different $\bm v$. However this strategy is much more intricate to study as the benchmark estimator is not linear in $\bm v$. 

Based on this discussion, for the remainder of Section \ref{sec:deterministic}, we focus on convex combinations of moment-based estimators rather than on the benchmark. In the following, we optimize the asymptotic variance $\tau_I^{-1} \Var(\bm v_I^\top \bm \Theta^I)$ w.r.t.~the weight vectors $\bm v_I$ and we find some approximation of this best aggregation.  We should notice that the optimal asymptotic variance $\Var(\bm v_I^\top \bm \Theta^I)$ may be degenerate even when $\tau_I>1$. In these cases the results hereinafter still hold with a null limit.

\begin{remark} \label{rem:degen-limit}
The asymptotic variance of the moment-based estimator is zero if and only if there exists a weight vector $\bm v_I$ such that $\bm v_I^\top \bm \Theta_I^I=1/{\tau_I}$ almost surely. This happens in case of an asymptotically independent vector $\bm X$ with  $\bm v$ satisfying $v_i=\frac1{|I|}$ for $i\in I$. It might also happen in more complex cases. For instance, assume that the vector $\bm \Theta_I^I$ takes at only a finite number of different possible values $\bm \theta_I^{(j)}$, $j \in 1,\ldots,m$. Then, $\bm v_I$ is given as the solution of the linear system $\bm v_I^\top \bm \theta_I^{(j)} = 1 / \tau_I$, $j \in 1,\ldots,m$, provided that the solution is in the simplex.
\end{remark}

\subsubsection{Variance minimization}

We focus on the moment estimators   $\widehat{M}_{n,u,I}(\bm v_I,1) / \widehat{P}_{n,u,I}$ for different $\bm v_I \in \partial B_1^+(\bm 0)$. As all of these estimators are consistent, it is appealing to choose the estimator that provides the minimal asymptotic variance.
Thus, in the following, we will focus on determining the weight
vectors $\bm v^*_I \in \partial B_1^+(\bm 0)$ that minimize  the asymptotic variance 
$$  \tau_I^{-1} \Var(\bm v_I^\top \bm \Theta_I^I) 
  = \tau_I^{-3} \left(\tau_I^2 \cdot \EE[(\bm v_I^\top \bm \Theta_I^I)^2] - 1 \right),$$
or, equivalently, the function
$$ f^{(I)}: \ \bm v_I \mapsto \EE[(\bm v_I^\top \bm \Theta_I^I)^2]. $$
  We will show that, under mild conditions,
 the minimizer of $f^{(I)}$ in $\partial B_1^+(\bm 0)$ is unique. 

\begin{proposition} \label{prop:uniqueness}
Let $\bm \Theta$ be a spectral vector. Then,  for non-empty $I \subset \{1,\ldots,d\}$, it holds:
 If the matrix $\bm V_I = (\EE[\Theta^I_i \Theta^I_j])_{i,j \in I}$ is conditionally positive definite, i.e.
 $$ \bm a^\top \bm V_I \bm a > 0\,,$$
 for all $\bm a \in \RR^{|I|} \setminus \{\bm 0\}$ such that $\sum_{i \in I} a_i = 0$,
 then the function $\bm v_I \mapsto f^{(I)}(\bm v_I)$
 attains its minimum on $\partial B_1^+(\bm 0)$ at some unique  
 $\bm v^*_I \in \partial B_1^+(\bm 0)$.
\end{proposition}

\begin{proof}
   First, we note that 
    $$ f^{(I)}(\bm v_I) = \sum\nolimits_{i \in I} \sum\nolimits_{j \in I}
     v_i v_j \EE(\bm \Theta_i^I \bm \Theta_j^I) 
    = (v_i)_{i \in I}^\top \bm V_I (v_i)_{i \in I}. $$
   By a slight abuse of notation, henceforth, we will write
   $f^{(I)}(\bm v_I) = \bm v_I^\top \bm V_I \bm v_I$.
   Now, assume that $f^{(I)}$ attains its minimum at two distinct points
   $\bm v_I \in \partial B_1^+(\bm 0)$ and $\bm w_I \in \partial B_1^+(\bm 0)$. Then, for all
   $ \lambda \in (0,1)$, we have $\lambda \bm v_I + (1-\lambda) \bm w_I
   \in \partial B_1^+(\bm 0)$, $\sum_{i \in I} (v_i-w_i) = 0$ and,
   therefore,
   \begin{align*}
   f^{(I)}(\lambda \bm v_I + &(1-\lambda) \bm w_I) ={}
  (\bm w_I + \lambda (\bm v_I - \bm w_I))^\top \bm V_I 
  (\bm v_I + (1-\lambda) (\bm w_I - \bm v_I)) \\
    ={}& - \lambda (1-\lambda) (\bm v_I - \bm w_I)^\top \bm V_I (\bm v_I - \bm w_I) + \lambda \bm v_I^\top \bm V_I \bm v_I + (1-\lambda) \bm w_I^\top \bm V_I \bm w_I\\
   <{}& \lambda f^{(I)}(\bm v_I) + (1-\lambda)  f^{(I)}(\bm w_I) ,
   \end{align*}
  which is a contradiction to the choice of  $\bm v_I$ and $\bm w_I$ as minimizers of $f^{(I)}$.
\end{proof}

\begin{example} \label{ex:minvar-known-margins}
 In the case of a pairwise extremal coefficient $\tau_I$ with $I =\{i,j\}$,
 we obtain an explicit expression for the minimizer $\bm v_{\{i,j\}}^*$ 
 with $v_i^*, v_j^* \geq 0$ and $v_i^* + v_j^* = 1$ of the function $f^{(I)}(\bm v_{\{i,j\}})  = \EE[(v_i \Theta_i^I + v_j \Theta_j^I)^2] = (v_i, v_j) \bm V_I (v_i, v_j)^\top$.
 To this end, we note that the function
 $$ v_i \mapsto  (v_i, 1-v_i) \bm V_I
          \left(\begin{array}{c} v_i\\ 1-v_i \end{array} \right)
      = v_i^2 \EE[(\Theta_i^I)^2] + 2 v_i (1-v_i) \EE[\Theta_i^I \Theta_j^I] + (1 - v_i)^2 \EE[(\Theta_j^I)^2] $$
 is a quadratic function which possesses a unique minimizer $v_i^* \in [0,1]$ if 
 and only if $\EE[(\Theta_i^I)^2] - 2 \EE[\Theta_i^I \Theta_j^I] + \EE[(\Theta_j^I)^2] > 0$, i.e.\ if and only if
 $\bm V_I$ is conditionally positive definite, cf.\ Prop.~\ref{prop:uniqueness}.
 From 
\begin{align*}
\EE \left[ \Theta_i^I \cdot \Theta_j^I \right] ={}& \frac{1}{\EE[\ell_I(\bm \Theta)]} \cdot \EE\left[ \frac{\Theta_i \Theta_j}{\Theta_i \vee \Theta_j} \right] 
{}={}  \frac{1}{\EE[\ell_{I}(\bm \Theta)]} \cdot \EE\left[ \Theta_i \wedge \Theta_j \right] \\
{}={}& \frac{1}{\EE[\ell_I(\bm \Theta)]} \cdot \EE\left[ \Theta_i + \Theta_j  - \Theta_i \vee \Theta_j \right] 
{}={} \frac{\tau }{\tau_{I}} \cdot
\left( \frac 2 {\tau } 
- \frac{\tau_{I}}{\tau } \right)
{}={} \frac 2 {\tau_{I}} - 1, 
\end{align*}
we obtain that
\begin{equation} \label{eq:assess-var-covar1}
 \EE[(\Theta_i^I)^2] \geq [\EE(\Theta_i^I)]^2 = \frac 1 {\tau_I^2} \geq \frac{2}{\tau_I} - 1 = \EE \left[ \Theta_i^I \cdot \Theta_j^I \right]
\end{equation}
 and, analogously,
 \begin{equation} \label{eq:assess-var-covar2}
 \EE[(\Theta_j^I)^2] \geq \EE \left[ \Theta_i^I \cdot \Theta_j^I \right] \,.
 \end{equation}
 Note that equality in Eq.~\eqref{eq:assess-var-covar1} and in Equation
 \eqref{eq:assess-var-covar2} holds if and only if we have $\EE[(\Theta_i^I)^2]
 =\EE[\Theta_i^I] = 1$ and $\EE[(\Theta_j^I)^2]
 =\EE[\Theta_j^I] = 1$, respectively. Thus, for $\tau_I > 1$,  
 there exists a unique minimizer which can be easily computed:
 \begin{align*}
   v^*_i ={}& \frac{\EE[(\Theta_j^I)^2] - \EE[\Theta_i^I \Theta_j^I]}
       {\EE[(\Theta_i^I)^2] - 2\EE[\Theta_i^I \Theta_j^I]
       	+ \EE[(\Theta_j^I)^2]}
       = \frac{\EE[(\Theta_j^I)^2] - \EE[\Theta_i^I \Theta_j^I]}
              {\EE[(\Theta_i^I - \Theta_j^I)^2]} \\
\text{and} \quad 
v_j^* ={}& 1 - v_i^* 
= \frac{\EE[(\Theta_i^I)^2] - \EE[\Theta_i^I \Theta_j^I]}
       {\EE[(\Theta_i^I - \Theta_j^I)^2]}.   
\end{align*}
By Eq.~\eqref{eq:assess-var-covar1} and
\eqref{eq:assess-var-covar2}, we obtain $v_i^*, v_j^* > 0$,
i.e.\ $\bm v_{\{i,j\}}^* \in \partial B_1^+(\bm 0)$.
Some further calculations lead to the minimal asymptotic variance $ \tau_I^{-3} (\tau_I^2 \cdot f^{(I)}(\bm v_I)-1)$ with
$$ f^{(I)}(\bm v^*_I) 
 = \frac{\EE[(\Theta_i^I)^2] \cdot \EE[(\Theta_j^I)^2] 
 	      - \left( \EE [\Theta_i^I \Theta_j^I] \right)^2 }
 	    {\EE[(\Theta_i^I - \Theta_j^I)^2]}. $$
\end{example}

\subsubsection{Asymptotic normality of the plug-in estimator} \label{sec:asymptotic}

We use the functional central limit theorems Cor.~\ref{coro:fclt-ratio} and Thm.~\ref{thm:fclt-ratio-rank} to show that estimators with minimal asymptotic variance can be obtained
by plugging-in consistent estimators of the optimal weight
vectors.

The variance minimizing vector $\bm v^*_I$ can be estimated by taking the minimum of the empirical counterpart of $f^{(I)}$:
$$ \widehat{\bm v}_{n,u,I}^* = \argmin_{\bm v_I \in \partial B_1^+(\bm 0)} \frac{\widehat{M}_{n,u,I}(\bm v_I,2)}{\widehat{P}_{n,u,I}}.$$
It is important to notice that, being constructed as the empirical second moment of a convex combination,
$\bm v_i \mapsto \widehat{M}_{n,u,I}(\bm v_I,2)/\widehat{P}_{n,u,I}$ is a positive semi-definite quadratic
form. Thus, the evaluation of $\widehat{M}_{n,u,I}(\bm v_I,2)/\widehat{P}_{n,u,I}$ for a finite number of vectors $\bm v_I$ is sufficient to reconstruct the
whole function $\{\widehat{M}_{n,u,I}(\bm v_I,2)/\widehat{P}_{n,u,I}: \, \bm v_I \in \partial B_1^+(\bm 0)\}$ and, thus, to calculate its minimizer. 

Even though $f^{(I)}$ is conditionally positive definite, i.e.\ the theoretical minimizer $\bm v^*_I$
is unique,  definiteness of the
empirical counterpart $\bm v_I \mapsto \widehat{M}_{n,u,I}(\bm v_I,2)/\widehat{P}_{n,u,I}$ is not
guaranteed for a finite sample, i.e.\ its minimizer $\widehat{\bm v}_{n,u,I}^*$ might no longer be unique.
In such a case, we just choose an arbitrary minimizer. Moreover, if $\ell_I(\bm X^*_{l}) \leq u$ then $\widehat{\bm v}_{n,u,I}^*$ is not well defined and can be an
arbitrary vector of $\partial B_1^+(\bm 0)$  for all $l \in \{1,\ldots,n\}$. 
We first show consistency of this estimator.

\begin{proposition} \label{prop:consist-opt-var}
Let $\bm X^*_l=(X^*_{l1},\ldots, X^*_{ld})^\top$, $l \in \NN$, be independent copies of a $d$-dimensional random  vector $\bm X^*$ satisfying the assumptions of Prop.~\ref{prop:uniqueness} for all non-empty set $I \subset \{1,\ldots,d\}$. Furthermore, let $\{u_n\}_{n \in \NN}$ be a sequence such that $u_n \to \infty$ and $n / a^*(u_n) \to \infty$ as $n \to \infty$ and assume that Eq.~\eqref{eq:bias-negligible} holds for $p=0,2$ and $\delta=0$.
  Then, $$\widehat{\bm v}_{n,u_n,I}^* \stackrel{p}{\longrightarrow} \bm v^*_I$$ for any sequence of minimizers $\widehat{\bm v}_{n,u_n,I}^*$ and any $I \subset \{1,\ldots,d\}$.\\
  Furthermore, 
  $\widehat{M}_{n,u_n,I}(\widehat{\bm v}_{n,u_n,I}^*,2) \to_p f^{(I)}(\bm v^*_I)$.
\end{proposition}
\begin{proof}
	
  We observe that,  non-empty $I \subset \{1,\ldots,d\}$, we have
  $f^{(I)}(\bm v_I) = c(\bm v_I, \bm 1_I, 2)$ for all $\bm v_I \in \partial B_1^+(\bm 0)$ and, thus, Cor.~\ref{coro:fclt-ratio} yields  
  \begin{equation} \label{eq:unif-conv}
  \sup\nolimits_{\bm v_I \in \partial B_1^+(\bm 0)} |\widehat{M}_{n,u_n,I}(\bm v_I,2) - f^{(I)}(\bm v_I)| \stackrel{p}{\longrightarrow} 0.
  \end{equation}
  As the continuous function $f^{(I)}$ attains its unique minimum on
  $\partial B_1^+(\bm 0)$ at $\bm v^*_I$ (see Prop.~\ref{prop:uniqueness}), for every $\varepsilon > 0$, there
  exists some $\eta = \eta(\varepsilon) > 0$ such that
  $ f^{(I)}(\bm v) - f^{(I)}(\bm v^*_I) > \eta $ for all $\bm v$ with 
  $\|\bm v - \bm v^*\| > \varepsilon$. Consequently, if $\|\widehat{\bm v}_{n,u_n,I}^* - \bm v^*\| > \varepsilon$ then we have that $ f^{(I)}(\widehat{\bm v}_{n,u_n,I}^*) - f^{(I)}(\bm v^*_I) > \eta $, which further implies 
  $$
  f^{(I)}(\widehat{\bm v}_{n,u_n,I}^*) - \widehat{M}_{n,u_n,I}(\widehat{\bm v}_{n,u_n,I}^*,2) +\widehat{M}_{n,u_n,I}(\bm v^*_I,2)- f^{(I)}(\bm v^*_I) > \eta 
  $$
  as $\widehat M_{n,u_n,I}(\widehat{\bm v}_{n,u_n,I}^*,2) \leq \widehat M_{n,u_n,I}(\bm v^*_I,2)$ by definition of $\widehat{\bm v}_{n,u_n,I}^*$. In particular, $$\sup_{\bm v \in \partial  B_1^+(\bm 0)} |\widehat M_{n,u_n,I}(\bm v,2) - f^{(I)}(\bm v)| > \eta/ 2.$$ Noticing that $\widehat{\bm v}_{n,u_n,I}^*$ is well defined only if $\bigvee\nolimits_{1\le l\le n } \ell_I(\bm X^*_l) > u_n$, we obtain
  $$ \PP(\|\widehat{\bm v}_{n,u_n,I}^* - \bm v_I^*\| > \varepsilon)
    \leq \PP\left( \ell_I(\bm X^*) \leq u_n\right)^n +     
    \PP\Big(\sup_{\bm v_I \in \partial B_1^+(\bm 0)} |\widehat M_{n,u_n}(\bm v_I,2) - f^{(I)}(\bm v_I)| > \frac{\eta} 2\Big),$$
  which tends to $0$ as $n \to \infty$ by Eq.~\eqref{eq:unif-conv}. 
  
  The statement that $$\widehat{M}_{n,u_n,I}(\widehat{\bm v}_{n,u_n,I}^*,2) \to_p f^{(I)}(\bm v^*_I)$$ follows from the facts that 
  $ f^{(I)}(\widehat{\bm v}_{n,u_n,I}^*) - \widehat M_{n,u_n,I}(\widehat{\bm v}_{n,u_n,I}^*,2) \to_p 0$ by Eq.~\eqref{eq:unif-conv} and 
  $f^{(I)}(\widehat{\bm v}_{n,u_n,I}^*) - f^{(I)}(\bm v_I^*) \to_p 0$
  due to $\widehat{\bm v}_{n,u_n,I}^* \to_p \bm v^*_I$ and the continuity of $f^{(I)}$.

  The second part can be shown analogously using the results from Thm.~\ref{thm:fclt-ratio-rank} instead of Cor.~\ref{coro:plugin-clt}.
\end{proof}

Now, we are ready establish the desired limit theorem for the plug-in
estimator, the moment-based estimator for the known margin case,
$$ \frac{1}{\widehat \tau_{I,n}^{MK}} = \frac{\widehat M_{n,u_n,I}(\widehat{\bm v}^*_{n,u_n,I})}{\widehat{P}_{n,u_n,I}}, $$
 by combining the results of Prop.~\ref{prop:consist-opt-var} with Cor.~\ref{coro:fclt-ratio}.

\begin{corollary} \label{coro:plugin-clt}
Let $\bm X^*_l=(X^*_{l1},\ldots, X^*_{ld})^\top$, $l \in \NN$, be independent copies of a $d$-dimensional random  vector $\bm X^*$ satisfying the assumptions of Prop.~\ref{prop:uniqueness} for all non-empty $I \subset \{1,\ldots,d\}$.   
  Furthermore, let $\{u_n\}_{n \in \NN}$ be a sequence such that $u_n \to \infty$ and $n / a^*(u_n) \to \infty$ as $n \to \infty$ and assume that Eq.~\eqref{eq:bias-negligible} holds for $p=0,1,2$ and $\delta=0$.	
  Then, for all non-empty $I \subset \{1,\ldots,d\}$,
	$$\sqrt{\frac{n}{a^*(u_n)}} 
  \left( \frac{1}{\widehat \tau_{I,n}^{MK}} - \frac{1}{\tau_I}  \right)
  \stackrel{d}{\longrightarrow} 
  \mathcal{N}\left(0, \tau_I^{-1} \Var({\bm v^*_I}^\top \bm \Theta^I)\right).$$
\end{corollary}  
\medskip
 
 \subsection{Unknown margins} \label{sec:random}
 
 Similarly, in case that the marginal distributions of $X_1,\ldots,X_d$ are unknown, i.e.\ Eq.~\eqref{eq:a-normal} is no longer satisfied and the indices $\alpha_1.\ldots,\alpha_d$ of regular variation may be different from $1$, we need to replace the deterministic threshold both in the benchmark estimator and in the moment-based estimator by a random one and include Hill type estimators as described in Subsection \ref{subsec:random-thresh}.
 
\subsubsection{The benchmark and the moment-based estimator}\label{sec:unknbench}

In case of the benchmark defined in Eq.~\eqref{eq:def-benchmark-estimator}, the transformations above lead to the estimator
$$ \frac{\widetilde{P}^{(conv)}_{n,k,I}(\bm v_I)}{\widetilde{P}_{n,k,I}}
  = \frac{n^{-1} \sum_{l=1}^n \mathds 1\{ \bm v_I^\top (\bm X_l/ \bm X_{k:n})^{\widehat \alpha} > 1\}}{n^{-1} \sum_{l=1}^n \mathds 1\{ \ell_I(\bm X_l / \bm X_{k:n}) > 1\}},
  \qquad \bm v_I \in \partial B_1^+(\bm 0). $$
Of particular interest is the case that $\bm v_I$ is a standard basis vector, i.e.\ $\bm v_I = \bm 1_{\{i\}}$ for some $i \in I$, since then the numerator satisfies $\widetilde{P}^{(conv)}_{n,k,I}(\bm 1_{\{i\}}) = k/n$ a.s., i.e.\ it possesses variance zero. Thus, the estimator simplifies to
$$ \frac{1}{\frac{n}{k} \widetilde{P}_{n,k,I}} = \frac 1 {\widehat L_{k,n}(\bm 1_I)} $$
where $\widehat L_{k_n,n}$ is the well-known non-parametric estimator for the stable-tail dependence function \citep[cf.][]{drees1998best}.  Provided that the bias is negligible and that the function 
$\bm s \mapsto \EE[\|\bm s \circ \bm \Theta^I\|]$ is differentiable at $\bm s = \bm 1_I$,
this estimator is asymptotically normal:
$$ \sqrt{k_n}\left( \frac 1 {\widehat L_{k_n,n}(\bm 1_I)} - \frac 1 {\tau_I}\right)
   \stackrel{d}{\longrightarrow} \mathcal{N}\left(0, \frac{\sigma_L^2}{\tau_I^4}\right) $$
where the asymptotic variance of the stable dependence function estimator $\widehat L_{k_n,n}(\bm 1_I)$ is provided in Chapter 7.4 of \cite{dehaan-ferreira-2006} by the expression
$$ \sigma_L^2 = \tau_I^3 \sum_{i \in I} \sum_{j \in I} \frac{\partial}{\partial s_i} \EE[\|\bm s \circ \bm \Theta^I\|] \Big|_{\bm s = \bm 1_I} \frac{\partial}{\partial s_j} \EE[\|\bm s \circ \bm \Theta^I\|] \Big|_{\bm s = \bm 1_I} \EE[\Theta^I_i \wedge \Theta^I_j] - \tau_I.$$
Using that $\EE[\Theta^I_i \wedge \Theta^I_j] \leq \EE(\Theta_i^*) = 1/\tau_I$ for all $i \in I$ and $\sum_{i \in I} \frac{\partial}{\partial s_i} \EE[\|\bm s \circ \bm \Theta^I\|] \big|_{\bm s = \bm 1_I} = 1$ by Euler's homogeneous function theorem due to the positive $1$-homogeneity of the function $\bm s \mapsto \EE[\|\bm s \circ \bm \Theta^I\|]$, we obtain that
$$ \frac{\sigma_L^2 }{\tau_I^4} \leq  \frac{\tau_I^2  - \tau_I}{\tau_I^4}
   = \frac{\tau_I - 1}{\tau_I^3}, $$
i.e.~the asymptotic variance of $1 / \widehat L_{k_n,n}(\bm 1_I)$ is always smaller than the asymptotic variance of the benchmark estimators for known margins, $ 1 / \widehat{\tau}_{I,n}(\bm v_I)$. In particular, the estimator becomes more efficient by using order statistics even if the marginal distributions are known.
\medskip
 
For the moment-based estimators, as already discussed in Subsection \ref{subsec:random-thresh},
the marginal transformations yield estimators of the type
$\widetilde{M}_{n,k,I}(\bm v_I,1) / \widetilde{P}_{n,k,I}$, $\bm v_I \in \partial B_1^+(\bm 0)$. Provided that the assumptions of Thm.~\ref{thm:fclt-ratio-rank}
hold, we obtain the asymptotic normality
 $$ \sqrt{k_n} \left( \frac{\widetilde{M}_{n,k_n,I}(\bm v_I,1)}
{\widetilde{P}_{n,k_n,I}} - \frac 1  {\tau_I} \right)
\stackrel{d}{\longrightarrow} 
\mathcal{N}(0, \widetilde V(\bm v_I)), \qquad n\to \infty\,,$$
with $\widetilde V(\bm v_I, 1)$ is the variance of the limiting process given in Thm.~\ref{thm:fclt-ratio-rank}. 

\subsubsection{Variance minimization}

Analogously to Section \ref{sec:deterministic}, we aim at choosing the estimator that provides the minimal asymptotic variance, i.e.\ use the weight vector $\widetilde{\bm v}_I \in \partial B_1^+(\bm 0)$ that minimizes $\widetilde V(\cdot)$, respectively. To this end, we first exploit the results given in Rem.~\ref{rem:asymp-variance-orderstats}; By linearity of this limiting process $\widetilde V(\bm v_I)$ is  a quadratic form $\bm v_I^T \widetilde \Sigma_I \bm v_I$ for the symmetric positive semidefinite covariance matrix
$$
\widetilde {\bm V}_I = \Var\Big(\Big( \widetilde G(\bm 1_{j},\bm 1_I,\bm 1,1) -  \sum_{i \in I}[ c_{s_i}(\bm 1_{j}, \bm 1_I, \bm 1, 1)  G^0(\bm 1_{\{i\}}) + c_{\beta_i}(\bm 1_{j}, \bm 1_I, \bm 1, 1) \bm \alpha_i\widetilde H_i]\Big)_{j\in I}\Big)\,.
$$
We obtain an analogous result to Prop.~\ref{prop:uniqueness} and the analogous proof is omitted.

\begin{proposition} \label{prop:uniqueness-random}
	Let $\bm \Theta$ be a spectral vector. Assume that all the partial derivatives
	\begin{align*}
	c_{s_i}(\bm v, \bm s, \beta, 1) ={}& \frac{\partial}{\partial s_i}c(\bm v, \bm s, \beta, 1),
	\qquad i \in \{1,\ldots,d\},\\
	c_{\beta_i}(\bm v, \bm s, \beta, 1) ={}& \frac{\partial}{\partial \beta}c(\bm v, \bm s, \beta, 1),
	\end{align*}
	exist and are continuous on $\partial B_1^+(\bm 0) \times A_{\delta}'$ and that
	all the partial derivatives of the function
	$ \bm s \mapsto \EE(\| \bm s \circ \bm \Theta^I\|) $
	exist and are continuous in $A_{\delta}$.\\
	Then, for all non-empty $I \subset \{1,\ldots,d\}$:
	If the matrix $\widetilde{\bm V}_I$
	is conditionally  positive definite, then the function $\bm v_I \mapsto \widetilde V(\bm v_I)$
	attains its minimum on $\partial B_1^+(\bm 0)$ at some unique  $\widetilde{\bm v}_I \in \partial B_1^+(\bm 0)$.             
\end{proposition}

\subsubsection{Asymptotic normality of the plug-in estimator}

Note that, in contrast to $f^{(I)}$, there is no direct empirical counterpart of $\widetilde{V}(\cdot)$. To enable analytical minimization of that function, we consider an estimator that retains the quadratic form of $\widetilde{V}(\cdot )$ by  estimating consistently the symmetric conditionally definite matrix $\widetilde {\bm V}_I$.
We consider a quadratic form $\widetilde V_{n,I}(\bm v_I) = \bm v_I^\top \widetilde{\bm V}_{n,I} \bm v_I$ for
a sequence of random symmetric matrices $(\widetilde{\bm V}_{n,I} )_{n \in \NN}$. Then we define
\begin{align*}
    \widetilde{\bm v}_{n,I} = \argmin_{\bm v_I \in \partial B_1^+(\bm 0)} \widetilde{V}_{n,I}(\bm v_I).
\end{align*}
Here, without further conditions on the underlying estimators, no kind of (semi-)definiteness
is guaranteed for the matrices $\widetilde{\bm V}_{n,I}$. However, being a continuous function, $\bm v_i 
\mapsto \widetilde V_{n,I}(\bm v_I)$ attains its minimum on the compact set $\partial B_1^+(\bm 0)$ at
least once. Here, an arbitrary minimizer might be chosen for $ \widetilde{\bm v}_{n,I}$ provided that it 
is not unique.

Analogously to Prop.~\ref{prop:consist-opt-var}, consistency of the estimator can be shown.

\begin{proposition} \label{prop:consist-opt-var-random}
 Let $\bm \Theta$ be a spectral vector satisfying the assumptions of
  Prop.~\ref{prop:uniqueness-random}. 
 Then, for all non-empty $I \subset \{1,\ldots,d\}$, it holds:
  If $\widetilde V_{I,n} \to_p \widetilde V_I$,
  then we have that $\widetilde{\bm v}_{n,I} \to_p \widetilde{\bm v}_I$ for any sequence of minimizers $\widetilde{\bm v}_{n,I}$. Furthermore, $\widetilde V_{n,I}(\widetilde{\bm v}_{n,I}) \to_p
  \widetilde V(\widetilde{\bm v}_I)$.
\end{proposition}
\begin{proof}
   Recall that, by construction, $\widetilde V_{n,I}(\bm v_I) = \bm v_I^\top \widetilde{\bm V}_{n,I} \bm v_I$
while
   $\widetilde V(\bm v_I) = \bm v_I^\top \widetilde{\bm V}_I \bm v_I$ for a conditionally positive definite matrix $\widetilde{\bm V}_I$. The assumptions on the consistency of all the estimators involved in $\widetilde{\bm V}_{n,I}$ imply that $\|\widetilde{\bm V}_{n,I} - \widetilde{\bm V}_I\|_\infty \to_p 0$. Consequently,
   $$ \sup_{\bm v_I \in \partial B_1^+(\bm 0)} | \widetilde V_{n,I}(\bm v_I) - \widetilde{ V}(\bm v_I) |
      \leq d^2 \|\widetilde{\bm V}_{n,I} - \widetilde{\bm V}_I\|_\infty \to_p 0, $$
   which is the analogue to Eq.~\eqref{eq:unif-conv}. The remainder of the proof runs analogously to the proof of Prop.~\ref{prop:consist-opt-var}.
\end{proof}

Combining the results of Prop.~\ref{prop:consist-opt-var-random} and Thm.~\ref{thm:fclt-ratio-rank}, we obtain the desired limit theorem for the plug-in estimator,
the moment based estimator for the case of unknown margins
$$ \frac 1 {\widehat \tau_{I,n}^{MU}} = \frac{\widetilde{M}_{n,k,I}(\widetilde{\bm v}_{n,I},1)}{\widetilde{P}_{n,k,I}}. $$

\begin{corollary} \label{coro:plugin-clt-final}
Let $\bm X_l$, $l \in \NN$, be independent copies of a regularly varying $[0,\infty)^d$-valued random vector $\bm X$  satisfying the assumptions of Prop.~\ref{prop:consist-opt-var-random} for all non-empty $I \subset \{1,\ldots,d\}$ and the assumptions of Thm.~\ref{thm:fclt-ratio-rank} for $K = \{1\}$. Then, the plug-in estimator satisfies
  \begin{align*}
	 & \sqrt{k} \left( \frac 1 {\widehat \tau_{I,n}^{MU}}  - \frac{1}{\tau_I}\right) \stackrel{d}{\longrightarrow} \mathcal{N}(0, \widetilde{V}(\widetilde{\bm v}_I))\,,\qquad n\to\infty\,.
	\end{align*}
\end{corollary}
In order to get a consistent estimator $\widetilde{\bm V}_{n,I}$ of  $\widetilde{\bm V}_I$ we use the $m$-out-of-$n$ bootstrap procedure. From Thm.~\ref{thm:fclt-ratio-rank} it is very natural to consider the covariance matrix of the vector
$$
\left(\sqrt{k_n} \left( \frac{\widetilde{M}_{n,k_n,I}(\bm 1_{\{i\}},1)}{\widetilde{P}_{n,k_n,I}} - c(\bm 1_{\{i\}}, \bm 1_I, \bm1, 1)\right)\right)_{1\le i\le d}
$$
as an approximation of $\widetilde{\bm V}_I$. In order to estimate each entry of this covariance matrix, we consider $N_m$ subsamples $(\bm X_1^{(j)},\ldots,\bm X_m^{(j)})$, $1\le j\le N_m$, of the original sample $(\bm X_1,\ldots,\bm X_n)$ drawn uniformly without replacement. We denote $\widetilde{M}_{m,k_m,I}^{(j)}$ and $\widetilde{P}^{(j)}_{m,k_m,I}$ the statistics of interest using the subsample $(\bm X_1^{(j)},\ldots,\bm X_m^{(j)})$. Then we derive a bootstrapped distribution for any pair $i$, $i'\in I$ such as
\begin{align*}
L_{n,i,i'}^{(m)}(x)=\dfrac1{N_m}\sum_{j=1}^{N_m}  & \mathds 1\left\{\sqrt{k_m} \left( \frac{\widetilde{M}_{m,k_m,I}^{(j)}
	(\bm 1_{\{i\}},1)+\widetilde{M}^{(j)}_{m,k_m,I}(\bm 1_{\{i'\}},1)}{2\widetilde{P}^{(j)}_{m,k_m,I}} \right. \right. \\
& \qquad \qquad \qquad \qquad \left. \left. - \frac{\widetilde{M}_{n,k_n,I}(\bm 1_{\{i\}},1)+\widetilde{M}_{n,k_n,I}(\bm 1_{\{i'\}},1)}{2\widetilde{P}_{n,k_n,I}}\right)\le x\right\}\,.
\end{align*}

We then calculate the first and third quartiles $q_{1,i,i'}^{(m)}$ and $q_{3,i,i'}^{(m)}$ of the bootstrapped distribution and consider the standardized interquartile statistics 
$$
\widetilde \sigma_{i,i'}^{(m)}=\dfrac{q_{3,i,i'}^{(m)}-q_{1,i,i'}^{(m)}}{q_3-q_1}
$$ 
where $q_1$ and $q_3$ are the first and third quartiles of a standard Gaussian random variable. Finally, we define the entries of our estimator $\widetilde{\bm V}_{n,I}^{(m)}$ such as 
$$
\widetilde{\bm V}_{n,I,i,i'}^{(m)}=\begin{cases}
{~\widetilde \sigma_{i,i}^{(m)}}^2 & \text{if}\qquad i=i'\,,\\
{2\widetilde \sigma_{i,i'}^{(m)}}^2-\big({~\widetilde \sigma_{i,i}^{(m)}}^2+{~\widetilde \sigma_{i,i}^{(m)}}^2\big)/2& \text{else}\,.
\end{cases}
$$
Notice that our estimation procedure also depends on the number of subsamples $N_m$ that we do not fix at ${n}\choose{m}$ for computational efficiency. We suppress this dependence for clarity.
\begin{proposition}
Let $\bm X_l$, $l \in \NN$, be independent copies of a regularly varying $[0,\infty)^d$-valued random vector $\bm X$  satisfying the assumptions of Cor.~\ref{coro:plugin-clt-final}. Then choosing $m_n$ such that $m_n\to \infty$, $m_n/n \to 0$ and $k_n/k_{m_n}\to \infty$ and $N_m\to \infty$ we have 
$\widetilde V_{I,n}^{(m)} \to_p \widetilde V_I$ as $n\to\infty$ for all $I\subset\{1,\ldots,d\}$.
\end{proposition}
\begin{proof}
We apply Theorem 2.1 (ii) of \cite{politis1994large} to the statistics
$$
\left(\sqrt{k_n} \left( \frac{\widetilde{M}_{n,k_n,I}(\bm 1_{\{i\}},1)+\widetilde{M}_{n,k_n,I}(\bm 1_{\{i'\}},1)}{2\widetilde{P}_{n,k_n,I}} -  \frac 1 2(c(\bm 1_{\{i\}}, \bm 1_I, \bm1, 1)+c(\bm 1_{\{i'\}}, \bm 1_I, \bm1, 1))\right)\right)$$
remarking that the limiting centered Gaussian distribution with variance
\begin{equation}\label{eq:expressvar}
\widetilde \sigma_{i,i'}^2=\begin{cases}\widetilde V_{I,i,i'}&\text{if}\qquad  i=i'\,,\\
\widetilde V_{I,i,i'}/2+\widetilde V_{I,i,i'}/4+\widetilde V_{I,i,i'}/4& \text{else}\,,
\end{cases}
\end{equation}
is continuous for all $ i,i'\in I$. Then $L_{n,i,i'}^{(m)}(x)$ converges in probability to the limiting distributions, $L_{i,i'}(x)$ say, uniformly for all $x\in \RR$. The uniform convergence ensures the convergence in probability of the quantiles $q_{j,i,i'}^{(m)}$ towards  $\widetilde \sigma_{i,i'}q_j$ for all $i,i'\in I$ and $j=1,3$. Thus $
\widetilde \sigma_{i,i'}^{(m)}$ converges in probability to  $
\widetilde \sigma_{i,i'}$ and $\widetilde V_{I,n}^{(m)}$ to $\widetilde V_I$ as well because of the expression \eqref{eq:expressvar} of the limiting variances $
\widetilde \sigma_{i,i'}^2$, $i,i'\in I$.
\end{proof}

Note that, by construction, the estimated matrices $\widetilde V_{I,n}^{(m)}$ might be indefinite, which impedes the minimization of the quadratic form $\bm v_I \mapsto \widetilde V_{I,n}^{(m)}(\bm v_I)$ on $\partial B_1^+(\bm 0)$. In practice, it turns out to define
$\widetilde \sigma_{i,i'}^{(m)}$ as the sample variance of the bootstrap sample. Provided that the same subsamples $(\bm X_1^{(j)},\ldots,\bm X_m^{(j)})$, $1 \leq j \leq N_m$, are used for each pair $(i,i') \in \{1,\ldots,d\}^2$, this guarantees that the corresponding matrix $\widetilde {\bm V}_{n,I}^{(m)}$ is positive semi-definite. Numerical experiments illustrate that the resulting estimator performs as expected (see Section \ref{sec:simu}) in practice.  However, a rigorous proof of weak consistency of $\widetilde {\bm V}_{n,I}^{(m)}$ would require convergence of second order moments of the bootstrapped distribution to the corresponding moments of the Gaussian limits, which we cannot show without imposing strong additional assumptions. 	

\section{Numerical illustrations} \label{sec:illustr}
\subsection{Simulated example I: H\"usler-Reiss distribution} \label{sec:simu}

As a first simulated example, we consider  a popular class of multivariate extreme value models, the H\"usler-Reiss distribution \citep{huesler-reiss-1989} whose cumulative distribution is given by
$$ \PP(X_1 \leq x_1, \ldots, X_d \leq x_d) = \exp\left(- \EE\left[ \max_{i=1,\ldots,d} \frac{1}{x_i} e^{W_i - \Var(W_i/2)}\right]\right), \quad x_1,\ldots,x_d > 0, $$
where $(W_1,\ldots,W_d)$ is a centered Gaussian random vector. By definition, the 
random vector $\bm X$ is max-stable with unit Fr\'echet margins. Note that its distribution is uniquely determined by the variogram matrix
$$ \Gamma = (\Gamma_{ij})_{1 \leq i,j \leq d}, \quad \Gamma_{ij} = \Var(W_i-W_j). $$

\subsubsection{Bivariate Case}

In the bivariate case, the extremal coefficient of the vector $(X_1, X_2)$ is of the form
$$ \tau = \tau_{\{1,2\}} = 2 \Phi(\sqrt{\Gamma_{12}}/2) \in [1,2], $$
i.e.\ the model can interpolate between full dependence $(\Gamma_{12}=0)$ and
asymptotic independence $(\Gamma_{12} \to \infty)$.

We compare the four types of estimators for $\tau$, namely the benchmark estimators and 
the moment-based estimators, both with known and unknown margins as described in Sections \ref{sec:deterministic} and \ref{sec:random}, respectively.
More precisely, in the case of known margins where we can work with standardized observations $\bm X^*_1, \bm X^*_2, \ldots, \bm X^*_n$ and use
\begin{itemize}
	\item the benchmark estimator 
	$$\frac 1 {\widehat \tau_{I,n}^{BK} (\bm v)} = \frac{\widehat{P}_{n,u,\{1,2\}}^{\mathrm{(conv)}}(\bm v)}
	{P_{n,u,\{1,2\}}};$$
	\item the moment-based plug-in estimator
	$$  \frac 1 {\widehat \tau_{I,n}^{MK}} =
	\frac{\widehat M_{n,u_n,\{1,2\}}(\widehat{\bm v}^*_{n,u_n,\{1,2\}})}
	{\widehat{P}_{n,u_n,\{1,2\}}} $$
	based on the estimated ``optimal'' weights $\widehat{\bm v}^*_{n,u_n,\{1,2\}};$
	\item and, for comparison, the perfect moment-based estimator
	$$  \frac 1 {\widehat \tau_{I,n, \mathrm{opt}}^{MK}} =
	\frac{\widehat M_{n,u_n,\{1,2\}}(\bm v^*)}
	{\widehat{P}_{n,u_n,\{1,2\}}} $$
	with theoretically optimal weight $\bm v^*$.
\end{itemize}
Note that, for simulated data from the bivariate H\"usler--Reiss model, no standardization is required as the (unit Fr\'echet) marginal distributions are already identical and regularly varying with index $\alpha=1$. From Example \ref{ex:minvar-known-margins}, it can be seen that the variance minimizing linear combination is unique whenever $0 < \Gamma_{12} < \infty$ and is given by 
$\bm v^* = \left(1/2, 1/2\right)$.
As discussed in Section \ref{sec:deterministic}, the asymptotic distribution of the benchmark estimator does not depend on the choice 
of $\bm v_I \in \partial B_1^+(\bm 0)$. Here we choose 
$\widehat \tau_{I,n}^{BK} = \widehat \tau_{I,n}^{BK} (\bm v^*)$
as benchmark.
\medskip

Even though the margins are known, it might sometimes be beneficial to work with the rank transformations designed for the case of unknown margins as discussed in Section
\ref{sec:random}. Thus, we also consider 
\begin{itemize}      
	\item the ``benchmark'' estimator
	$$ \frac 1 {\widehat \tau_{I,n}^{BU}} = \frac{1}{\frac{n}{k} \widetilde{P}_{n,k,\{1,2\}}} 
	= \frac 1 {\widehat L_{k,n}(\bm 1)} ;$$
	\item the moment-based plug-in estimator
	$$ \frac 1 {\widehat \tau_{I,n}^{MU}} = \frac{\widetilde{M}_{n,k,\{1,2\}}(\widetilde{\bm v}_{n,\{1,2\}},1)}
	{\widetilde{P}_{n,k,\{1,2\}}} $$
	based on estimated ``optimal'' weights $\widetilde{\bm v}_{n,\{1,2\}};$
	\item and, again, the perfect moment-based plug-in estimator
	$$ \frac 1 {\widehat \tau_{I,n, \mathrm{opt}}^{MU}} = \frac{\widetilde{M}_{n,k,\{1,2\}}(\widetilde{\bm v},1)}
	{\widetilde{P}_{n,k,\{1,2\}}} $$
	with theoretically optimal weight $\widetilde{\bm v}.$
\end{itemize} 
From symmetry arguments, it can be easily seen that the ``optimal''
linear combination (if unique) is necessarily given by
$\widetilde{\bm v} = \left(1/2, 1/2\right)$.

We perform a simulation study to compare the estimators in a setting with finite sample sizes. Here, we consider three different scenarios:
\begin{itemize}
	\item Scenario 1: $\Gamma_{12}=0.1$ (strong dependence; $\tau \approx 1.13$)
	\item Scenario 2: $\Gamma_{12}=2$ (medium dependence; $\tau \approx 1.52$)
	\item Scenario 3: $\Gamma_{12}=10$ (weak dependence; $\tau \approx 1.89$)
\end{itemize}	

In each of these scenarios, we consider $n=5\,000$ independent observations of $\bm X = (X_1,X_2)$ and apply the estimators with a threshold $u_n = \Phi_1^{-1}(0.98)$ and $k_n= 0.02 \cdot n$, respectively. For the bootstrap procedure that is needed to estimate the optimal variance, we choose the subsample size $m_n=0.25 \cdot n$ and $N_m=n$.	
We determine the biases and standard deviations by $500$ repetitions.

The results are displayed in Table \ref{tab:simu-res}. First, it can be seen that, at least in the cases of strong and moderate dependence, all the biases are much smaller than the standard deviations and are therefore negligible. Additional
calculations of the asymptotic variances $\mathrm{AVar}^{BK}$, $\mathrm{AVar}^{MK}_{\mathrm{opt}}$, $\mathrm{AVar}^{BU}$ and
$\mathrm{AVar}^{MU}_{\mathrm{opt}}$ according to the theoretical results given in Sections \ref{sec:deterministic} and \ref{sec:random}, respectively, show that the empirical standard deviations are very close to their theoretical counterparts even though the bootstrap procedure tends to slightly underestimate the true asymptotic standard deviation of the moment-based estimator in case of unknown margins (by roughly $10\,\%$). In particular, we obtain the expected ordering of the four methods if the optimal weights for the moment-based estimators are chosen:
$$ \mathrm{AVar}^{MK}_{\mathrm{opt}} < \mathrm{AVar}^{MU}_{\mathrm{opt}} < \mathrm{AVar}^{BU} < \mathrm{AVar}^{BK}.$$

This ordering is kept if we replace the theoretically optimal weights by the estimated ones; in our simulation study  the variances of the plug-in moment estimators are almost identical to the ones with optimal wights $v^*$ and $\widetilde v$, which is well in line with the theoretical results. However, it has to be noted that taking non-optimal weights might result in significantly larger variances -- in particular in the case of known margins.  This is demonstrated in Figure \ref{fig:avar_v} where the asymptotic variance is plotted as a function of the weight vector $\bm v$. Compared to the case of known margins, the relative changes of the asymptotic variance for the moment-based estimator for unknown margins turn out to be rather small.

\begin{table}
	\begin{center}
		$\Gamma_{12} =0.1$ ($\tau\approx 1.13$, strong dependence)\\[1mm]
		
		\begin{tabular}{|l||r|r|r|r|r|r|}
			\hline \phantom{$A^{X^X}_{Y_Y}$}
			& $1 / \widehat{\tau}_{I,n}^{BK}$ & $1 / \widehat{\tau}_{I,n}^{MK}$
			& $1 / \widehat{\tau}_{I,n,\mathrm{opt}}^{MK}$
			& $1 / \widehat{\tau}_{I,n}^{BU}$ & $1 / \widehat{\tau}_{I,n}^{MU}$
			& $1 / \widehat{\tau}_{I,n,\mathrm{opt}}^{MU}$ \\ \hline	
			bias                 & 0.003 & 0.001 & 0.000 & 0.019 & -0.001 & 0.001\\ \hline
			std.~deviation       & 0.030 & 0.007 & 0.007 & 0.019 & 0.012 & 0.012\\ \hline			
		\end{tabular}
		\medskip

		$\Gamma_{12}=2$ ($\tau\approx 1.52$, moderate dependence)\\[1mm]
				
		\begin{tabular}{|l||r|r|r|r|r|r|}
		\hline \phantom{$A^{X^X}_{Y_Y}$}
		& $1 / \widehat{\tau}_{I,n}^{BK}$ & $1 / \widehat{\tau}_{I,n}^{MK}$
		& $1 / \widehat{\tau}_{I,n,\mathrm{opt}}^{MK}$
		& $1 / \widehat{\tau}_{I,n}^{BU}$ & $1 / \widehat{\tau}_{I,n}^{MU}$
		& $1 / \widehat{\tau}_{I,n,\mathrm{opt}}^{MU}$ \\ \hline
		bias                 & 0.009 & 0.005 & 0.004 & 0.010 & 0.004 & 0.004\\ \hline
		std.~deviation       & 0.040 & 0.010 & 0.010 & 0.020 & 0.016 & 0.016\\ \hline			
		\end{tabular}
		\medskip
		
		$\Gamma_{12} =10$ ($\tau\approx 1.89$, weak dependence)\\[1mm]			
		
		\begin{tabular}{|l||r|r|r|r|r|r|}
		\hline \phantom{$A^{X^X}_{Y_Y}$}
		& $1 / \widehat{\tau}_{I,n}^{BK}$ & $1 / \widehat{\tau}_{I,n}^{MK}$
		& $1 / \widehat{\tau}_{I,n,\mathrm{opt}}^{MK}$
		& $1 / \widehat{\tau}_{I,n}^{BU}$ & $1 / \widehat{\tau}_{I,n}^{MU}$
		& $1 / \widehat{\tau}_{I,n,\mathrm{opt}}^{MU}$ \\ \hline
			bias      & 0.022 & 0.014 & 0.014 & 0.010 & 0.013 & 0.013 \\ \hline
			std.~deviation  & 0.038 & 0.006 & 0.006 & 0.009 & 0.007 & 0.007  \\ \hline			
		\end{tabular}
		
	\end{center} 
	\caption{Results for 500 simulations from the max-stable H\"usler--Reiss models in the three scenarios specified above with
		$n=5\,000$, $k=100$ and $u_n = \Phi^{-1}_1(1-k/n) = \Phi^{-1}_1(0.98)$.}
	\label{tab:simu-res}
\end{table}

\begin{figure}
	\centering \includegraphics[width=0.99\textwidth]{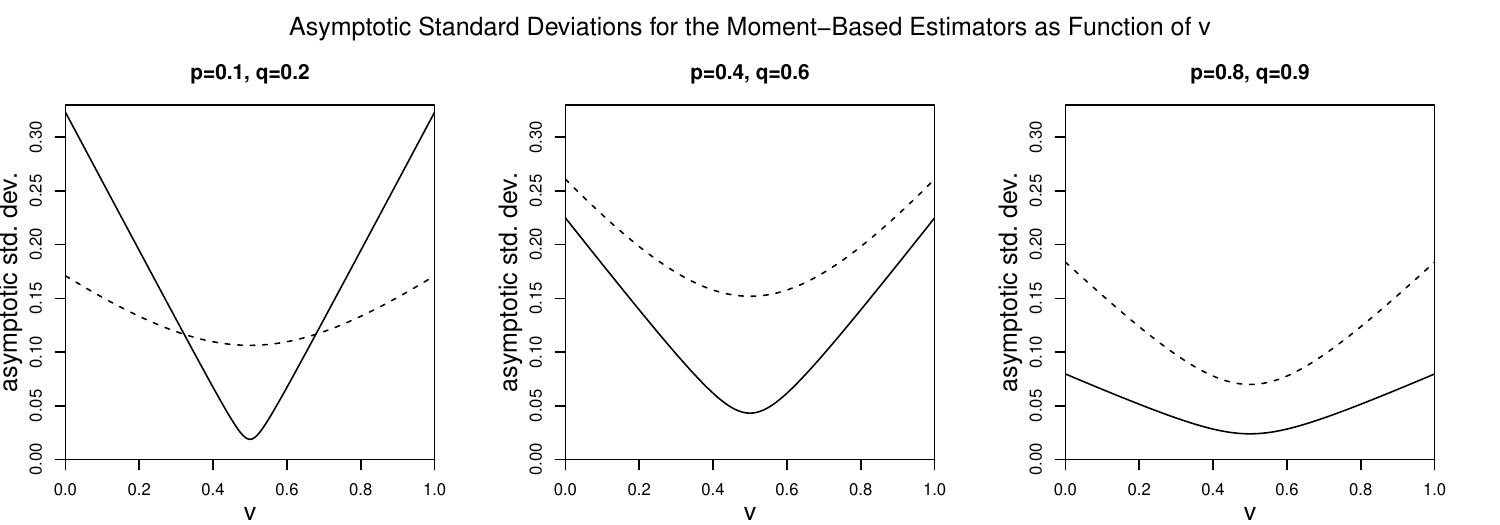}
	\caption{The asymptotic standard deviations $\sqrt{V((v,1-v))}$ (solid lines)
		and $\sqrt{\widetilde{V}((v,1-v))}$ (dashed line) of the moment-based estimators with known and unknown margins, respectively, for each of
		the three scenarios reflecting strong, moderate and weak dependence (from left to right).} \label{fig:avar_v}
\end{figure}
We further analyze the sensitivity of the novel moment-based estimators with respect to the choice of $k$. To this end, the results of the estimators for ten realizations from the max-stable H\"usler--Reiss in the scenario of moderate dependence described above are displayed as a function of $k$ are displayed in Figure \ref{fig:kstability}. Here, we choose $k \in \{50, 100, \ldots, 500\}$ corresponding to the upper $1\,\%, 2\,\%, \ldots, 10\%$ of the data. It can be seen that the results are rather stable with the estimator  $1 / \widehat{\tau}_{I,n}^{MU}$ being slightly more uncertain than  $1 / \widehat{\tau}_{I,n}^{MK}$, which is in line with the results from Table 	\ref{tab:simu-res}.
    \begin{figure}
    	\centering \includegraphics[width=0.99\textwidth]{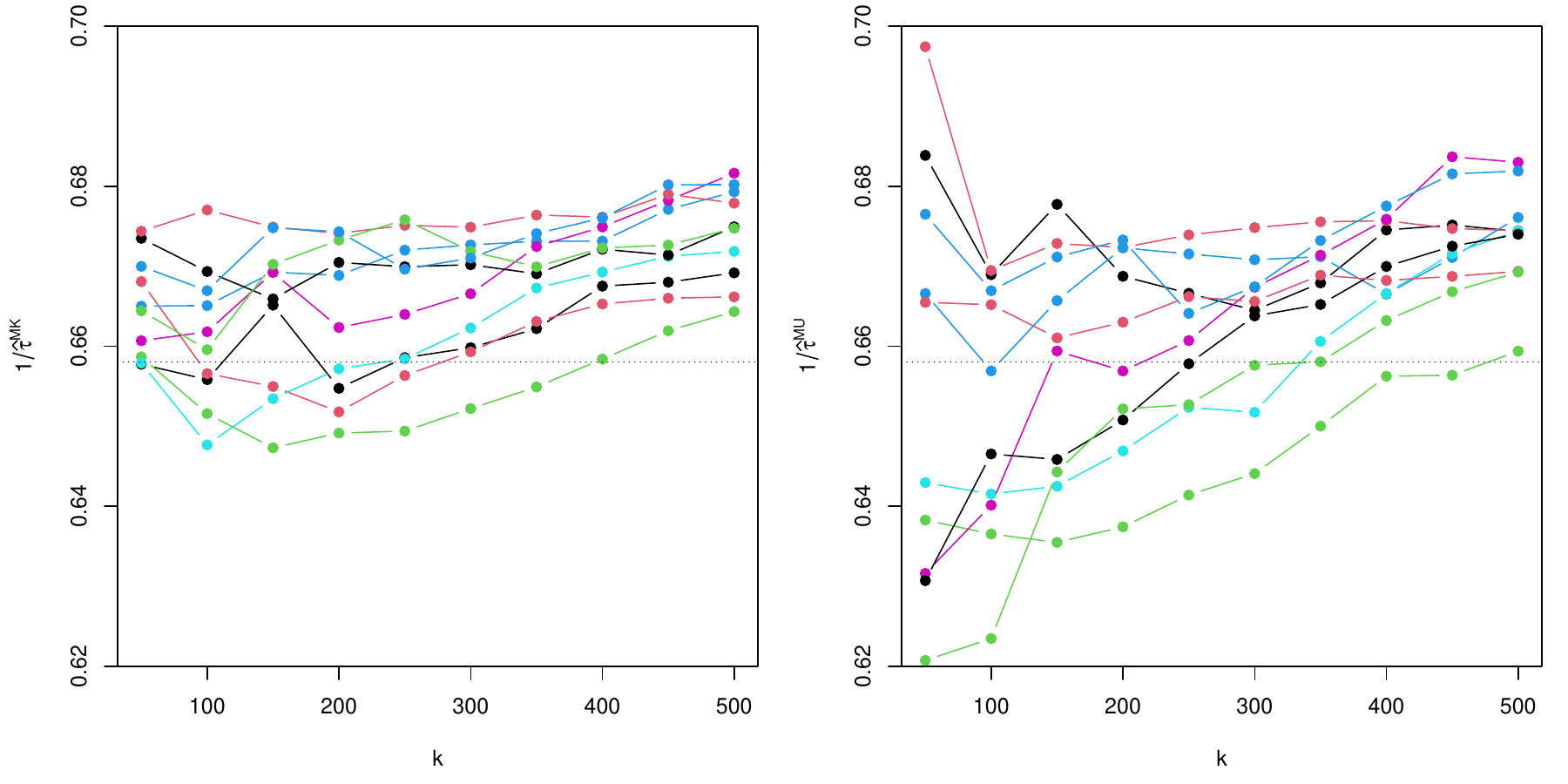}
    	\caption{The results of the moment-based estimators, $1 / \widehat{\tau}_{I,n}^{MK}$ (left) and $1 / \widehat{\tau}_{I,n}^{MU}$ (right), as a function of $k$ for ten realizations of the max-stable H\"usler--Reiss model with moderate dependence ($\Gamma_{12}=2$) and $n=5\,000$.} \label{fig:kstability}
    \end{figure}

	As mentioned at the beginning of the section, there is a one-to-one correspondence between $\tau = \tau_{\{1,2\}}$ and $\Gamma_{12}$ via 
	$\tau = 2 \Phi(\sqrt{\Gamma_{12}}/2)$ or, equivalently, $\Gamma_{12} = \left[ 2 \Phi^{-1}\left( \frac {\tau} 2\right)\right]^2$.
	Thus, we obtain a plugin estimator $ \widehat{\Gamma}_{12,n}^{\text{m}}$ from 
	$\widehat \tau_{I,n}^{\text{m}}$ via 
	$$  \widehat{\Gamma}_{12,n}^{\text{m}} = \left[ 2 \Phi^{-1}\left( \frac {\widehat{\tau}_{I,n}^{\text{m}}} 2\right)\right]^2, \qquad \text{m} \in \{\text{BK}, \text{MK}, \text{BU}, \text{MU}\}. $$
	By the delta method, the asymptotic variances are expected to be ordered in the same way as the variances for the estimators for $\tau$. To confirm this behaviour also in a finite sample setting, we apply the plugin estimators to the data from the simulation study above. The results are displayed in Table \ref{tab:simu-res-gamma}. Similarly to the study above, in the cases of strong and 
	moderate dependence, the bias is negligible and the standard deviations show the expected ordering. In the case of weak dependence, the biases of the of the estimators for $\Gamma_{12}$ are even more pronounced then the biases of the estimators for $1/\tau$. In particular, in some simulations, the benchmark estimator $\widehat{\tau}_{I,n}^{\text{MK}}$ even exceeds the theoretical maximum of $2$ -- consequently, the corresponding value for $\Gamma_{12}$ cannot be calculated, i.e., we obtain invalid estimates. For the other three estimators, the standard deviations again exhibit the expected ordering.

\begin{table}
	\begin{center}	
		\begin{tabular}{|cl||r|r|r|r|}
			\hline \phantom{$A^{X^X}_{Y_Y}$}
			& & $\widehat{\Gamma}_{12,n}^{BK}$ & $\widehat{\Gamma}_{12,n}^{MK}$
		      & $\widehat{\Gamma}_{12,n}^{BU}$ & $\widehat{\Gamma}_{12,n}^{MU}$
		    \\ \hline
		   \multirow{2}{*}{$\Gamma_{12}=0.1$} &	bias   & 0.007 & -0.001 & -0.016 & 0.004 \\ \cline{3-6}
			& std.~deviation  & 0.069 & 0.013 & 0.036 & 0.025  \\ \hline 	
			\multirow{2}{*}{$\Gamma_{12}=2$} &	bias   & -0.001 & -0.008 & -0.016 & -0.006 \\ \cline{3-6}
			& std.~deviation  & 0.889 & 0.207 & 0.381 & 0.323  \\ \hline 
			\multirow{2}{*}{$\Gamma_{12}=10$} &	bias   & NA & -2.123 & -1.457 & -1.983 \\ \cline{3-6}
			& std.~deviation  & NA & 0.760 & 1.413 & 0.956  \\ \hline	
		\end{tabular}
	\end{center} 
	\caption{Results for the plugin estimators for $\Gamma_{12}$ for 500 simulations from the max-stable H\"usler--Reiss models in the three scenarios specified above with
		$n=5\,000$, $k=100$ and $u_n = \Phi^{-1}_1(1-k/n) = \Phi^{-1}_1(0.98)$.}
	\label{tab:simu-res-gamma}
\end{table}

Here, it is important to note that neither the benchmark nor the moment-based estimator are restricted to the setting of limiting multivariate extreme value distributions, but allow for estimation of the extremal coefficients, or, equivalently, the estimation of the corresponding dependence parameters, for any distribution from the max-domain of attraction. Thus, we duplicate the setting of the simulation study above simulating from the H\"usler--Reiss--Pareto distributions from the max-domain of attraction instead. As the marginal distributions are no longer unit Fr\'echet, but only asymptotically equivalent to that distribution, we do not consider the estimators for known margins, but only their variants for the case of unknown margins. The results are displayed in Table \ref{tab:simu-res-gamma-mda}.
While the results for the cases of strongly and moderately dependent data are very similar to the max-stable setting, the biases in the weakly dependent case are much smaller now, coming with the price of slightly larger standard deviations. This phenomenon can be explained by the fact that the angular part of a H\"usler--Reiss--Pareto random vector exactly follows the spectral measure, while the spectral measure still provides an approximation in the max-stable setting. In any case, the newly developed moment-based estimators provide smaller variances than the corresponding benchmark estimators.

\begin{table}
	\begin{center}	
			\begin{tabular}{|cl||r|r|}
				\hline \phantom{$A^{X^X}_{Y_Y}$}
				& & $\widehat{\Gamma}_{12,n}^{BU}$ & $\widehat{\Gamma}_{12,n}^{MU}$
				\\ \hline
				\multirow{2}{*}{$\Gamma_{12}=0.1$} & bias   & -0.011 & 0.005 \\ \cline{3-4}
				& std.~deviation  & 0.037 & 0.024 \\ \hline 	
				\multirow{2}{*}{$\Gamma_{12}=2$} &	bias   &  -0.071 & 0.081 \\ \cline{3-4}
				& std.~deviation  & 0.383 & 0.333  \\ \hline 
				\multirow{2}{*}{$\Gamma_{12}=10$} &	bias   & -0.720 & 0.317  \\ \cline{3-4}
				& std.~deviation  & 1.580 & 1.378  \\ \hline	
			\end{tabular}
	\end{center} 
\caption{Results for the plugin estimators for $\Gamma_{12}$ for 500 simulations from the H\"usler--Reiss--Pareto models in the three scenarios specified above with
		$n=5\,000$, $k=100$ and $u_n = \Phi^{-1}_1(1-k/n) = \Phi^{-1}_1(0.98)$.}
	\label{tab:simu-res-gamma-mda}
\end{table}

\subsubsection{Multivariate case}

We repeat our study for the multivariate case. Here, we consider the case of a fully symmetric variogram matrix
$$ \Gamma =
  \begin{pmatrix}
   \phantom{.}0\phantom{.} & \phantom{.}1\phantom{.} & \dots & \dots & \phantom{.}1\phantom{.} \\
   1 & 0 & \ddots & & \vdots \\
   \vdots & \ddots & \ddots & \ddots & \vdots \\
   \vdots & &  \ddots & 0 & 1 \\
   1 & \dots & \dots & 1 & 0
\end{pmatrix} \in \RR^{d \times d}. $$
We consider three different scenarios with dimensions $d=3$, $d=5$ and $d=8$, respectively. For each of these scenarios, we compare the same six estimators as in the previous subsection. Again we can argue that, for symmetry reasons, the optimal weight vector, if unique, is necessarily given by
$\bm v^* = \widetilde{\bm v} = \left( 1/d, \ldots, 1/d\right)$.  
The results for $n=5\,000$, $k=100$ and $500$ repetitions, i.e., the same setup as in the bivariate case, are displayed in Table 
\ref{tab:simu-res-multivariate}.

\begin{table}
	\begin{center}
		$d=3:$\\[1mm]
		
		\begin{tabular}{|l||r|r|r|r|r|r|}
			\hline \phantom{$A^{X^X}_{Y_Y}$}	
			& $1 / \widehat{\tau}_{I,n}^{BK}$ & $1 / \widehat{\tau}_{I,n}^{MK}$
			& $1 / \widehat{\tau}_{I,n,\mathrm{opt}}^{MK}$
			& $1 / \widehat{\tau}_{I,n}^{BU}$ & $1 / \widehat{\tau}_{I,n}^{MU}$
			& $1 / \widehat{\tau}_{I,n,\mathrm{opt}}^{MU}$ \\ \hline	
			bias         & 0.006 & 0.003 & 0.002 & 0.010 & 0.002 & 0.001\\ \hline
			std.~deviation  & 0.038 & 0.010 & 0.010 & 0.024 & 0.020 & 0.020\\ \hline			
		\end{tabular}
		\medskip
		
		$d=5:$\\[1mm]
		
		\begin{tabular}{|l||r|r|r|r|r|r|}
			\hline \phantom{$A^{X^X}_{Y_Y}$}	
			& $1 / \widehat{\tau}_{I,n}^{BK}$ & $1 / \widehat{\tau}_{I,n}^{MK}$
			& $1 / \widehat{\tau}_{I,n,\mathrm{opt}}^{MK}$
			& $1 / \widehat{\tau}_{I,n}^{BU}$ & $1 / \widehat{\tau}_{I,n}^{MU}$
			& $1 / \widehat{\tau}_{I,n,\mathrm{opt}}^{MU}$ \\ \hline
			bias               & 0.008 & 0.003 & 0.002 & 0.010 & 0.002 & 0.003\\ \hline
			std.~deviation     & 0.034 & 0.009 & 0.009 & 0.022 & 0.021 & 0.022\\ \hline			
		\end{tabular}
		\medskip
		
		$d=8:$\\[1mm]		
		
		\begin{tabular}{|l||r|r|r|r|r|r|}
			\hline \phantom{$A^{X^X}_{Y_Y}$}
			& $1 / \widehat{\tau}_{I,n}^{BK}$ & $1 / \widehat{\tau}_{I,n}^{MK}$
			& $1 / \widehat{\tau}_{I,n,\mathrm{opt}}^{MK}$
			& $1 / \widehat{\tau}_{I,n}^{BU}$ & $1 / \widehat{\tau}_{I,n}^{MU}$
			& $1 / \widehat{\tau}_{I,n,\mathrm{opt}}^{MU}$ \\ \hline
			bias            & 0.008 & 0.004 & 0.003 & 0.009 & 0.000 & 0.001 \\ \hline
			std.~deviation  & 0.034 & 0.007 & 0.007 & 0.021 & 0.023 & 0.023 \\ \hline
		\end{tabular}
		
	\end{center} 
	\caption{Simulation results for the three multivariate scenarios specified above with $n=5\,000$, $k=100$ and $u_n  = \Phi^{-1}_1(1 - k/n) = \Phi^{-1}_1(0.98)$.}
	\label{tab:simu-res-multivariate}
\end{table}

The ordering of the estimators is similar to the bivariate setting with the exception that the asymptotic variances of the estimators with unknown margins are very close to each other: 
$$ \mathrm{AVar}^{MK}_{\mathrm{opt}} < \mathrm{AVar}^{MU}_{\mathrm{opt}} \approx \mathrm{AVar}^{BU} < \mathrm{AVar}^{BK}.$$
Again, the observed standard deviations correspond well to theoretical ones apart from the fact that the asymptotic standard deviation of the moment-based estimator for the case of unknown margins is again slightly underestimated (again by roughly $10\,\%$) by the bootstrap procedure with $m_n = 0.2 \cdot n$ and $N_m=n$.

\subsection{Simulated example II: Multivariate Gaussian distribution}

In order to demonstrate the performance of our estimators in the case of asymptotic independence, we also consider the case that $\bm X$ follows a distribution with unit Fr\'echet margins, but a multivariate Gaussian copula. We consider the symmetric case with
$$ \begin{pmatrix}
\phantom{.}1\phantom{.} & \phantom{.}\rho\phantom{.} & \dots & \dots & \phantom{.}\rho\phantom{.} \\
\rho & 1 & \ddots & & \vdots \\
\vdots & \ddots & \ddots & \ddots & \vdots \\
\vdots & &  \ddots & 1 & \rho \\
\rho & \dots & \dots & \rho & 1
\end{pmatrix} \in \RR^{d \times d} $$
being the covariance matrix of the underlying Gaussian copula.
Thus, we obtain the theoretical extremal coefficient $\tau_{\{1,\ldots,d\}} = d$ whenever $\rho < 1$.
 Here, we choose $\rho=0.5$ and consider the $3$-, $5$- and $8$-dimensional case, respectively. In the case of asymptotic independence, the matrices $\bm V$ and $\widetilde{\bm V}$ might be degenerate, which hampers the optimization of the asymptotic variance. Therefore, we just fix the weights to the default choice
 $ \bm v^* = \widetilde{\bm v} = \left(1/d, \ldots, 1/d \right)$. 
The results for $n=5\,000$, $k=100$ and $500$ repetitions are displayed in Table 
 \ref{tab:simu-res-gauss}.
 
\begin{table}
	\begin{center}
	$d=3:$\\[1mm]
		
	\begin{tabular}{|l||r|r|r|r|}
	\hline \phantom{$A^{X^X}_{Y_Y}$}
	& $1 / \widehat{\tau}_{I,n}^{BK}$ & $1 / \widehat{\tau}_{I,n}^{MK}$
	& $1 / \widehat{\tau}_{I,n}^{BU}$ & $1 / \widehat{\tau}_{I,n}^{MU}$ \\ \hline
	bias            & 0.087 & 0.075 & 0.049 & 0.071  \\ \hline
	std.~deviation  & 0.017 & 0.005 & 0.008 & 0.005    \\ \hline			
	\end{tabular}
	\medskip
		
	$d=5:$\\[1mm]
		
	\begin{tabular}{|l||r|r|r|r|}
	\hline \phantom{$A^{X^X}_{Y_Y}$}
	& $1 / \widehat{\tau}_{I,n}^{BK}$ & $1 / \widehat{\tau}_{I,n}^{MK}$
	& $1 / \widehat{\tau}_{I,n}^{BU}$ & $1 / \widehat{\tau}_{I,n}^{MU}$ \\ \hline
	bias            & 0.084 & 0.077 & 0.047 & 0.072  \\ \hline
	std.~deviation  & 0.013 & 0.003 & 0.005 & 0.003  \\ \hline			
	\end{tabular}
	\medskip
		
	$d=8:$\\[1mm]		
		
	\begin{tabular}{|l||r|r|r|r|}
	\hline \phantom{$A^{X^X}_{Y_Y}$}
	& $1 / \widehat{\tau}_{I,n}^{BK}$ & $1 / \widehat{\tau}_{I,n}^{MK}$
	& $1 / \widehat{\tau}_{I,n}^{BU}$ & $1 / \widehat{\tau}_{I,n}^{MU}$ \\ \hline
	bias            & 0.072 & 0.071 & 0.041 & 0.067  \\ \hline
	std.~deviation  & 0.009 & 0.002 & 0.003 & 0.002  \\ \hline			
	\end{tabular} 
		
	\end{center} 
	\caption{Simulation results for the three asymptotically independent multivariate scenarios specified above with $n=5\,000$, $k=100$ and $u_n = \Phi^{-1}_1(1 - k/n) = \Phi^{-1}_1(0.98)$ based on Gaussian copulas with $\rho=0.5$.}
	\label{tab:simu-res-gauss}
\end{table}

Here, it is important to note that all the estimators exhibit a significant bias which can be explained by the fact that the true value is on the boundary of the estimation interval. However, it can be seen that the variance of the moment estimators is still smaller than the one of the benchmark estimators, which indicates that the moment estimators still show a reasonably good performance in the case of asymptotic independence. 
 
\subsection{Application to Precipitation Data in France} \label{subsec:precip}

We apply our estimators to estimate extremal dependence among different weather stations in France recording heavy precipitation. More precisely, we use daily precipitation data for the fall seasons (September, October, November) in the years 1976 to 2015 at nine weather stations in France, grouped into three regions (northwest, south and northeast) consisting of three stations per region. This dataset has also been analyzed in \citep{buritica2021stable}. 

We start with an explanatory analysis. Based on block maxima over 15 days, we fit generalized extreme value distributions via maximum likelihood separately for each station. All the extreme value indices are estimated to be positive (between 0.01 and 0.60). Even though, due to the rather small number of blocks, the positivity of the estimates is not significant, these results are in line with \cite{bernard-etal-2013} who analyzed heavy precipitation in France for the fall season based on a similar data set. Furthermore, we empirically estimate the pairwise tail dependence coefficients for each pair of stations. The table estimates that there is moderate extremal dependence within each region, while extremal dependence across different regions is much weaker (though not vanishing), see Appendix \ref{subsec:precipE} for details. Overall, the explanatory analysis suggests that the data are heavy-tailed and (at least within each region) asymptotically dependent. Therefore, it is reasonable to assume that they are covered by the framework of regular variation investigated in this work.

In order to avoid marginal standardization we make use of our flexible moment estimator for unknown margins  
$\widehat{\tau}_{I,n}^{MU}$. While this estimator allows for varying tail behaviour in the different components of the random vector $\bm X$, in particular for different tail indices, we also consider a modified version $
 \widetilde{\tau}_{I,n}^{MU}$ based on the assumption of a single tail index:
\begin{align*}
 \widetilde{\tau}_{I,n}^{MU}:=\frac{n^{-1} \sum_{l=1}^n \left(\bm v_I^\top 
 \frac{(\bm X_{l}/\bm X_{k:n})^{\widehat{ \alpha}}}{\ell_I((\bm X_{l}/\bm X_{k:n})^{\widehat{ \alpha}})}\right)^p
	 \mathds 1\{\ell_I(\bm X_l / \bm X_{k:n}) > 1\}}
{n^{-1} \sum_{l=1}^n \mathds 1\{\ell_I(\bm X_l / \bm X_{k:n}) > 1\}}\,,
\end{align*}
using a unique tail index estimator over each region $I$
 \begin{align*}
 \frac 1 {\widehat \alpha_{n,k}} = \frac{n^{-1} \sum_{l=1}^{n} \log\left( \ell_{I}(\bm X_l / \bm X_{k:n}) \right) \mathds 1 \left\{ \ell_{I}(\bm X_l / \bm X_{k:n}) > 1\right\}  }{n^{-1} \sum_{l=1}^n \mathds 1\{\ell_I(\bm X_l / \bm X_{k:n}) > 1}\,.
 \end{align*}

We first consider each region separately, i.e.\ we consider three different 3-dimensional vectors $\bm X^{(northwest)}$, $\bm X^{(south)}$ and $\bm X^{(northeast)}$. In order to avoid effects of temporal clustering, we consider block maxima over 5 days leading to a (nearly uncorrelated) time series of 720 observations at each station. For both estimators, we choose $k=40$. The results for the estimated extremal coefficients within each region are displayed in Table \ref{tab:application_regions}.

\begin{table}
	
\begin{center}	
		\begin{tabular}{|l||r|r|}
	\hline	
	region & based on different tail indices & based on single tail index \\ \hline
	northwest  & 1.89 (0.020) & 1.85 (0.018)  \\ \hline
	south      & 1.98 (0.019) & 1.89 (0.018)    \\ \hline			
	northeast  & 1.95 (0.020) & 1.95 (0.019)    \\ \hline		
\end{tabular}

\end{center} 
\caption{Estimated extremal coefficients for extreme precipitation within each region. The estimated standard deviations via bootstrap are given in brackets.}
\label{tab:application_regions}
\end{table}

It can be seen that there is moderate dependence within each region with extremal indices between 1.85 and 2 given a theoretical range of $[1,3]$. This is in line with the results of the explanatory analysis in Appendix \ref{subsec:precipE}. Interestingly, for both regions in the north, there are only little differences between the two estimators, while there are larger differences in the south. There the use of a single tail index would yield spurious dependence among the stations and reduce artificially the estimated tail dependence coefficient. The estimated standard deviations of both estimators are similar. Therefore, the more general estimator $\widehat{\tau}_{I,n}^{MU}$ appears to be preferable in this context. For this estimator, the effect of the choice of $k$ is also further analyzed in Appendix \ref{subsec:precipE}.
\medskip

Furthermore, we consider extremal dependence across the different regions. To this end, we consider the vector $\bm Y = (\max\{\bm X^{(northwest)}\}, \max\{\bm X^{(south)}\}, \max\{\bm X^{(northeast)} \}) \in \RR^3$ and estimate the extremal coefficient using the same estimators as above. This results in estimates of $2.62$ $(0.007)$ and $2.60$ $(0.007)$, respectively, showing that there is only little dependence across the different regions as also suggested by the results of the explanatory analysis in Appendix \ref{subsec:precipE}.

\section{Conclusion}\label{sec:conclusion}

We show some benefits of estimating extremal properties with moments of convex combination of spectral components. An analysis of the asymptotic variances yields a natural strategy consisting in plugging in the asymptotically optimal weights. This methodology is shown to efficiently reduce the variance when estimating the extremal coefficient by first order moments.  As demonstrated in the case of H\"usler--Reiss models, for many popular parametric models in multivariate
extremes there is a one-to-one correspondence between the extremal coefficients and the parameters. Consequently, our approach will result in more accurate parameter estimation for these models.

For more sophisticated parametric models, moments of higher orders $\EE[(\bm v_I^T \bm \Theta^I)^p]$, $p\ge1$, are also meaningful and can be used for asserting other properties of the spectral distribution. They represent interesting alternatives to the stable tail dependence function $L$ as, for small $p$, such functions of $\bm v$ have nice regularity properties (constant for $p=1$, quadratic for $p=2$) that have proven to be very useful for inference purpose.

We focus on reducing the variance under ad-hoc assumptions implying bias negligibility. In our numerical illustrations, the biases of our procedures were empirically small. A systematic bias analysis and potential bias reduction  for this new class of estimators should be considered in the future. Remind that, in the case of unknown margins, our procedure requires some variance estimation. To this end, we propose a subsampling step which could also be useful for other inference methods in extremes. In our simulation study, we found empirically that a subsampling ratio of $1/4$ provides sufficiently accurate estimates of the asymptotic standard deviation. Larger ratios lead to significant underestimation of the standard deviation. A more detailed analysis of this phenomenon in a general setting, however, is beyond the scope of this paper.

 Finally the first two steps of our procedure  in the unknown margin settings are specific to tail-equivalence of transformed margins $X_i^*=r_iX_i^{\alpha_i}$, $1\le i\le d$. A one-step procedure based on rank transforms similar to  the one of \cite{einmahl-dehaan-piterbarg-2001} could be investigated. However the ratio-type statistics involved in moment estimators have not yet been proven to relate well with rank transforms.

\section*{Acknowledgements}
The authors are grateful to three anonymous referees and an associate editor for their valuable comments and suggestions.
The work of MO was partially supported by Deutsche Forschungsgemeinschaft (DFG, German Research Foundation) under Germany's Excellence Strategy -- EXC 2075 -- 390740016. OW acknowledges the support of the French Agence Nationale de la Recherche (ANR) under reference ANR20-CE40-0025-01 (T-REX project).

\bibliographystyle{abbrvnat}
\bibliography{lit}

\appendix
	
\section{Proof of Theorem  \ref{thm:fclt}}\label{app:th5}

 In order to prove functional convergence, we apply the functional Central Limit Theorem 2.11.9 in \citet{vdv-wellner-1996}. Let us first check the convergence of the covariances.\\
	
	For $\bm v \in \partial B_1^+(\bm 0)$, $(\bm s,\bm \beta) \in A_\delta'$ and
	$p \in K_0$, define the function
	\begin{align*}
	f_{\bm v, \bm s, \bm \beta, p}: \ (0,\infty)^d \to [0,\infty), \quad
	f_{\bm v, \bm s, \bm \beta, p}(\bm x) 
	={} \left( \bm v^\top \left(\frac{\bm s \circ \bm x}{\|\bm s \circ \bm x\|}\right)^{1/\bm \beta} \right)^p 
	\mathds 1\{\|\bm s \circ \bm x\| > 1\}.
	\end{align*}
	Define the function space
	$\mathcal F = \{f_{\bm v, \bm s, \bm \beta, p}: \ \bm v \in \partial B_1^+(\bm 0), \ (\bm s, \bm \beta) \in A_\delta', \ p \in K_0\}$ 
	equipped with the semimetric 
	$\rho(f_{\bm v, \bm s, \bm \beta, p_1}, f_{\bm w, \bm t, \bm \gamma, p_2}) 
	= \max\{\|\bm v - \bm w\|, \|\bm s - \bm t\|,
	\|\bm \beta - \bm \gamma\|, |p_1 - p_2|\}$.
	Furthermore, for $l=1,\ldots,n$, $n \in \NN$, let
	$$ Z_{nl}(f) = \sqrt{\frac{a^*(u_n)}{n}} f( u_n^{-1} \bm X^*_l), \quad f \in \mathcal F.$$
	Then, each element of the sequence $(\{G_n(\bm v, \bm s, \bm \beta, p); 
	\, \bm v \in \partial B_1^+(\bm 0), \, (\bm s, \bm \beta) \in A_\delta', \, p \in K_0\})_{n \in \NN}$ can be rewritten as
	$$ \left\{ G_n(\bm v, \bm s, \bm \beta, p)\right\}_{ \bm v \in \partial B_1^+(\bm 0), \, (\bm s,\bm \beta) \in A_\delta', \, p \in K}
	= \left\{ \sum\nolimits_{l=1}^n ( Z_{nl}(f) - \EE Z_{nl}(f)) \right\}_{f \in \mathcal{F}}\,,\qquad n\in\NN\,.$$
	Analogously to the results from Prop.~\ref{prop:limits-ell}, we can calculate
	\begin{align*}
	& \Cov\left( \sum\nolimits_{l=1}^n Z_{nl}(f_{\bm v, \bm s, \bm \beta, p_1}),
	\sum\nolimits_{l=1}^n Z_{nl}(f_{\bm w, \bm t, \bm \gamma, p_2})\right) \\
	={}& a^*(u_n) \EE \left[f_{\bm v, \bm s,\bm \beta, p_1}(u_n^{-1} \bm X^*) 
	f_{\bm w, \bm t,\bm \gamma, p_2}(u_n^{-1} \bm X^*)\right] \\ 
	& - a^*(u_n) \EE[f_{\bm v, \bm s, \bm \beta, p_1}(u_n^{-1} \bm X^*)] 
	             \EE[f_{\bm w, \bm t, \bm \gamma, p_2}(u_n^{-1} \bm X^*)] \\
	={}& \EE \left[ \left( \bm v^\top \left(\frac{\bm s \circ \bm X^*}
	{\|\bm s \circ \bm X^*\|} \right)^{1/\bm \beta}\right)^{p_1} 
	\left(\bm w^\top \left(\frac{\bm t \circ \bm X^*}
	{\|\bm t \circ \bm X^*\|}\right)^{1/\bm \gamma} \right)^{p_2} \, \Big| \, 
	\|\bm s \circ \bm X^*\|  \wedge 
	\|\bm t \circ \bm X^*\| > u_n \right]\\
	&\cdot a^*(u_n) \PP\left(\|\bm s \circ \bm X^*\|  \wedge 
	\|\bm t \circ \bm X^*\|  > u_n\right) \\
	& - \frac 1 {\sqrt{a^*(u_n)}} 
	\EE \left[ \left( \bm v^\top \left(\frac{\bm s \circ \bm X^*}
	{\|\bm s \circ \bm X^*\|}\right)^{1/\bm \beta}\right)^{p_1} 
	\, \Big| \, \|\bm s \circ \bm X^*\|   > u_n \right) \cdot
	a^*(u_n) \PP\left(\|\bm s \circ \bm X^*\| > u_n \right) \\
	& \quad \cdot \frac 1 {\sqrt{a^*(u_n)}} 
	\EE \left[ \left( \bm w^\top \left(\frac{\bm t \circ \bm X^*}
	{\|\bm t \circ \bm X^*\|}\right)^{1/\bm \gamma}\right)^{p_2}
	\, \Big| \, \|\bm t \circ \bm X^*\|   > u_n \right) \cdot 
	a^*(u_n)  \PP\left(\|\bm t \circ \bm X^*\|   > u_n \right) \displaybreak[0]\,.
	\end{align*}	
	This sequence of covariances converges to, as $n\to \infty$,
	\begin{align*}& \frac{\EE\left[ 
		\left( \bm v^\top \left(\frac{\bm s \circ \bm \Theta}
		{\|\bm s \circ \bm \Theta\|}\right)^{1/\bm \beta} \right)^{p_1}
		\left( \bm w^\top \left(\frac{\bm t \circ \bm \Theta}
		{\|\bm t \circ \bm \Theta\|}\right)^{1/\bm \gamma} \right)^{p_2}                    
		\left(\|\bm s \circ \bm \Theta\|  \wedge
		\|\bm t \circ \bm\Theta\| \right)  \right]}
	{\EE \left[ \|\bm s \circ \bm \Theta\|  \wedge 
		\|\bm t \circ \bm \Theta\|  \right]} 
	 \cdot
	\tau \cdot \EE \left[ \|\bm s \circ \bm\Theta\|
	\wedge \|\bm t \circ \bm\Theta\|  \right] \\
	&={} \tau \cdot \EE\left[ 
	\left( \bm v^\top \left(\frac{\bm s \circ \bm \Theta}
	{\|\bm s \circ \bm \Theta\|}\right)^{1/\bm \beta} \right)^{p_1}
	\left( \bm w^\top \left(\frac{\bm t \circ \bm \Theta}
	{\|\bm t \circ \bm \Theta\|}\right)^{1/\bm \gamma} \right)^{p_2}                    
	\left(\|\bm s \circ \bm \Theta\| \wedge \|\bm t \circ \bm\Theta\|\right)  \right].
	\end{align*}
	
	Furthermore, as $\sup_{\bm x \in (0,\infty)^d} |f(\bm x)| \leq 1$ for all 
	$f \in \mathcal{F}$, we have that $$\|Z_{nl}\|_{\mathcal{F}}
	= \sup_{f \in \mathcal{F}} |Z_{nl}(f)| \leq \sqrt{\frac{a^*(u_n)}{n}} \ \text{a.s.} $$ 
	for all $l=1,\ldots,n$ and $n \in \NN$.\\
	
	Consequently, we check the Lindeberg condition: for any $k \in \NN$, we have
	\begin{align}
	\lim_{n \to \infty} \sum_{i=1}^n \EE[\|Z_{ni}\|_{\mathcal F}^k  
	\mathds 1\{ \|Z_{ni}\|_{\mathcal F} > \eta\}] 
	\leq{}& \lim_{n \to \infty} n \left(\frac{a^*(u_n)}{n}\right)^{k/2}  \mathds 1\{\sqrt{a^*(u_n)/n} > \eta\}  = 0, \label{eq:lindeberg}
	\end{align}
	where we used that, for any $\eta>0$, eventually $a^*(u_n)/n < \eta^2$ since $\lim_{n \to \infty} n / a^*(u_n) = \infty$. Making use of \eqref{eq:lindeberg} for $k=2$ one can easily see that, for all $a_i, \ldots, a_m \in \RR$, $f_1,\ldots,f_m \in \mathcal{F}$ and $m \in \NN$, the triangular arrays  $(\sum_{i=1}^m a_i Z_{nl}(f_i))_{1\le l\le n,n \in \NN}$ satisfy a Lindeberg condition and, therefore, $G_n$ converges to $G$ in terms of finite-dimensional distributions. Setting $k=1$ in \eqref{eq:lindeberg} gives the Lindeberg type condition of Theorem 2.11.9 in \citet{vdv-wellner-1996}.\\ 
	
	We further check the equi-continuity condition, i.e.\ the second condition of that theorem. To this end, we consider two elements 
	$f_{\bm v^{(1)}, \bm s^{(1)}, \bm \beta, p} \in \mathcal F$ and 
	$f_{\bm v^{(2)}, \bm s^{(2)}, \bm \gamma, p} \in \mathcal F$ with distance
	$\max\{\|\bm v^{(1)} - \bm v^{(2)}\|, \|\bm s^{(1)} - \bm s^{(2)}\|, \|\bm \beta - 
	\bm \gamma|\} < \eta$ for some  $0<\eta<\delta$. Notice that for sufficiently small $\eta$, we have that $\bm s^{(1)}$ and $\bm s^{(2)}$ necessarily belong to $A_{\delta,I}$ for the same non-empty index set $I$, i.e.\ that the same components of both vectors are different from zero. Moreover we assume with no loss of generality that there exists some $\bm y\in A_{\delta,I}$ satisfying $s_i^{(1)},s_i^{(2)}\in [y_i, y_i+\eta]$ for all $i\in I$.
	
	We distinguish three cases to obtain an upper bound for 
	$|  f_{\bm v^{(1)}, \bm s^{(1)},\bm \beta^{(1)}, p}(\bm x) 
	  - f_{\bm v^{(2)}, \bm s^{(2)},\bm \beta^{(2)}, p}(\bm x) |$:
	\begin{enumerate}
		\item $\| \bm s^{(1)} \circ \bm x\| \leq  1$ and 
		$\| \bm s^{(2)} \circ \bm x\| \leq  1$:\\
		In this case, we have
		$$|  f_{\bm v^{(1)}, \bm s^{(1)}, \bm \beta^{(1)}, p}(\bm x) 
		   - f_{\bm v^{(2)}, \bm s^{(2)}, \bm \beta^{(2)}, p}(\bm x) | = |0 - 0| = 0.$$
		\item  $\| \bm s^{(1)} \circ \bm x\| \geq  1$ and 
		$\| \bm s^{(2)} \circ \bm x\| \geq  1$:\\
		In this case, both the indicator functions in 
		$f_{\bm v^{(1)}, \bm s^{(1)}, \bm \beta^{(1)}, p}(\bm x)$ and in 
		$f_{\bm v^{(2)}, \bm s^{(2)}, \bm \beta^{(2)}, p}(\bm x) $ are equal to $1$. 
		The fact that $b^p - a^p = \int_a^b p x^{p-1} \, \mathrm{d}x \leq p (b-a)$ for $0 \leq a \leq b \leq 1$ and $p \in \NN_0$, leads to the bound
		\begin{align*}
		| f_{\bm v^{(1)}, \bm s^{(1)}, \bm \beta^{(1)}, p}(\bm x) 
		- f_{\bm v^{(2)}, \bm s^{(2)}, \bm \beta^{(2)}, p}(\bm x) |
		\leq{}& p \cdot | f_{\bm v^{(1)}, \bm s^{(1)}, \bm \beta^{(1)}, 1}(\bm x) 
		                - f_{\bm v^{(2)}, \bm s^{(2)}, \bm \beta^{(2)}, 1}(\bm x) |.
		\end{align*}
		Thus, with $\bm v = \bm v^{(1)} \wedge \bm v^{(2)}$, 
		$\bm s = \bm s^{(1)} \wedge \bm s^{(2)}$ and $ \bm \beta = \bm \beta^{(1)} \wedge \bm \beta^{(2)}$ denoting componentwise minima and the fact that $ \bm s + \eta \bm 1_I \leq (1 + (1+\delta) \eta) \bm s$ for 
		$\bm s \in A_{\delta,I}$, we obtain
		\begin{align*}
		& | f_{\bm v^{(1)}, \bm s^{(1)}, \bm \beta^{(1)}, p}(\bm x)
		  - f_{\bm v^{(2)}, \bm s^{(2)}, \bm \beta^{(2)}, p}(\bm x) | \nonumber \\
		\leq{}& 	p (\bm v + \eta \bm 1_I)^\top \left[
		\left(\frac{(\bm s + \eta \bm 1_I) \circ \bm x}{\|\bm s \circ \bm x\|}\right)^{1/(\bm \beta + \eta \bm 1)}
		\vee \left(\frac{(\bm s + \eta \bm 1_I) \circ \bm x}{\|\bm s \circ \bm x\|}\right)^{1/\bm \beta}\right] \nonumber \\
		& \qquad \qquad -  p \bm v^\top \left[\left(\frac{\bm s \circ \bm x}{\|(\bm s + \eta \bm 1_I) \circ \bm x\| }\right)^{1/\bm \beta} \wedge \left(\frac{\bm s \circ \bm x}{\|(\bm s + \eta \bm 1_I) \circ \bm x\| }\right)^{1/(\bm \beta + \eta \bm 1)} \right] \nonumber \\
		\leq{}&  p (\bm v + \eta \bm 1_I)^\top \left[
		\left(\frac{\bm s \circ \bm x}{\|\bm s \circ \bm x\|} (1 + (1+\delta) \eta) \right)^{1/(\bm \beta + \eta \bm 1)}
		\vee \left(\frac{\bm s \circ \bm x}{\|\bm s \circ \bm x\|} (1 + (1+\delta) \eta) \right)^{1/\bm \beta}\right] \nonumber \\
		& \qquad \qquad - p \bm v^\top \left[\left(\frac{\bm s \circ \bm x}{\|\bm s \circ \bm x\|}
		\frac 1 {1 + (1+\delta) \eta} \right)^{1/\bm \beta} \wedge \left(\frac{\bm s \circ \bm x}{\|\bm s \circ \bm x\|} \frac{1}{1 + (1+\delta) \eta}\right)^{1/(\bm \beta + \eta \bm 1)} \right] \displaybreak[0] \\  
		\leq{}& p (\bm v + \eta \bm 1)^\top 
		\left(\frac{\bm s \circ \bm x}{\|\bm s \circ \bm x\|}\right)^{1/(\bm \beta + \eta \bm 1)} (1+(1+\delta)\eta )^{1+\delta}
		- p \bm v^\top \left(\frac{\bm s \circ \bm x}{\|\bm s \circ \bm x\| }\right)^{1/\bm \beta} \frac{1}{(1+(1+\delta)\eta)^{1+\delta}} \displaybreak[0] \\
		\leq{}& p \eta \bm 1^\top \bm 1_I (1+(1+\delta)\eta)^{1+\delta} \\
		& + p \bm v^\top \bm 1_I  \max_{i \in I} \left(  \left(\frac{s_i x_i}{\|\bm s \circ \bm x\|}\right)^{1/(\bm \beta + \eta \bm 1)}   (1+(1+\delta)\eta)^{1+\delta} - 
		\left(\frac{s_i x_i}{\|\bm s \circ \bm x\|}\right)^{1/\bm \beta} \frac{1}{(1+(1+\delta)\eta)^{1+\delta}}\right) \displaybreak[0] \\
		\leq{}& p \eta d (1+(1+\delta)\eta)^{1+\delta} + p \max_{i \in I} \left(\frac{s_i x_i}{\|\bm s \circ \bm x\|}\right)^{1/(\bm \beta + \eta \bm 1)} \cdot \left[(1+(1+\delta)\eta)^{1+\delta} - \frac{1}{(1+(1+\delta)\eta)^{1+\delta}} \right] \\
		& + p \max_{i \in I} \left(  \left(\frac{s_i x_i}{\|\bm s \circ \bm x\|}\right)^{1/(\bm \beta + \eta \bm 1)}   - 
		\left(\frac{s_i x_i}{\|\bm s \circ \bm x\|}\right)^{1/\bm \beta} \right) \cdot 
		\frac{1}{(1+(1+\delta)\eta)^{1+\delta}} \\
		\leq{}& p d (1+(1+\delta)\eta)^{1+\delta} \eta  + p \frac{(1+(1+\delta)\eta)^{2+2\delta} - 1}{(1+(1+\delta)\eta)^{1+\delta}} \\
		& + p \max_{i \in I} \left(  \left(\frac{s_i x_i}{\|\bm s \circ \bm x\|}\right)^{1/(\bm \beta + \eta \bm 1)}   - 
		\left(\frac{s_i x_i}{\|\bm s \circ \bm x\|}\right)^{1/\bm \beta} \right) \\	 
		\leq{}& p d (1+(1+\delta)\eta)^{1+\delta}  \eta + p 2 (1+\delta)^2 (1+(1+\delta)\eta)^{\delta} \eta
		+ p (1+\delta)^2 \left[\sup\nolimits_{z \in [0,1]} \log(z) z^{1/(1+\delta)}\right] \eta .
		\end{align*}
		Bounding $p$ by $p_{\max} = \max_{p\in K_0} p$ and noting that  $\sup_{z \in [0,1]} \log(z) z^{1/(1+\delta)} < \infty$,  
		we have an overall bound of the form
		$C \eta$ some constant  $C = C(\delta,p_{\max},d) > 0$ for sufficiently small $\eta$.\\
		This case occurs only if $(1+\delta) \|\bm x\| > 1$ as $ \bm s^{(1)}$ and $\bm  s^{(2)}$ belong to $A_{\delta,I}$.
		\item $\min\{\|\bm s^{(1)} \circ \bm x\|, \|\bm s^{(2)} \circ \bm x\|\} \leq 1 < \max\{\|\bm s^{(1)} \circ \bm x\|, \|\bm s^{(2)} \circ \bm x\|\}
		$:\\	
		In this case one of the functions $f_{\bm v^{(1)}, \bm s^{(1)},\bm \beta^{(1)}, p}(\bm x)$ and $f_{\bm v^{(2)}, \bm s^{(2)},\beta_2, p}(\bm x)$ equals zero, while the other one is bounded by 1. Consequently, noting that this case does not appear if $\bm s^{(1)}= \bm s^{(2)}$, we have
		\begin{align}
		& |f_{\bm v^{(1)}, \bm s^{(1)}, \bm \beta^{(1)}, p}(\bm x) 
		 - f_{\bm v^{(2)}, \bm s^{(2)}, \bm \beta^{(2)}, p}(\bm x)| \nonumber \\
		\leq{}& \mathbf 1 \left\{ \bigcup\nolimits_{i \in I, s_i^{(1)}\ne s_i^{(2)}} 
		\{ u_n / (s_i^{(1)}\vee s_i^{(2)})  \leq x_{i} <  u_n
		/ (s_i^{(1)}\wedge s_i^{(2)}) \}\right\}. \label{eq:case3} 
		\end{align}       
		As, by construction, $y_i\le\min\{s^{(1)}_i, s_i^{(2)}\}\le \max\{s^{(1)}_i, s_i^{(2)}\}\le y_i+\eta$  this case occurs only if 
		$y_i x_i \leq 1 <
		(y_i+\eta) x_i$
		for at least one $i \in I$. 
	\end{enumerate}  
	Combining the three cases, we obtain the following uniform bound on 
	$\|\bm v^{(1)} - \bm v^{(2)}\| < \eta$, $\|\bm \beta^{(1)} - \bm \beta^{(2)}\| < \eta$ and $s_i^{(1)},s_i^{(2)}\in [y_i, y_i+\eta]$, $i\in I$, for  some $\bm y\in A_{\delta,I}$: 
	\begin{align}
	& (Z_{nl} (f_{\bm v^{(1)}, \bm s^{(1)},\bm \beta^{(1)}, p})
	- Z_{nl}(f_{\bm v^{(2)}, \bm s^{(2)},\beta_2, p}))^2 \nonumber\\
	\leq{}&   \frac{a^*(u_n)}{n} C^2\, \eta^2\, \mathbf 1 \left\{  \|\bm X^*_l\| > (1+\delta)^{-1} u_n\} \right\} \nonumber \\ 
	& +   \frac{a^*(u_n)}{n} \cdot \mathbf 1 \left\{ \bigcup\nolimits_{i \in I} 
	\{ u_n (y_i+\eta)^{-1}  \leq X^*_{li} <  u_n
	y_i^{-1} \}\right\} \nonumber \\
	=:{}& Q^{(1)}_{nl}(\eta) 
	+ Q^{(2)}_{nl}(\eta,\bm y)\,.  \label{eq:def-q}  
	\end{align}
	In the following, we will consider the two summands separately.
	From the definition of $Q^{(1)}_{nl}$ and regular variation of the function $a^*$, it directly follows that
	\begin{align*}
	\sum_{l=1}^n \EE\left[Q^{(1)}_{nl}(\eta) \right]
	&{} \leq a^*(u_n) C^2\, \eta^2\,  
	\PP\left(\|\bm X^*\| > (1+\delta)^{-1} u_n \right)\\
	&{} \sim  \, \tau\,  \frac{a^*(u_n)}{a^*((1+\delta)^{-1}u_n)} C^2 \,\eta^2
	\, \sim  \, \tau\, (1+\delta)\, C^2 \,\eta^2 \,,\qquad  n\to \infty. 	        	      
	\end{align*}
	As $\tau\le d$ we obtain  for $n$ sufficiently large
	\begin{align} \label{eq:q1}
	\sum_{l=1}^n  \EE\left[Q^{(1)}_{nl}(\eta) \right]
	\leq{} 2\, d\,(1+\delta) \, C^2\, \eta^2.  	        	      
	\end{align}

	For $Q^{(2)}_{nl} $  we have that, as $n\to \infty$,
	\begin{align*}
	\sum_{l=1}^n \EE\left[ Q^{(2)}_{nl}(\eta,\bm y)\right] 
	\leq{}& a^*(u_n) \sum_{i \in I} \PP\left(u_n  (y_i+\eta)^{-1} \leq X^*_i <  u_n
	y_i^{-1} \right) \\
	\sim{}& \sum_{i \in I} \left( \frac{\PP(X_i^* > u_n  (y_i+\eta)^{-1}   )}
	{\PP(X_i^* > u_n)} - \frac{\PP(X_i^* > u_n  y_i^{-1} )}
	{\PP(X_i^* > u_n)} \right).
	\end{align*}
	with $\bm s^{(1)}, \bm s^{(2)}$ in $A_{\delta,I}$.
	As the function $u \mapsto \PP(X_i^* > u)$ is regular varying with index $-1$ for each $i=1,\ldots,d$, the convergence of $\PP(X_i^* > u_n y) / \PP(X_i^* >u_n) \to y^{-1}$ is uniform in
	$y \in [(1+\delta)^{-1}, 1+\delta]$, i.e.\ there exists a sequence of positive numbers $\{f_n\}_{n \in \NN}$ with $f_n \to 0$ such that
	$$ (1 - f_n) y^{-1} \leq \frac{\PP(X_i^* > u_n y)}{\PP(X_i^* > u_n)}
	\leq (1 + f_n) y^{-1} $$
	for all $n \in \NN$, $i \in \{1,\ldots,d\}$ and $y \in [(1+\delta)^{-1}, 1+\delta]$. Thus, we obtain, as $n\to \infty$, 
	\begin{align*}
	& \sup_{\bm y \in A_{\delta,I}}  \sum_{i \in I} \left( \frac{\PP(X_i^* > u_n  (y_i+\eta)^{-1} )}
	{\PP(X_i^* > u_n)} - \frac{\PP(X_i^* > u_n  y_i^{-1} )}
	{\PP(X_i^* > u_n)} \right)  \\
	{}\leq{}& |I| \cdot (1+f_n) \cdot \eta + 2 f_n \cdot \sup_{\bm y \in A_{\delta,I}} \sum_{i \in I}  y_i 
	\leq{} 2 d \eta +  2 d f_n (1 + \delta)
	\end{align*}
	for sufficiently large $n$.	
	
	Combining the estimates obtained for $Q^{(1)}_{nl} $ and $Q^{(2)}_{nl} $ and denoting $\bm s^{(1)},\bm s^{(2)}\in [\bm y, \bm y+\eta]$ for $\bm s^{(1)},\bm s^{(2)}\in A_{\delta,I}$ satisfying $s_i^{(1)},s_i^{(2)}\in [y_i, y_i+\eta]$ for all $i\in I$, $I$ uniquely determined by $\bm y\in A_{\delta,I}$ for any $\bm y\in A_\delta$, we obtain for $n$ large enough
	\begin{align*}
	\sup_{\bm y \in A_{\delta}} \sum_{l=1}^n &\EE\left[\sup_{\|\bm v^{(1)} - \bm v^{(2)}\| < \eta,(\bm s^{(1)},\bm s^{(2)})\in [\bm y, \bm y+\eta]^2, \|\bm \beta^{(1)} - \bm \beta^{(2)}\| < \eta} 
	  ( Z_{nl}(f_{\bm v^{(1)}, \bm s^{(1)}, \beta^{(1)}, p})
	  - Z_{nl}(f_{\bm v^{(2)}, \bm s^{(2)}, \beta^{(2)}, p}))^2 \right]\\
	&\le 2\, d\left(\,(1+\delta) \, C^2\, \eta^2
	+ \eta + (1+\delta) f_n\right) \,.
	\end{align*}
	Reminding that $f_n \to 0$, we obtain
   	$$ \sup_{\rho(f,g) < \eta_n} \sum_{l=1}^n \EE\left[ (Z_{nl}(f)-Z_{nl}(g))^2 \right] \leq 2\, d\left(\,(1+\delta) \, C^2\, \eta_n^2
   	+ \eta_n + (1+\delta) f_n\right) \to 0 $$
   	as $\eta_n \to 0$.

	Now, we are checking the entropy condition of Theorem 2.11.9 in \citet{vdv-wellner-1996} by constructing and counting $\varepsilon$-brackets for $\varepsilon>0$. 
	To this end, for each $i \in I$ and $n \in \NN$, we consider the conditional distribution of $ u_n / X_i^* \mid X_i^* > (1+\delta)^{-1} u_n$ which, by definition is supported on $(0,1+\delta)$.\\
	
	Putting all the atoms $\{m_j^{ (i,n)}\}$ of mass larger than $\varepsilon^2/(4(1+\delta)d)$ separately into singletons (there are at most $1\le j\le \lceil 4(1+\delta)d/\varepsilon^2\rceil$ of them), the set $[(1+\delta)^{-1}, 1+\delta]\setminus \cup_j \{m_j^{ (i,n)}\}$ can be divided in
	at most $ \lceil 4(1+\delta)d/\varepsilon^2\rceil$ sets $J_j^{(i,n)}$, $ \lceil 4(1+\delta)d/\varepsilon^2\rceil+1 \le j\le 2 \lceil 4(1+\delta)d/\varepsilon^2\rceil$
	such that    
	$$ \PP\left(  \frac{u_n}{X_i^*} \in J_j^{(i,n)}
	\, \bigg| \, X_i^* > (1+\delta)^{-1} u_n\right) \leq \frac{\varepsilon^2}{4(1+\delta)d},$$
	for $\lceil 4(1+\delta)d/\varepsilon^2\rceil+1 \le j\le 2 \lceil 4(1+\delta)d/\varepsilon^2\rceil $.
	Thus, denoting $J^{(i,n)}_j=\{m_j^{(i,n)}\}$ for any $1\le j\le \lceil 4(1+\delta)d/\varepsilon^2\rceil$, the set $A_\delta \subset [0,1+\delta]^{d}$
	can be divided into $2^d\lceil 4 (1 +\delta) d/\varepsilon^2\rceil^d$ sets of the form
	\begin{equation} \label{eq:j-sets}
	J = J^{ (n)}_{j_1,\ldots,j_d} = J^{(1,n)}_{j_1} \times J^{(2,n)}_{j_2} \times \ldots \times J^{(d,n)}_{j_d} , \quad 1\le j_1,\ldots, j_d \le 2 \lceil 4(1+\delta)d/\varepsilon^2\rceil\, .
	\end{equation}
	Choosing $n$ sufficiently large such that 
	$ a^*(u_n) \leq 2(1+\delta) \PP(X_i^* > u_n (1+\delta)^{-1})^{-1}$
	for all $i=1,\ldots,d$, we obtain that
	\begin{align*}
	& \EE\left[ \sup_{\bm s^{(1)}, \bm s^{(2)} \in J}
	a^*(u_n) \mathbf 1 \left\{ \bigcup\nolimits_{i \in I, s_i^{(1)}\ne s_i^{(2)}} 
	\{ u_n / (s_i^{(1)}\vee s_i^{(2)})  \leq X^*_{i} <  u_n
	/ (s_i^{(1)}\wedge s_i^{(2)}) \}\right\} \right] \\
	& \leq 2 (1 + \delta) \sum_{i\in I, j_i> \lceil 4(1+\delta)d/\varepsilon^2\rceil} 
	\PP \left(\frac{u_n}{X_i^*} \in J_{j_i}^{(i,n)} \, \Big| \, X_i^* > (1+\delta)^{-1} u_n \right) \leq \frac{\varepsilon^2}{2}	 
	\end{align*}	    
	Setting $\eta_0 = \varepsilon / (2 C \sqrt{d(1+\delta)})$,
	Eq.~\eqref{eq:case3} and \eqref{eq:q1} imply that 
	\begin{align*}
	& \EE\left[\sup_{\|\bm v^{(1)} - \bm v^{(2)}\| < \eta_0, \|\bm s^{(1)} - \bm s^{(2)}\| < \eta_0, \bm s^{(1)},\bm s^{(2)} \in J,\|\bm \beta^{(1)} - \bm \beta^{(2)}\|<\eta_0} 
	  (Z_{nl}(f_{\bm v^{(1)}, \bm s^{(1)}, \bm \beta^{(1)}, p})
	 - Z_{nl}(f_{\bm v^{(2)}, \bm s^{(2)}, \bm \beta^{(2)}, p}))^2 \right] \\
	& \leq{} \sum_{l=1}^n \EE\left[ Q_{nl}^{(1)}(\eta_0)\right] \\
	& + \EE\left[ \sup_{\bm s^{(1)}, \bm s^{(2)} \in J}
	a^*(u_n) \mathbf 1 \left\{ \bigcup\nolimits_{i \in I,s_i^{(1)}\ne s_i^{(2)}} 
	\{ u_n / (s_i^{(1)}\vee s_i^{(2)})  \leq X^*_{i} <  u_n
	/ (s_i^{(1)}\wedge s_i^{(2)}) \}\right\} \right] \\
	& \leq \frac{\varepsilon^2}{2} + \frac{\varepsilon^2}{2}
	= \varepsilon^2.
	\end{align*}
	Now, we note that $\partial B_1^+(\bm 0) \subset [0,1]^d$ can be covered by $ \lceil 1 / \eta_0 \rceil^d$ hyperrectangles $B_1, B_2,\ldots$ of side length $\eta_0$, while $A_\delta' \subset [0,1+\delta]^{2d}$
	can be covered by at most $ \lceil (1+\delta) / \eta_0 \rceil^{2d}$
	hyperrectangles $C_1, C_2, \ldots$ of the same side length: 
	Combining each set $B_k$ with each intersection of a set $C_j$ with each set $J^{(n)}_{j_1,\ldots,j_d}$ from \eqref{eq:j-sets} extended to $J^{ (n)}_{j_1,\ldots,j_d}\times [(1+\delta)^{-1}, 1+\delta]^d$, we obtain
	at most 
	$$\left\lceil \frac{1}{\eta_0} \right\rceil^{d} \cdot
	\left\lceil \frac{1+\delta}{\eta_0} \right\rceil^{2d} \cdot
	\left\lceil \frac{4(1+\delta)}{\varepsilon^2} \right\rceil^d 
	= \left\lceil \frac{2 C \sqrt{d(1+\delta)}}{\varepsilon} \right\rceil^d  
	\cdot \left\lceil \frac{2 C \sqrt{d(1+\delta)^3}}{\varepsilon} \right\rceil^{2d} 
	\cdot \left\lceil \frac{4(1+\delta)}{\varepsilon^2} \right\rceil^d $$ 
	sets of the type $D^{(n)}_{k,j,j_1,\ldots,j_d} = B_k \times (C_j \cap (J^{(n)}_{j_1,\ldots,j_d} \times [(1+\delta)^{-1},1+\delta]))$ satisfying
	$$
	\EE\left[\sup_{	
			(\bm v^{(1)}, \bm s^{(1)}, \bm \beta^{(1)}, p),
			(\bm v^{(2)}, \bm s^{(2)}, \bm \beta^{(2)}, p) \in D^{(n)}_{k,j,j_1,\ldots,j_d}}
	 (Z_{nl}(f_{\bm v^{(1)}, \bm s^{(1)}, \bm \beta^{(1)}, p})
	- Z_{nl}(f_{\bm v^{(2)}, \bm s^{(2)}, \bm \beta^{(2)}, p}))^2 \right]  \leq \varepsilon^2. 	 
	$$	
	As this number are the same for  each $p \in K_0$, the function space $\mathcal{F}$ can be partitioned into at most
	$$  N_{[\,]}(\varepsilon, \mathcal{F}, L_2) 
	= |K_0|  \left\lceil \frac{2 C \sqrt{d(1+\delta)}}{\varepsilon} \right\rceil^d 
	\cdot \left\lceil \frac{2 C \sqrt{d(1+\delta)^3}}{\varepsilon} \right\rceil^{2d} 
	\cdot \left\lceil \frac{4(1+\delta)}{\varepsilon^2} \right\rceil^d
	\in \mathcal{O}(\varepsilon^{-5d})
	$$
	$\varepsilon$-brackets. Consequently, 
	$$ \int_{0}^{\varepsilon_n} \sqrt{\log  N_{[\,]}(\varepsilon, \mathcal{F}, L_2)} \, \mathrm{d}\varepsilon 
	\to 0 $$
	as $\varepsilon_n \to 0$. Now, weak convergence follows from Theorem 2.11.9 in \citet{vdv-wellner-1996}.

\section{Proof of Cor.~\ref{coro:fclt-ratio}}\label{app:cor6}

	From Thm.~\ref{thm:fclt} and Eq.~\eqref{eq:bias-negligible},
	we deduce that
	\begin{align} \label{eq:unbiased-fclt}
	\Big\{ \sqrt{\frac{n}{a^*(u_n)}} \Big( a^*(u_n) \widehat{M}_{n,u_n}(\bm v, \bm s, \bm \beta, p) -  \tau  \EE\Big[ \Big( \bm v^\top \Big(\frac{\bm s \circ \bm \Theta}{\|\bm s \circ \bm \Theta\|}\Big)^{1/\beta} \Big)^p \|\bm s \circ \bm \Theta\| \Big] \Big) \Big\} \to G
	\end{align}
	weakly in $\ell^\infty(\partial B_1^+(\bm 0) \times A_{\delta}' \times (K \cup \{0\}))$ as $n \to \infty$.
	\medskip
	
	For the case $p=0$ where 
	$\widehat{M}_{n,u_n}(\bm v, \bm s, \bm \beta, 0) = \widehat{P}_{n, u_n}(\bm s)$, 
	we can use the monotonicity of  
	$\widehat{P}_{n, u_n}$ to conclude that, for all $\bm s \in A_{\delta,I}$,
	$$a^*(u_n) \widehat{P}_{n, u_n,\ell}(\bm s) 
	\geq a^*(u_n) \widehat{P}_{n, u_n,\ell}((1+\delta)^{-1} \bm 1_I)
	\to_p \tau \cdot \EE(\|(1+\delta)^{-1} \cdot \bm \Theta\|)
	= \tau_I (1 + \delta)^{-1},$$
	due to Eq.~\eqref{eq:norm-const-subset}.
	Thus, for any $\varepsilon \in (0, \tau_I (1+\delta)^{-1})$, we have
	\begin{align} \label{eq:neglect-lowerbound}
	\lim_{n \to \infty} & \PP( a^*(u_n) \widehat{P}_{n, u_n}(\bm s) > \varepsilon 
	\text{ for all } \bm s \in A_{\delta,I}) \\
	\geq{}& \lim_{n \to \infty} \PP( a^*(u_n) \widehat{P}_{n, u_n} ((1+\delta)^{-1} \bm 1_I) > \varepsilon) = 1. \nonumber	 
	\end{align}
	This can be strengthened to 
	$$
	\lim_{n \to \infty} \PP( a^*(u_n) \widehat{P}_{n, u_n}(\bm s) > \varepsilon 
	\text{ for all } \bm s \in A_{\delta})=1$$
	as they are only finitely many $\emptyset \neq I \subset \{1,\ldots,d\}$.
	This implies that, for $p=0$ in Eq.~\eqref{eq:unbiased-fclt}, the term
	$ a^*(u_n) \widehat{M}_{n,u_n}(\bm v, \bm s, \bm \beta, 0) = 
	a^*(u_n) \widehat{P}_{n, u_n}(\bm s)$ may be replaced by
	the pointwise maximum $\varepsilon \vee \{a^*(u_n) \widehat{P}_{n, u_n}(\bm s)\}$.
	As the map 
	\begin{align*}
	\phi:{} & \ell^\infty(\partial B_1^+(\bm 0) \times A'_{\delta})
	\times \left( \ell^\infty(\partial B_1^+(\bm 0) \times A'_{\delta}) \cap [\varepsilon,\infty)^{\partial B_1^+(\bm 0) \times A'_{\delta}}\right) \to 
	\ell^\infty(\partial B_1^+(\bm 0) \times A'_{\delta}), \\
	& \phi(f,g)( \bm v, \bm s, \bm \beta) 
	  = f(\bm v,\bm s, \bm \beta) / g(\bm v, \bm s, \bm \beta)
	\end{align*} 
	is Hadamard-differentiable \citep[cf.\ Lemma 3.9.25 in][]{vdv-wellner-1996},
	weak convergence of the process
	$$\bigg( \bigg\{ \sqrt{\frac{n}{a(u_n)}} 
	\Big( \frac{a^*(u_n) \widehat{M}_{n, u_n}(\bm v, \bm s, \bm \beta, p)}
	{\varepsilon \vee \{a^*(u_n) \widehat{P}_{n, u_n}(\bm s)\}} - 
	c(\bm v, \bm s, \bm \beta, p) \Big); \bm v \in \partial B_1^+(\bm 0), \,(\bm s, \bm \beta) \in A_\delta', \, p \in K\bigg\}\bigg)_{n\in \NN} $$
	to $\widetilde G$ is obtained by the functional delta method \citep[cf.\ Theorem 3.9.4 in][for instance]{vdv-wellner-1996} with covariance of the form
	\begin{align} \label{eq:covar-tildeg}
	& \Cov(\widetilde G(\bm v, \bm s, \bm \beta, p_1), 
	       \widetilde G(\bm w, \bm t, \bm \gamma, p_2)) \\
	={}& \frac{\EE \big[ 
		\big(\bm v^\top \big(\frac{\bm s \circ \bm\Theta}
		{\|\bm s \circ \bm \Theta\|}\big)^{1/\bm \beta}\big)^{p_1} 
		\big(\bm w^\top \big(\frac{\bm t \circ \bm\Theta}
		{\|\bm t \circ \bm \Theta\|}\big)^{1/\bm \gamma}\big)^{p_2}
		\left(\|\bm s \circ \bm \Theta\| \wedge 
		\|\bm t \circ \bm \Theta\| \right) \big]}
	{\tau \cdot \EE\left[\|\bm s \circ \bm \Theta\| \right] \cdot 
		\EE\left[\|\bm t \circ \bm \Theta\| \right] } \nonumber \\
	&- \frac{\EE \big[ 
		\big(\bm v^\top \big(
		\frac{\bm s \circ \bm \Theta}{\|\bm s \circ \bm \Theta\|}\big)^{1/\bm \beta}\big)^{p_1} 
		\|\bm s \circ \bm \Theta\| \big] 
		\cdot \EE \big[ 
		\big(\bm w^\top \big(
		\frac{\bm t \circ \bm \Theta}{\|\bm t \circ \bm \Theta\|}\big)^{1/\bm \gamma}\big)^{p_2}
		\big(\|\bm s \circ \bm \Theta\| \wedge
		\|\bm t \circ \bm \Theta\| \big) \big]}      
     	{\tau \cdot \left(\EE\left[\|\bm s \circ \bm \Theta\|\right]\right)^2 \cdot 
		\EE\left[\|\bm t \circ \bm \Theta\|\right] } \nonumber \\
	&  - \frac{\EE \big[ 
		\big(\bm v^\top \big(\frac{\bm s \circ \bm\Theta}{\|\bm s \circ \Theta\|}\big)^{1/\bm \beta}\big)^{p_1}
		\big(\|\bm s \circ \bm \Theta\| \wedge 
		\|\bm t \circ \bm \Theta\|\big) \big] \cdot
		\EE \big[ \big(\bm w^\top \big(\frac{\bm t \circ \bm \Theta}
		{\|\bm t \circ \bm \Theta\|}\big)^{1/\bm \gamma}\big)^{p_2} \|\bm t \circ \bm \Theta\| \big]}      
	{\tau \cdot \EE\left[\|\bm s \circ \bm \Theta\|\right] \cdot 
		\left(\EE\left[\|\bm t \circ \bm \Theta\|\right]\right)^2 } \nonumber \\
	& + \frac{\EE \big[ 
		\big(\bm v^\top \big(\frac{\bm s \circ \bm \Theta}
		{\|\bm s \circ \bm \Theta\|}\big)^{1/\bm \beta}\big)^{p_1} 
		\|\bm s \circ \bm \Theta\| \big] 
		\cdot \EE \big[ 
		\big(\bm w^\top \big(\frac{\bm t \circ \bm \Theta}
		{\|\bm t \circ \bm \Theta\|}\big)^{1/\bm \gamma}\big)^{p_2} 
		\|\bm t \circ \bm \Theta\| \big] \cdot
		\EE\left[\|\bm s \circ \bm \Theta\| \wedge 
		\|\bm t \circ \bm \Theta\| \right]}      
	{\tau \cdot \left(\EE\left[\|\bm s \circ \bm \Theta\| \right] \right)^2 \cdot 
		\left(\EE\left[\|\bm t \circ \bm \Theta\| \right] \right)^2 }. \nonumber     
	\end{align} 
	Using again Eq.~\eqref{eq:neglect-lowerbound}, the assertion follows.

\section{Proof of Thm.~\ref{thm:an-hill}}\label{app:th8}

Analogously to Cor.~\ref{coro:fclt-ratio}, 
We first show a functional limit theorem for the estimator $\widehat L_{n,u}(\bm s) / \widehat P_{n,u_n}(\bm s)$ based on normalized observations $\bm X_i^*$.

\begin{proposition} \label{thm:fclt-hill}
	Let the assumptions of Thm.~\ref{thm:fclt} hold and assume that there exists some $\delta > 0$ such that Eq.~\eqref{eq:bias-hill} is satisfied and Eq.~\eqref{eq:bias-negligible} holds for $p=0$. Then, the sequence of processes 
	$(\{\widetilde H_n(\bm s): \, \bm s \in \bigcup_{i=1}^d A_{\delta,\{i\}}\})_{n \in \NN}$
	 with
	$$ \widetilde H_n(\bm s) = \sqrt{\frac{n}{a^*(u_n)}} \left(\frac{\widehat L_{n,u}(\bm s)}{\widehat P_{n,u_n}(\bm s)} - \frac 1 {\alpha_i} \right), \qquad \bm s \in A_{\delta,\{i\}}, \ 1 \leq i \leq d, $$
	converges weakly in $\ell^\infty(A_\delta)$ to a tight centered Gaussian process $\widetilde H$ with covariance
	\begin{equation}\label{eq:covtildeh}
	\Cov(\widetilde H(\bm s), \widetilde H(\bm t)) \\
	=\frac{1}{\tau\alpha_i\alpha_j}\cdot\frac{
		\EE\left[(s_i \Theta_i) \wedge (t_j \Theta_j) \right]}      
	{\EE\left[s_i \Theta_i \right] \cdot \EE\left[t_j \Theta_j\right]}\,,\qquad 
	\bm s \in A_{\delta,\{i\}}, \ \bm t \in A_{\delta,\{j\}}, \	1 \leq i,j \leq d.    
	\end{equation}  
\end{proposition}
\begin{proof}
	First, we will show that the sequence of processes $(\{H_n(\bm s):\, \bm s \in  \bigcup_{i=1}^d A_{\delta,\{i\}}\})_{n \in \NN}$ with    
	$$ H_n(\bm s) = \sqrt{\frac{n}{a^*(u_n)}} \left(a^*(u_n) \widehat L_{n,u_n}(\bm s) - a^*(u_n) \EE[\widehat L_{n,u_n}(\bm s)]\right) $$
	converges weakly in $\ell^\infty(\bigcup_{i=1}^d A_{\delta,\{i\}})$ to a centered
	Gaussian process $H$.
	Analogously to the proof of Thm.~\ref{thm:fclt}, we define the function space
	$\mathcal{F}' = \{ f_{\bm s}: \, \bm s \in \bigcup_{i=1}^d A_{\delta,\{i\}}\}$ 
	equipped with the semi-metric 
	$\rho'(f_{\bm s^{(1)}}, f_{\bm s^{(2)}}) = \|\bm s^{(1)} - \bm s^{(2)} \|$ where
	$$ f_{\bm s}: (0,\infty)^d \to [0,\infty), \quad
	 f_{\bm s}(\bm x) = \log(\|(\bm s \circ \bm x)^{1/\bm \alpha}\|) 
	 \bm 1 \{ \|\bm s \circ \bm x\|> 1\}. $$
	For $l=1,\ldots,n$, $n \in \NN$, we set
	$$ Z_{nl}(f) = \sqrt{\frac{a^*(u_n)}{n}} f(u_n^{-1} \bm X_l^*)$$
	such that 
	$$ \{H_n(\bm s)\}_{\bm s \in \bigcup_{i=1}^d A_{\delta,\{i\}}} = \left\{ \sum\nolimits_{l=1}^n (Z_{nl}(f) - \EE Z_{nl}(f)) \right\}_{f \in \mathcal{F}'}.$$
	By construction, each $H_n$ is a centered process and it can be shown that its covariance function
	converges pointwise.
	
	Now, using the Potter bound, for every $\varepsilon > 0$, there exists some $C_\varepsilon > 0$ such
	that 
	$$ \frac{\PP(\|\bm X_I^*\| > u_n t)}{\PP(\|\bm X_I^*\| > u_n)} \leq C_\varepsilon t^{-(1-\varepsilon)}, \qquad t \geq 1, $$
	for sufficiently large $n$. Thus, for every $\eta > 0$ and $k \in \{1,2\}$,	we obtain
	\begin{align*}
	& \sum_{l=1}^n \EE[\|Z_{n,l}\|^k_{\mathcal{F}'} \bm 1 \{\|Z_{n,l}\|_{\mathcal{F}'} > \eta\}] \\
	\leq{}& n\left(\frac{a^*(u_n)}{n}\right)^{k/2} \left\| \frac{1}{\bm \alpha} \right\|^k \EE\left[ \log^{k}((1+\delta) \|\bm X_I^*\| / u_n) 
	\mathds 1\{ \log((1+\delta) \|\bm X_I^*\| / u_n) > \sqrt{n/a^*(u_n)} \eta\}\right] \\
	={}& n \left(\frac{a^*(u_n)}{n}\right)^{k/2} \left\| \frac{1}{\bm \alpha} \right\|^k
	\int_0^\infty \PP\left( \log^{k}((1+\delta) \|\bm X_I^*\| / u_n) 
	\mathds 1\{ \log((1+\delta) \|\bm X_I^*\| / u_n) > \sqrt{n/a^*(u_n)} \eta\} > x\right) \, \mathrm{d}x \displaybreak[0]\\
	={}& n \left(\frac{a^*(u_n)}{n}\right)^{k/2} \left\| \frac{1}{\bm \alpha} \right\|^k
	\bigg[ \PP\left(\|\bm X_I^*\| > \frac{u_n}{1+\delta} \exp\left(\sqrt{\frac{n}{a^*(u_n)}} \eta\right)\right) \cdot \left(\sqrt{\frac{n}{a^*(u_n)}} \eta\right)^{k} \\
	& \hspace{4cm} + \int_{(\sqrt{n/a^*(u_n)}\eta)^{k}}^{\infty}
	\PP\left(\|\bm X_I^*\| > \frac{u_n}{1+\delta} \exp(x^{1/k})\right) \, \mathrm{d} x \bigg] \\
	={}& n \left(\frac{n}{a^*(u_n)}\right)^{k/2} \left\| \frac{1}{\bm \alpha} \right\|^k \eta^k \cdot
	\bigg[ \PP\left(\|\bm X_I^*\| > \frac{u_n}{1+\delta} \exp\left(\sqrt{\frac{n}{a^*(u_n)}} \eta\right)\right) \\
	& \hspace{4.5cm} + \int_{1}^{\infty}
	\PP\left(\|\bm X_I^*\| > \frac{u_n}{1+\delta} \exp\left(\sqrt{\frac{n}{a^*(u_n)}} \eta u^{1/k}\right)\right) \, \mathrm{d} u \bigg]\\
	\leq{}& C_\varepsilon n \PP(\|\bm X_I^*\| > u_n) \left(\frac{n}{a^*(u_n)}\right)^{k/2}
	\left\| \frac{1}{\bm \alpha} \right\|^k \frac{\eta^k}{(1+\delta)^{1-\varepsilon}} \\
	& \quad \cdot \left[ \exp\left(-(1-\varepsilon) \eta  \sqrt{\frac{n}{a^*(u_n)}}\right) + \int_1^\infty
	\exp\left(-(1-\varepsilon) \eta \sqrt{\frac{n}{a^*(u_n)}} u^{1/k} \right) \, \mathrm{d} u \right] \\
	\sim{}& C_\varepsilon \tau_I \left(\frac{n}{a^*(u_n)}\right)^{1+k/2} \left\| \frac{1}{\bm \alpha} \right\|^k \frac{\eta^k}{(1+\delta)^{1-\varepsilon}} \cdot  \bigg[ \exp\left(-(1-\varepsilon) \eta  \sqrt{\frac{n}{a^*(u_n)}}\right) \\
	 & \hspace{5.75cm}  + \int_1^\infty
	\exp\left(-(1-\varepsilon) \eta  \sqrt{\frac{n}{a^*(u_n)}} u^{1/k} \right) \, \mathrm{d} u \bigg] \to 0
	\end{align*}
	because of $n/a^*(u_n) \to 0$ as $n \to \infty$. For $k=2$, this results in a Lindeberg condition that guarantees convergence of $H_n$ to a centered Gaussian process $H$ in terms of finite-dimensional distributions. 
	
	Analogously to the proof of Thm.~\ref{thm:fclt}, we verify the conditions from Theorem 2.11.9 in \citet{vdv-wellner-1996} in order to ensure functional convergence. Thus, we consider two functions
	$f_{\bm s^{(1)}}, f_{,\bm s^{(2)}} \in \mathcal{F}'$ such that $\|\bm s^{(1)} - \bm s^{(2)}\| < \eta$
	for some $0 < \eta < \delta$ and $\eta<1$ so that $\bm s^{(1)} \in A_{\delta,\{i\}}$ and $\bm s^{(2)} \in A_{\delta,\{i\}}$ for the same index $i \in \{1,\ldots,d\}$, that is, there exists some $\bm y \in A_{\delta,\{i\}\}}$ such that $s_i^{(1)}, s_i^{(2)} \in [y_i, y_i + \eta]$. Again, denoting $\underline{\bm \alpha}=\min_{1\le i\le d}\alpha_i$, we distinguish three cases:
	\begin{enumerate}
		\item $\| \bm s^{(1)} \circ \bm x\| \leq  1$ and 
		$\| \bm s^{(2)} \circ \bm x\| \leq  1$:\\
		In this case, we have
		$$|  f_{\bm s^{(1)}}(\bm x) - f_{\bm s^{(2)}}(\bm x) | = |0 - 0| = 0.$$
		\item  $\| \bm s^{(1)} \circ \bm x\| \geq  1$ and 
		$\| \bm s^{(2)} \circ \bm x\| \geq  1$:\\
		In this case, both the indicator functions in 
		$f_{\bm s^{(1)}}(\bm x)$ and in $f_{\bm s^{(2)}}(\bm x) $ are equal to $1$. 
		Choosing $0 < \beta < 1/2$ and using that $|\log(x) - \log(y)| < \beta^{-1} |x-y|^\beta$,
		we obtain that
		$$ |  f_{\bm s^{(1)}}(\bm x) - f_{\bm s^{(2)}}(\bm x) | \leq \frac 1{\alpha_i \beta} \eta^\beta x_i^\beta
		\leq \left\|\frac 1 {\bm \alpha}\right\| \cdot \frac 1 \beta \cdot 
		\eta^\beta \cdot \|\bm x\|^\beta.$$
		This case occurs only if $(1+\delta) \|\bm x\| > 1$ as $ \bm s^{(1)}$ and $\bm  s^{(2)}$ belong to $ \bigcup_{i=1}^d A_{\delta,\{i\}} \subset A_\delta$.
		\item $\min\{\|\bm s^{(1)} \circ \bm x\|, \|\bm s^{(2)} \circ \bm x\|\} \leq 1 < \max\{\|\bm s^{(1)} \circ \bm x\|, \|\bm s^{(2)} \circ \bm x\|\}$:\\	
		In this case one of the functions $f_{\bm s^{(1)}}(\bm x)$ and $f_{\bm s^{(2)}}(\bm x)$ equals zero, while the other one is bounded by 1.
		Here, it is important to use that $\|\bm s \circ \bm X\| \leq u_n$ for some $\bm s \in A_{\delta,\{i\}}$
		implies that $\|u_n^{-1} X_i\| \leq 1+\delta$. Thus, we will have $ \max_{j=1,2} f_{\bm s^{(j)}}(\bm x) \leq \alpha_i^{-1} \log((1+\delta)^2)$ 
		in this case, and, thus,
		$$ |f_{\bm s^{(1)}}(\bm x) - f_{\bm s^{(2)}}(\bm x)	|\leq \frac 1 {\alpha_i} \log(1+\delta)  \mathbf 1 \Big\{ 
		 u_n / (s_i^{(1)}\vee s_i^{(2)})  \leq x_{i} <  u_n
		/ (s_i^{(1)}\wedge s_i^{(2)}) \Big\}.$$       
		This case occurs only if 
		$y_i x_i \leq 1 < (y_i+\eta) x_i$.
	\end{enumerate}  
	
	Thus, for each $i \in \{1,\ldots,d\}$ we obtain a uniform bound on $s_i^{(1)}, s_i^{(2)} \in [y_i, y_i + \eta]$, $i \in I$, for some $\bm y \in A_{\delta,\{i\}}$
	\begin{align*}
	(Z_{nl}(f_{\bm s^{(1)}}) - Z_{nl}(f_{\bm s^{(2)}}))^2 &\leq \frac{a^*(u_n)}{n} \left\| \frac 1 {\bm \alpha \beta} \right\|^2 \cdot \|u_n^{-1} \bm X^*\|^{2\beta} \eta^{2\beta} \bm 1 \{\|\bm X_{l}^*\| > (1+\delta)^{-1} u_n\} \\
	& \quad + \frac{a^*(u_n)}{n} \left\| \frac 1 {\bm \alpha} \right\|^2 \log^2(1+\delta) \cdot \bm 1 \left\{u_n(y_i \eta)^{-1} \leq X^*_{li} \leq u_n y_i^{-1}\right\}\\
	& =: \widetilde{Q}^{(1)}_{nl}(\eta) +  \left\| \frac 1 {\bm \alpha} \right\|^2  \log^2(1+\delta) \cdot Q_{nl}^{(2)}(\eta, \bm y)
	\end{align*}
	where $Q_{nl}^{(2)}(\eta, \bm y)$ is as in \eqref{eq:def-q}.
	
	The definition of $\widetilde{Q}^{(1)}_{nl}(\eta)$ and regular variation of the function $a^*$ imply that
	\begin{align*}
	\sum\nolimits_{l=1}^n \EE\left[ \widetilde{Q}^{(1)}_{nl}(\eta)\right] 
	& \leq{} a^*(u_n)  \left\| \frac 1 {\bm \alpha \beta} \right\|^2  \eta^{2\beta} \EE\left( \|u_n^{-1} \bm X^*\|^{2\beta} \bm 1 \{ \|\bm X^*\| > (1+\delta)^{-1} u_n \}\right]\\
	& \sim{} \tau \frac{a^*(u_n)}{a^*((1+\delta)^{-1} u_n)}  \left\| \frac 1 {\bm \alpha \beta} \right\|^2  \eta^{2\beta}
	\EE\left( \|u_n^{-1} \bm X^*\|^{2\beta} \mid \|\bm X^*\| > (1+\delta)^{-1} u_n \right] \\
	& \sim{}  \left\| \frac 1 {\bm \alpha} \right\|^2  \frac{\tau (1+\delta)^{1-2\beta}}{\beta^2} \EE(\|Y \bm \Theta^*\|^{2\beta}) \cdot \eta^{2\beta} 
	\leq  \left\| \frac 1 {\bm \alpha} \right\|^2  \frac{\tau (1+\delta)^{1-2\beta}}{(1-2\beta) \beta^2}\cdot \eta^{2\beta}, \qquad n \to \infty,   
	\end{align*}
	since $\EE(Y^{2\beta}) = 1 / (1 - 2 \beta) < \infty$ for all $\beta < 1/2$. Thus, analogously to  \eqref{eq:q1}, for $n$ 
	sufficiently large, we have that
	$$ \sum_{l=1}^n \EE\left[ \widetilde{Q}^{(1)}_{nl}(\eta)\right] \leq 2 \left\| \frac 1 {\bm \alpha} \right\|^2  \frac{d (1+\delta)^{1-2\beta}}{(1-2\beta) \beta^2}\cdot \eta^{2\beta}$$
	Combining this bound with the bound with the bounds for $\sum_{l=1}^n \EE\left[ Q^{(2)}_{nl}(\eta)\right]$ provided in the proof of Thm.~\ref{thm:fclt}, the equi-continuity condition and the entropy condition of Theorem 2.11.9 in \citet{vdv-wellner-1996} can be verified,
	following the lines of the proof of Thm.~\ref{thm:fclt}. Thus, we obtain weak convergence of
	of $\{H_n(\bm s): \, \bm s \in \bigcup_{i=1}^d A_{\delta,\{i\}} \}$ to $H$ in $\ell^\infty(\bigcup_{i=1}^d A_{\delta,\{i\}})$.
	
	Making use of \eqref{eq:bias-hill}, the bias can be neglected asymptotically, i.e.
	\begin{align*}
	\left\{\sqrt{\frac{n}{a^*(u_n)}} \left( a^*(u_n) \widehat L_{n,u_n}(\bm s) - \frac{s_i}{\alpha_i} \right); \ \bm s \in A_{\delta,\{i\}}, \, i=1,\ldots,d \right\} \to \left\{H(\bm s); \ \bm s \in \bigcup\nolimits_{i=1}^d A_{\delta,\{i\}}\right\}
	\end{align*}
	weakly in $\ell^\infty(A_\delta)$. By Thm.~\ref{thm:fclt}
	and \eqref{eq:bias-negligible} for $p=0$, we also obtain
	\begin{align*}
	\left\{\sqrt{\frac{n}{a^*(u_n)}} \left( a^*(u_n) \widehat P_{n,u_n}(\bm s) - s_i \right); \ \bm s \in A_{\delta,\{i\}}, \, i=1,\ldots,d \right\} \to \left\{G(\bm v, \bm s, \beta, 0); \ \bm s \in \bigcup\nolimits_{i=1}^d A_{\delta,\{i\}} \right\}
	\end{align*}
	weakly  $\ell^\infty(\bigcup_{i=1}^d A_{\delta,\{i\}})$ for any fixed $\bm v \in \partial B_1^+(\bm 0)$ and $\beta \in [(1+\delta)^{-1}, 1+ \delta]$ as the RHS does not depend on these variables.
	Note that both convergences hold jointly. Thus, analogously to the proof of Cor.~\ref{coro:fclt-ratio}, we can apply the
	functional Delta method to obtain the desired convergence
	$$ \left\{ \sqrt{\frac{n}{a(u_n)}} \left(\frac{\widehat L_{n,u}(\bm s)}{\widehat P_{n,u_n}(\bm s)} -\frac{s_i}{\alpha_i}\right); \ \bm s \in \times A_{\delta,\{i\}}, \, i=1,\ldots,d \right\} \to_w \left\{\widetilde H(\bm s); \ \bm s \in \bigcup\nolimits_{i=1}^d A_{\delta,\{i\}}\right\}$$ 
	for some centered Gaussian process $\widetilde H$. The covariance structure of $\widetilde H$ follows from classical computation.
\end{proof}	
	
Using this result, we are ready to prove Thm.~\ref{thm:an-hill}.	

 \begin{proof}[Proof of Thm.~\ref{thm:an-hill}]
	First, we note that 
	\begin{align*}
	\frac{1}{\widehat \alpha_{n,k_n,\{i\}}}
	={}& \frac{n^{-1} \sum_{l=1}^n \log(\|\bm 1_{\{i\}} \circ \bm X_l/\bm X_{k_n:n}\|) \bm 1 \{ \|\bm 1_{\{i\}} \circ \bm X_l / \bm X_{k_n:n} \|>1\}}{\widetilde P_{n,k_n.\{i\}}} \\
	={}& \frac{n^{-1} \sum_{l=1}^n \log(\|\bm 1_{\{i\}} \circ \bm X^*_l/\bm X^*_{k_n:n}\|^{1/\alpha_i}) \bm 1 \{ \|\bm 1_{\{i\}} \circ \bm X^*_l / \bm X^*_{k_n:n} \|>1\}}{\widetilde P_{n,k_n.\{i\}}} \\
	={}& \frac{\widehat L_{n,u_n}(u_n/\bm X_{k_n:n}^* \circ \bm 1_{\{i\}})}
	{\widehat P_{n,u_n}(u_n/\bm X_{k_n:n}^* \circ \bm 1_{\{i\}})},
	\end{align*}
	cf.~Eq.~\eqref{eq:rel-deterministic-random}. Thus, we have to show that
	$$
	(\widetilde H_n(u_n/\bm X_{k_n:n}^* \circ \bm 1_{\{i\}}) )_{1\le i\le d}\to (\widetilde H(\bm 1_{\{i\}}))_{1\le i\le d} =: \widetilde H.
	$$
	To this end, we use the consistency of the tail empirical measure to obtain
	\begin{equation} \label{eq:consistency-orderstats}
	\frac{u_n}{\bm X^*_{k_n:n}} \sim \left(\frac{(F_i^*)^{-1}(1-k_n/n)}{X^*_{k_n:n,i}}\right)_{i=1,\ldots,d} \to_p \mathbf{1}\,,\qquad n\to \infty\,,
	\end{equation}
	cf.\ Equation (4.17) in \citet{resnick-2007}. By Skohorod's representation theorem, without loss of
	generality, we may reformulate the statement in Prop.~\ref{thm:fclt-hill} as
	\begin{equation} \label{eq:as-conv-hn}
	\sup_{\bm s \in \bigcup_{i=1}^d A_{\delta,\{i\}}}
	\left| \widetilde H_n(\bm s) - \widetilde H(\bm s) \right| \to 0
	\qquad \text{a.s.}\,,\qquad n\to\infty.    
	\end{equation}
	Now, we estimate 
	\begin{align} \label{eq:bound-two-summands-h}
	& \left\|(\widetilde H_n(u_n/\bm X^*_{k_n:n} \circ \bm 1_{\{i\}}))_{1\le i\le d}
	-    \widetilde H\right\| \nonumber \\
	\leq{}& \phantom{+}
	\max_{1\le i\le d}\left|
	    \widetilde H_n(u_n/\bm X^*_{k_n:n} \circ \bm 1_{\{i\}})
	-   \widetilde H(u_n/\bm X^*_{k_n:n} \circ \bm 1_{\{i\}}) \right|
	+ \
	\left\|(\widetilde H(u_n/\bm X^*_{k_n:n} \circ \bm 1_{\{i\}}))_{1\le i\le d} - \widetilde H \right\|.
	\end{align}
	Eq.~\eqref{eq:consistency-orderstats} guarantees that, with probability tending to $1$,
	$\{ u_n/\bm X^*_{k_n:n} \circ \bm 1_{\{i\}} \in A_{\delta,\{i\}} \}$, and, in that case, we can bound the
	right-hand side of Eq.~\eqref{eq:bound-two-summands-h} by
	$$ \sup_{\bm s \in \bigcup_{i=1}^d A_{\delta,\{i\}}} 
	\left|\widetilde H_n(\bm s)
	-     \widetilde H(u_n/\bm X^*_{k_n:n} \circ \bm s) \right|
	+ \
	\left\|(\widetilde H(u_n/\bm X^*_{k_n:n} \circ \bm 1_{\{i\}}))_{1\le i\le d} - \widetilde H \right\|. $$
	Here, the first term vanishes asymptotically by Eq.~\eqref{eq:as-conv-hn} and the
	second term tends to zero because of Eq.~\eqref{eq:consistency-orderstats} and the
	uniform continuity of the sample paths of the tight process $\widetilde H$ (cf.~Thm.~\ref{thm:fclt-hill}). Thus, we obtain the desired convergence
	\begin{equation*}
	\left\|(\widetilde H_n(u_n/\bm X^*_{k_n:n} \circ \bm 1_{\{i\}}))_{1\le i\le d}
	-     \widetilde H \right\| \to_p 0.
	\end{equation*}
\end{proof}

\section{Proof of Thm.~\ref{thm:fclt-ratio-rank}}\label{app:th10}

Before showing Thm.~\ref{thm:fclt-ratio-rank}, we prove two auxiliary results.
We start by establishig a first (but biased) version of the functional central limit theorem for the sequence $\{\widetilde{M}_{n,k_n,I}(\bm v,p)/\widetilde{P}_{n,k_n,I}; \ \bm v \in \partial B_1^+(\bm 0) , \, p \in K \}$ of processes.

\begin{proposition} \label{prop:fclt-orderstats}
    Let $\bm X_l$,
	$l \in \NN$, be independent copies of a non-standard regularly varying
	$[0,\infty)^d$-valued random vector $\bm X$ satisfying the assumptions of Cor.~\ref{coro:fclt-ratio} and Thm.~\ref{thm:an-hill} for some $\delta > 0$ and some $K \subset \NN$. Then, for all non-empty $I \subset \{1,\ldots,d\}$, we have that
	\begin{align*}
	\left\{ \sqrt{k_n} \left( \frac{\widetilde{M}_{n,k_n,I}(\bm v,p)}{\widetilde{P}_{n,k_n,I}}
	- c\left(\bm v, \frac{u_n}{\bm X^*_{k_n:n}} \circ \bm 1_I, \frac{\bm \alpha}{\widehat {\bm \alpha}_{n,k_n}}, p\right)\right); \ \bm v \in \partial B_1^+(\bm 0) , \, p \in K \right\} \\
	\qquad \qquad
	\to \left\{\widetilde G(\bm v, \bm 1_I,\bm 1, p); \ \bm v \in \partial B_1^+(\bm 0), \, p \in K\right\}\,,\qquad n\to\infty\,,
	\end{align*}
	weakly in $\ell^\infty(\partial B_1^+(\bm 0) \times K)$	where ${\bm X^*_{k_n:n}}$ denotes the $k_n$-th order statistics of the transformed vector sample $\bm X^*_l$, $1\le l\le n$, and $\widetilde G$ is defined as in
	Cor.~\ref{coro:fclt-ratio}.
\end{proposition}
\begin{proof}
	By Eq.~\eqref{eq:rel-deterministic-random}, we have that 
	$$ \sqrt{k_n} \left( \frac{\widetilde{M}_{n,k_n,I}(\bm v,p)}{\widetilde{P}_{n,k_n,I}}
	- c\left(\bm v, \frac{u_n}{\bm X^*_{k_n:n}} \circ \bm 1_I, \frac{\bm \alpha}{\widehat {\bm \alpha}_{n,k_n}}, p\right)\right) = \widetilde G_n\left(\bm v, \frac{u_n}{\bm X^*_{k_n:n}} \circ \bm 1_I, \frac{\bm \alpha}{\widehat {\bm \alpha}_{n,k_n}}, p\right), $$
	where $\widetilde G_N$ is defined as in Cor.~\ref{coro:fclt-ratio}.
	Furthermore, Eq.~\eqref{eq:consistency-orderstats} and Thm.~\ref{thm:an-hill} provide
	\begin{equation} \label{eq:consistency-orderstats-alpha}
	\left(\frac{u_n}{\bm X^*_{k_n:n}} \circ \bm 1_I, \frac{\bm \alpha}{\widehat {\bm \alpha}_{n,k_n}}\right)
	\to_p (\bm 1_I, \bm 1), \qquad n \to \infty\,.
	\end{equation}
	Again, Skohorod's representation theorem allows us to rewrite the result of Cor.~\ref{coro:fclt-ratio} as
	\begin{equation} \label{eq:as-conv-gn}
	\sup_{\bm v \in \partial B_1^{+}(\bm 0), (\bm s,\bm \beta) \in A_{\delta}', p \in K}
	\left| \widetilde G_n(\bm v, \bm s,\bm \beta, p)
	- \widetilde G(\bm v, \bm s,\bm \beta, p) \right| \to 0
	\qquad \text{a.s.}\,,\qquad n\to\infty.    
	\end{equation}
	Using Eq.~\eqref{eq:consistency-orderstats-alpha}, with probability tending to $1$, we can bound the right-hand side of
	\begin{align*}
	& \sup_{\bm v \in \partial B_1^{+}(\bm 0), p \in K}
	\left|\widetilde G_n(\bm v, u_n/\bm X^*_{k_n:n} \circ \bm 1_I,\bm \alpha/\widehat {\bm \alpha}_{n,k_n}, p)
	-     \widetilde G(\bm v, \bm 1_I, \bm 1, p) \right| \nonumber \\
	\leq{}& \phantom{+} \sup_{\bm v \in \partial B_1^{+}(\bm 0), p \in K}
	\left|\widetilde G_n(\bm v, u_n/\bm X^*_{k_n:n} \circ \bm 1_I, \bm \alpha/\widehat {\bm \alpha}_{n,k_n}, p)
	-     \widetilde G(\bm v, u_n/\bm X^*_{k_n:n} \circ \bm 1_I, \bm \alpha/\widehat {\bm \alpha}_{n,k_n}, p) \right|
	\nonumber \\
	& + \sup_{\bm v \in \partial B_1^{+}(\bm 0), p \in K}
	\left|\widetilde G(\bm v, u_n/\bm X^*_{k_n:n} \circ \bm 1_I, \bm \alpha/\widehat {\bm \alpha}_{n,k_n} p) 
	-     \widetilde G(\bm v, \bm 1_I, \bm 1, p) \right|.
	\end{align*}
	by
	\begin{align*}
	& \sup_{\bm v \in \partial B_1^{+}(\bm 0), (\bm s,\bm \beta) \in A_{\delta}', p \in K}
	\left|\widetilde G_n(\bm v, \bm s,\bm \beta, p) - \widetilde G(\bm v, \bm s,\bm \beta, p) \right| \\
	& \qquad \qquad + \sup_{\bm v \in \partial B_1^{+}(\bm 0), p \in K}
	\left|\widetilde G(\bm v,  u_n/\bm X^*_{k_n:n} \circ \bm 1_I, \bm \alpha/\widehat {\bm \alpha}_{n,k_n}, p)
	- \widetilde G(\bm v, \bm 1_I, \bm 1, p) \right|.
	\end{align*}
	
	While the first term tends to zero by Eq.~\eqref{eq:as-conv-gn},
	the second term vanishes asymptotically because of Eq.~\eqref{eq:consistency-orderstats-alpha}
	and the uniform continuity of $\widetilde G(\cdot,\cdot,p)$, cf.\ the third part of Rem.~\ref{rem:joint-conv-cont}. Consequently,
	$$\sup_{\bm v \in \partial B_1^{+}(\bm 0), p \in K}
	\left|\widetilde G_n(\bm v, u_n/\bm X^*_{k_n:n} \circ \bm 1_I, \bm \alpha/\widehat {\bm \alpha}_{n,k_n}, p)
	- \widetilde G(\bm v, \bm 1_I, \bm 1, p) \right| \to_p 0\,, \qquad n\to \infty.$$ 
\end{proof}

 Note that the centering term in Prop.~\ref{prop:fclt-orderstats},
$c(\bm v_I, u_n /\bm X^*_{k_n:n} \circ \bm 1_I,\bm \alpha/\widehat {\bm \alpha}_{n,k_n}, p)$ is random and needs 
to be replaced by the desired limit  $c(\bm v_I,\bm 1_I, \bm 1, p)$.
It turns out that, in general, the difference is not negligible, i.e.\
the replacement of the centering term will influence the asymptotic
distribution. The proof will rely on the following lemma.

\begin{lemma} \label{lem:orderstats}
	Let the assumptions of Prop.~\ref{prop:fclt-orderstats} hold
	and assume that there exists some $\delta > 0$ such that
	Eq.~\eqref{eq:one-dim-bias} holds  for all $i \in \{1,\ldots,d\}$. Then we have
	$$ \left\{\sqrt{k_n} \left( \frac{u_n}{X_{\lfloor sk_n\rfloor:n,i}^*} - s\right); 
	\ (1+ \delta)^{-1} \leq s \leq 1 + \delta\right\}
	\to_w \left\{ -G^0(s \bm 1_{\{i\}}) \right\}
	$$
	in $\ell^\infty([(1+\delta)^{-1}, 1+\delta])$.  
\end{lemma}
\begin{proof}
	Applying Thm.~\ref{thm:fclt} to the case $p=0$ and $\bm s = s \bm 1_{\{i\}}$ and making use of Eq.~\eqref{eq:one-dim-bias}, we obtain that
	$$ \left\{ \sqrt{k_n} \left( \frac 1 {k_n} \sum\nolimits_{l=1}^n \mathbf{1}\{s X^*_{li} > u_n\} - s\right); \ (1+ \delta)^{-1} \leq s \leq 1 + \delta \right\} \to_w
	\left\{G^0(s \bm 1_{\{i\}}) \right\} $$
	in $\ell^\infty([(1+\delta)^{-1}, 1+\delta])$   as $n\to\infty$. Using the notation $\widehat F_{n,i}$ for the empirical cumulative distribution function of the sample $X^*_{1i},\ldots,X^*_{ni}$ and Skohorod's representation theorem, we may assume that
	$$ \sup_{(1+\delta)^{-1} \leq s \leq 1 + \delta} 
	\left| \sqrt{k_n} \left( \frac{n}{k_n} (1-\widehat F_{n,i}(u_n/s)) - s\right) - G^0(s \bm 1_{\{i\}})\right| \to 0 \qquad \text{a.s.}\,,\qquad n\to \infty.$$
	Applying Vervaat's Lemma \citep[cf.\ Lemma A.0.2 in][for instance]{dehaan-ferreira-2006} to the nondecreasing functions 
	$x_n(s) = n/k_n (1-\widehat F_{n,i}(u_n/s))$, we obtain   
	$$ \sup_{(1+\delta)^{-1} \leq s \leq 1 + \delta} 
	\left| \sqrt{k_n} \left( u_n / X_{\lfloor sk_n\rfloor:n,i}^* - s\right) + G^0(s \bm 1_{\{i\}})\right| \to 0 \qquad \text{a.s.},$$
	where we used that $x_n^\inv(s) = u_n /  X_{\lfloor sk_n\rfloor:n,i}^*$.
\end{proof}	

Imposing some additional conditions on the function $c$ and using our previous results, we can follow the same lines as the proof of Theorem
4.6 in \citet{einmahl-krajina-segers-2012} in order to establish the desired convergence result.

\begin{proof}[Proof of Thm.~\ref{thm:fclt-ratio-rank}]
	By Skohorod's representation theorem, we may assume that the weak convergences
	in Thm.~\ref{thm:an-hill}, Prop.~\ref{prop:fclt-orderstats} and in Lemma \ref{lem:orderstats} take place a.s. Then, the triangular inequality yields the bound
	\begin{align*}
	& \sup_{\bm v \in B_1^+(\bm 0), p \in K} \bigg|
	\sqrt{k} \left( \frac{\widetilde{M}_{n,k_n}(\bm v_I,p)}{\widetilde{P}_{n,k_n}} 
	- c(\bm v_I, \bm 1_I, p)\right) 
	- \widetilde G(\bm v_I,\bm 1_I,\bm 1 ,p) \\
	& \qquad \qquad \qquad
	+ \sum_{i \in I}\left( c_{s_i}(\bm v_I,\bm 1_I,\bm 1,p) \cdot G^0(\bm 1_{\{i\}}) - c_{\beta_i}(\bm v_I,\bm 1_I,\bm 1,p)
	\cdot \bm \alpha_i \cdot \widetilde H_i\right) \bigg| \\
	&{}\leq\sup_{\bm v \in \partial B_1^+(\bm 0), p \in K}
	\left| \widetilde G_n\left(\bm v_I, \frac{u_n}{\bm X^*_{k_n:n}} \circ \bm 1_I, \frac{\bm \alpha}{\widehat {\bm \alpha}_{n,k_n}}, p\right) 
	- \widetilde G(\bm v_I, \bm 1_I, \bm 1, p) \right| \\
	& \quad  +   \sup_{\bm v \in \partial B_1^+(\bm 0), p \in K} \bigg| \sqrt{k} 
	\left(  c\left(\bm v_I, \frac{u_n}{\bm X^*_{k_n:n}} \circ \bm 1_I, \frac{\bm \alpha}{\widehat {\bm \alpha}_{n,k_n}}, p\right)
	- c(\bm v_I, \bm 1_I,\bm 1, p) \right) \\
	& \qquad \qquad \qquad \qquad \qquad
	+ \sum_{i \in I} \left(c_{s_i}(\bm v_I, \bm 1_I, \bm 1, p) \cdot G^0(\bm 1_{\{i\}}) 
	-  c_{\beta_i}(\bm v_I,\bm 1_I,\bm 1,p) \cdot \bm \alpha_i \cdot\widetilde H_i\right)\bigg| \\
	&{}=: D_1 + D_2.    
	\end{align*}
	We will show that $D_1 \to_p 0$ and $D_2 \to_p 0$.
	By Prop.~\ref{prop:fclt-orderstats}, we directly obtain that $D_1 \to 0$ a.s.
	In order to show that $D_2 \to_p 0$, we first
	define
	$$ C_n(\bm v_I, p) 
	:= \sqrt{k} 
	\left(c\left(\bm v_I, \frac{u_n}{\bm X^*_{k:n}} \circ \bm 1_I, 
	\frac{\bm \alpha}{\widehat {\bm \alpha}_{n,k_n}}, p\right) - c(\bm v_I, \bm 1_I,\bm 1, p) \right),    \quad \bm v_I \in \partial B_1^+(\bm 0), \ p \in K. $$
	The mean value theorem yields
	\begin{equation} \label{eq:mvt}
	\frac 1 {\sqrt{k}} C_n(\bm v_I, p) = \sum\nolimits_{i \in I} \Big(\Big( \frac{u_n}{\bm X^*_{k:n,i}} - 1 \Big) \cdot c_{s_i}(\bm v_I, \bm \xi_n, \bm \psi_n, p)
	+ \Big( \frac{\bm \alpha_i}{\widehat {\bm \alpha}_{n,k_n,i}}- 1 \Big) \cdot c_{\beta_i}(\bm v_I, \bm \xi_n, \bm \psi_n, p)\Big)
	\end{equation}
	for some $\bm \xi_n = \bm \xi_n(\bm v_I,p)$ between $u_n / \bm X^*_{k:n} \circ \bm 1_I$ and $\bm  1_I$
	and $\bm \psi_n = \bm \psi_n(\bm v_I, p)$ between $\bm \alpha/\widehat {\bm \alpha}_{n,k_n}$ and $\bm 1$. 
	If $(u_n / \bm X^*_{k:n} \circ \bm 1_I, \bm \alpha/\widehat {\bm \alpha}_{n,k_n}) \in A_{\delta,I}'$, which happens with probability tending to 1 by Eq.~\eqref{eq:consistency-orderstats-alpha}, we obtain from \eqref{eq:mvt} that
	\begin{align*}
	D_2 ={}&  \sup_{\bm v \in \partial B_1^+(\bm 0), p \in K} \left| C_n(\bm v_I,p) 
	+ \sum\nolimits_{i \in I} c_{s_i}(\bm v_I, \bm 1_I,\bm  1, p) G^0(\bm 1_{\{i\}})  -
	c_{\beta_i}(\bm v_I, \bm 1_I, \bm 1, p)\bm \alpha_i\widetilde H_i \right| \\
	\leq{}& \sum_{i \in I} \sup_{\bm v \in \partial B_1^+(\bm 0), p \in K}   
	\left| c_{s_i}(\bm v_I, \bm \xi_n,\bm \psi_n, p) \sqrt{k} \left( \frac{u_n}{\bm X^*_{k_n:n,i}} - 1 \right) + c_{s_i}(\bm v_I, \bm 1_I, \bm 1, p) G^0(\bm 1_{\{i\}}) \right|  \\
	& \ + \,\sum_{i \in I} \sup_{\bm v \in \partial B_1^+(\bm 0), p \in K} \left| c_{\beta_i}(\bm v_I, \bm \xi_n,\bm  \psi_n, p) \sqrt{k} \left( \frac{\bm\alpha_i}{\widehat{\bm \alpha}_{n,k_n,i}} - 1 \right) - c_{\beta_i}(\bm v_I, \bm 1_I, \bm 1, p) \bm \alpha_i\widetilde H_i \right| \\
	\leq{}& \sum_{i \in I} \Bigg( \sup_{\bm v \in B_1^+(\bm 0), (\bm s,\bm \beta) \in A_{\delta,I}', p \in K}
	|c_{s_i}(\bm v_I, \bm s,\bm \beta, p)| \cdot \left| \sqrt{k} \left( \frac{u_n}{\bm X^*_{k:n,i}} - 1 \right) + G^0(\bm 1_{\{i\}}) \right| \\
	& \qquad \quad + \sup_{\bm v \in B_1^+(\bm 0), p \in K} 
	\left| c_{s_i}\left(\bm v_I, \bm \xi_n, p\right)
	-  c_{s_i}(\bm v_I, \bm 1_I, p) \right| \cdot |G^0(\bm 1_{\{i\}})|\\
	& + \sup_{\bm v \in B_1^+(\bm 0), (\bm s,\beta) \in A_{\delta,I}', p \in K}
	|c_{\beta_i}(\bm v_I, \bm s,\bm \beta, p)| \cdot \left| \sqrt{k} \left( \frac{\bm\alpha_i}{\widehat{\bm \alpha}_{n,k_n,i}} -  1 \right)  -\bm \alpha_i \widetilde H_{I} \right| \\
	& + \sup_{\bm v \in B_1^+(\bm 0), p \in K} 
	\left| c_{\beta_i}\left(\bm v_I, \bm \xi_n, \bm \psi_n, p\right)
	-  c_{\beta_i}(\bm v_I, \bm 1_I, \bm \beta, p) \right| \cdot |\bm \alpha_i \widetilde H_{i}|  \Bigg) \\
	=:{}& \sum\nolimits_{i \in I} (D_{3,i} \cdot D_{4,i} + D_{5,i} \cdot D_{6,i}
	+ D_{3,\beta,i} \cdot D_{4,\beta,i} +  D_{5,\beta,i} \cdot D_{6,\beta,i}).         
	\end{align*}	
	Since $c_{s_i}$ and $c_{\beta_i}$, $i \in I$,  are continuous on the compact set $\partial B_1^*(\bm 0) \times A_{\delta,I}' \times K$, we have $D_{3,i} < \infty$, $i \in I$, and $D_{3,\beta,i} < \infty$ while
	$D_{4,i} \to_p 0$, $i \in I$ and $D_{4,\beta,i} \to_p 0$ by Lemma \ref{lem:orderstats} and Thm.~\ref{thm:an-hill}, respectively. Furthermore, we have 
	$$ \sup_{\bm v \in B_1^+(\bm 0), p \in K} 
	\left\| \left( \begin{array}{c} \bm \xi_n\\ \bm \psi_n \end{array}\right) -
	\left( \begin{array}{c}	\bm 1_I  \\ \bm 1_I     \end{array}\right)\right\| 
	\leq {}
	\left\| \left( \begin{array}{c} \frac{u_n}{\bm X^*_{k_n:n}} \circ \bm 1_I\\ \frac{\bm\alpha}{\widehat{\bm \alpha}_{n,k_n}} \circ \bm 1_I\end{array}\right) - \left( \begin{array}{c}
	\bm 1_I \\ \bm 1_I \end{array} \right)\right\| \to_p 0$$
	by \eqref{eq:consistency-orderstats-alpha}. Consequently, using the uniform continuity of $c_{s_i}$ and $c_{\beta_i}$, $i \in I$, we obtain that $D_{5,i} \to_p 0$ and $D_{5,\beta,i} \to_p 0$.
	The facts that $D_{6,i} < \infty$ a.s., $i \in I$, and $D_{6,\beta,i} < \infty$ a.s.\ finally yield
	$$ D_2 =\sup_{\bm v \in \partial B_1^+(\bm 0), p \in K} \left| C_n(\bm v_I,p) 
	+ \sum_{i \in I} c_{s_i}(\bm v_I, \bm 1_I, \bm 1, p) G^0(\bm 1_{\{i\}})  -
	c_{\beta_i}(\bm v_I, \bm 1_I, \bm 1, p) \alpha_i\widetilde H_i \right|  \to_p 0.$$
\end{proof}

\section{Further Analysis of the Precipitation Data Considered in Subsection \ref{subsec:precip}}\label{subsec:precipE}
In this section, we further analyze the precipitation data in the fall season in different regions in France as considered in Subsection \ref{subsec:precip}. First, we estimate the pairwise upper tail dependence coefficients
$$ \chi_{ij} = \lim_{u \to \infty} \PP(X_i > u \mid X_j > u), \quad 1 \leq i,j \leq 9 $$
empirically via
$$\widehat \chi_{ij}(u_i,u_j) = \frac{\sum_{l=1}^n \mathds 1 \{X_{li} > u_i, X_{lj} > u_j\}}{\sum_{l=1}^n \mathds 1 \{X_{lj} > u_j\}}, \quad 1 \leq i,j \leq 9 $$
for large thresholds $u_1,\ldots,u_9$. Here, we choose the componentwise empirical $95\,\%$-quantiles as thresholds. The results are displayed in Figure \ref{fig:etdc}. Here, the stations within the same region exhibit moderate dependence with empirical tail dependence coefficients approximately between 0.4 and 0.6, while all pairs of stations in different regions have empirical tail dependence less than 0.2, i.e., there is only weak extremal dependence across different regions. However, the three regions cannot be assumed to be asymptotically independent. In particular, there is still some weak dependence between the two regions in the northwest and the northeast. 

\begin{figure}
	\centering \includegraphics[width=0.5\textwidth]{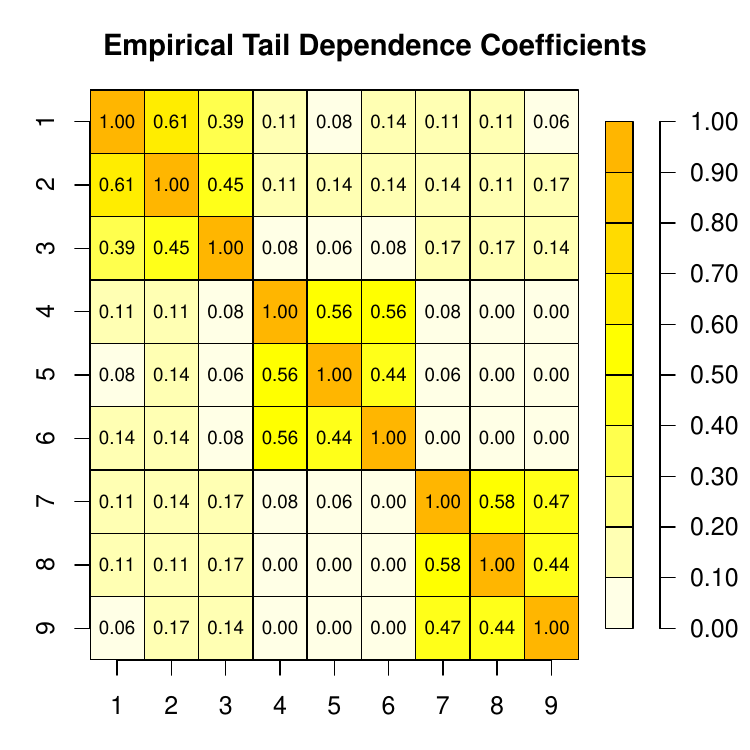}
	\caption{The empirical pairwise tail dependence coefficients $\widehat \chi_{ij}(u_i,u_j)$, $1 \leq i,j \leq j$ based on the empirical $95\,\%$-quantiles as thresholds. Here stations 1--3 belong to the region in the northwest, stations 4--6 to the region in the south and stations 7--9 to the region in the northeast of France.} \label{fig:etdc}
\end{figure}
Furthermore, as detailed in Subsection \ref{subsec:precip}, we apply our estimator for unknown margins $\widehat{\tau}_{I,n}^{MU}$ separately to each of the three regions. While the results in Table \ref{tab:application_regions} are obtained for fixed $k=40$, we also analyzed the effect of the choice of $k$. For different $k$ from $k=8$ to $k=180$ (corresponding marginal to $1\,\%$--$25\,\%$ of the observations), the results are shown in Figure \ref{fig:application-k}.
In all cases, the estimated extremal coefficients grow slightly, i.e.\ dependence weakens, as the threshold grows (which means that $k$ gets smaller), a phenomenon that is often observed in environmental data. This effect is even a bit more pronounced for the region in the south of France.
\begin{figure}
	\centering \includegraphics[width=0.9\textwidth]{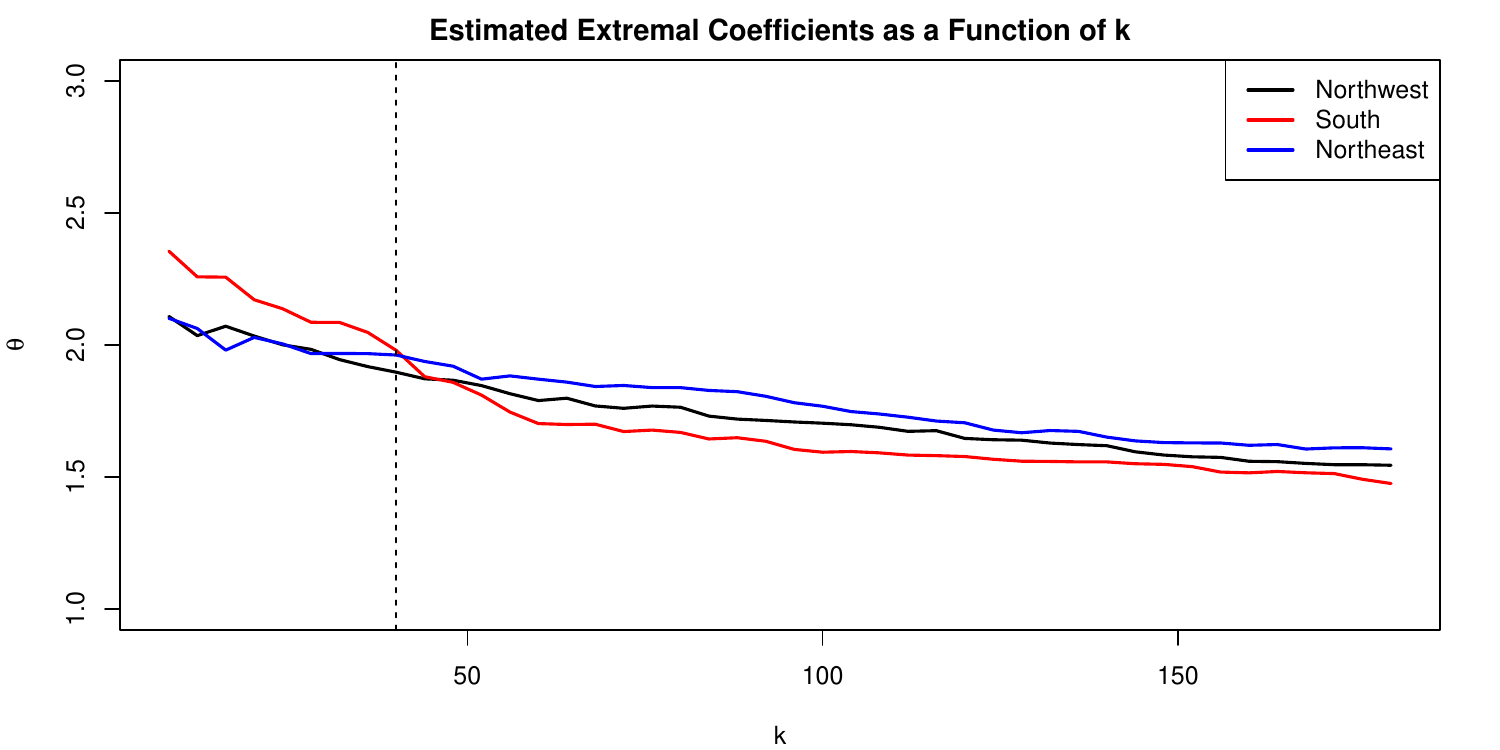}
	\caption{The estimated extremal coefficients $1/\widehat{\tau}_{I,n}^{MU}$ for the three regions in the northwest, the south and the northeast of France and different choices of $k$.} \label{fig:application-k}
\end{figure}

\end{document}